\documentclass[10pt]{article}
\usepackage[utf8]{inputenc} 
\usepackage[T1]{fontenc}
\usepackage{lmodern}

\usepackage{amsthm,amsmath,amsfonts,amssymb,stmaryrd,bbm}
\usepackage{mathtools}
\usepackage{enumitem}
\usepackage{mathrsfs} 

\usepackage{subcaption} 

\usepackage{graphicx}
\usepackage[dvipsnames]{xcolor}
\usepackage{caption,tikz}

\usepackage[english]{babel}
\usepackage[colorlinks=true, linkcolor=blue, urlcolor=black, citecolor=blue,pdfstartview=FitH]{hyperref}

\usepackage{bm} 
\usepackage{mathtools} 

\usepackage{subcaption} 

\numberwithin{equation}{section}

\let\originalleft\left
\let\originalright\right
\renewcommand{\left}{\mathopen{}\mathclose\bgroup\originalleft}
\renewcommand{\right}{\aftergroup\egroup\originalright}

\newlength{\bibitemsep}
\newlength{\bibparskip}
\let\oldthebibliography\thebibliography
\renewcommand\thebibliography[1]{\oldthebibliography{#1}
  \setlength{\parskip}{\bibitemsep}
  \setlength{\itemsep}{\bibparskip}}

\DeclareMathOperator{\Tr}{Tr}

\DeclareMathOperator{\OO}{O}
\DeclareMathOperator{\oo}{o}

\DeclareMathOperator{\Id}{Id}
\DeclareMathOperator{\diag}{diag}

\DeclareMathOperator{\Spec}{Spec}

\DeclareMathOperator{\arccot}{arccot}
\DeclareMathOperator{\Vol}{Vol}
\DeclareMathOperator{\rank}{rank}
\DeclareMathOperator{\erfc}{erfc}

\newcommand{\mc}[1]{\mathcal{#1}}
\newcommand{\mf}[1]{\mathfrak{#1}}
\newcommand{\ms}[1]{\mathscr{#1}}

\newcommand{\ii}{\mathrm{i}}

\newcommand{\defeq}{\vcentcolon=}
\newcommand{\eqdef}{=\vcentcolon}

\renewcommand{\epsilon}{\varepsilon}

\renewcommand{\leq}{\leqslant}
\renewcommand{\geq}{\geqslant}
\renewcommand{\le}{\leq}
\renewcommand{\ge}{\geq}

\renewcommand{\P}{\mathbb{P}}
\newcommand{\E}{\mathbb{E}}
\newcommand{\R}{\mathbb{R}}

\newcommand{\N}{\mathbb{N}}

\newcommand{\abs}[1]{\left\lvert #1 \right\rvert}

\newcommand{\vertiii}[1]{{\left\vert\kern-0.25ex\left\vert\kern-0.25ex\left\vert #1 
    \right\vert\kern-0.25ex\right\vert\kern-0.25ex\right\vert}}
\newcommand{\ip}[1]{\left\langle #1 \right\rangle}
\newcommand{\diff}{\mathop{}\!\mathrm{d}}

\theoremstyle{plain} 
\newtheorem{thm}{Theorem}[section]

\newtheorem{lem}[thm]{Lemma}

\newtheorem{prop}[thm]{Proposition}

\newtheorem{rem}[thm]{Remark}

\definecolor{color1}{RGB}{119,170,221} 
\definecolor{color2}{RGB}{153,221,255} 
\definecolor{color3}{RGB}{68, 187, 153} 
\definecolor{color4}{RGB}{187,204,51} 
\definecolor{color5}{RGB}{170,170,0} 
\definecolor{color6}{RGB}{238,221,136} 
\definecolor{color7}{RGB}{238,136,102} 
\definecolor{color8}{RGB}{255,170,187} 
\definecolor{color9}{RGB}{221,221,221} 
\definecolor{color10}{RGB}{136,34,85} 

\makeatletter
\renewcommand{\section}{\@startsection
{section}
{1}
{0mm}
{-2\baselineskip}
{1\baselineskip}
{\normalfont\large\scshape\centering}} 
\makeatother

\makeatletter
\renewcommand{\subsection}{\@startsection
{subsection}
{2}
{0mm}
{-\baselineskip}
{0 \baselineskip}
{\normalfont\bf\itshape}} 
\makeatother

\makeatletter
\renewcommand{\subsubsection}{\@startsection
{subsubsection}
{3}
{0mm}
{-\baselineskip}
{0 \baselineskip}
{\normalfont\bf\itshape}} 
\makeatother

\def\author#1{\par
    {\centering{\authorfont#1}\par\vspace*{0.05in}}
}

\def\titlefont{\fontsize{13}{15}\bfseries\boldmath\selectfont\centering{}}
\def\authorfont{\fontsize{13}{15}}

\let\affiliationfont\rhfont

\def\address#1{\par
    {\centering{\affiliationfont#1\par}}\par\vspace*{11pt}
}

\def\title#1{
    \thispagestyle{plain}
    \vspace*{-14pt}
    \vskip 79pt
    {\centering{\titlefont #1\par}}%
    \vskip 1em
}

\begin{document}

~\vspace{-1.4cm}

\title{Extremal Statistics of Quadratic Forms of GOE/GUE Eigenvectors}

\vspace{1cm}
\noindent

\begin{center}

\begin{minipage}[c]{0.5\textwidth}
 \author{L\'{a}szl\'{o} Erd\H{o}s}%
\address{Institute of Science and Technology Austria (ISTA)\\%
   E-mail: lerdos@ista.ac.at}%
 \end{minipage}%
\begin{minipage}[c]{0.5\textwidth}
 \author{Benjamin M\textsuperscript{c}Kenna}%
\address{Institute of Science and Technology Austria (ISTA)\\%
   E-mail: bmckenna@fas.harvard.edu}%
 \end{minipage}%
 
\end{center}

\begin{abstract}
We consider quadratic forms of deterministic matrices $A$ evaluated at the random eigenvectors of 
a large $N\times N$ GOE or GUE matrix, or equivalently evaluated at the columns of a Haar-orthogonal or Haar-unitary random matrix. 
We prove that, as long as the deterministic matrix has rank much smaller than $\sqrt{N}$, 
 the distributions of the extrema of these quadratic forms are asymptotically the same as if the eigenvectors were
  independent Gaussians. 
  This reduces the problem to Gaussian computations, which we carry out in several cases 
  to illustrate our result, finding Gumbel or Weibull limiting distributions depending on the signature of $A$. 
Our result also naturally applies to the eigenvectors of any invariant ensemble.
\end{abstract}

\vspace*{0.05in}

\noindent \emph{Date:} October 6, 2022

\vspace*{0.05in}

\noindent \hangindent=0.2in \emph{Keywords and phrases:} Gaussian Orthogonal Ensemble, Gaussian Unitary Ensemble, Haar measure, extreme value statistics, Gumbel distribution, Weibull distribution, Gram--Schmidt

\vspace*{0.05in}


{
	\hypersetup{linkcolor=black}
	\tableofcontents
}


\newpage

\section{Introduction}


\subsection{Overview.}\

This paper studies fluctuations of the random variables
\begin{equation}
\label{eqn:intro_variables}
	\max_{i=1}^N \ip{\mathbf{u}_i, A_N \mathbf{u}_i} \qquad \text{and} \qquad \max_{i=1}^N \abs{\ip{\mathbf{u}_i, A_N \mathbf{u}_i}}
\end{equation}
asymptotically in the large-$N$ limit. Here the $(\mathbf{u}_i = \mathbf{u}^{(N)}_i)_{i=1}^N$ are the $\ell^2$-normalized eigenvectors of an $N \times N$ random matrix drawn from the Gaussian Orthogonal Ensemble (GOE), and $(A_N)_{N=1}^\infty$ is a sequence of deterministic $N \times N$ real-symmetric matrices, which can be chosen freely subject to an important rank restriction, which we state informally as
\begin{equation}
\label{eqn:intro_rank}
	\rank(A_N) \ll N^{1/2}.
\end{equation}
We recall that the $N \times N$ GOE consists of real-symmetric random matrices $H_N$, whose entries are centered Gaussians, independent up to the symmetry constraint, normalized as $\E[(H_N)_{ij}^2] = \frac{1}{N}(1+\delta_{ij})$. Since GOE eigenvectors are distributed as the columns of a Haar-orthogonal random matrix, the variables \eqref{eqn:intro_variables} can equivalently be thought of as observables of Haar measure.

The joint distribution of a few Haar columns is almost that of independent Gaussian vectors. Our main result, very informally, says that this approximation holds for extremal statistics of quadratic forms involving $\ll \sqrt{N}$ columns. More precisely, our main
 Theorem \ref{thm:main} states that, under the rank restriction \eqref{eqn:intro_rank}, the variables \eqref{eqn:intro_variables} have the same fluctuations as their Gaussian counterparts, i.e., as the analogues of \eqref{eqn:intro_variables} with the vectors $\mathbf{u}_i$ replaced by appropriately normalized i.i.d. Gaussian vectors. This reduces the study of \eqref{eqn:intro_variables} to (much easier) Gaussian computations, which can be carried out for any choice of $A_N$, with an outcome depending on the signature structure of $A_N$ and on the multiplicity of its extreme eigenvalues. To keep the paper at a manageable length, however, we only present these computations for some representative choices of $A_N$. We study all the cases where $A_N$ has an $N$-independent rank, plus the case where $A_N = \diag(1,\ldots,1,0,\ldots,0)$ with $\rank(A_N) \approx N^\alpha$ for some $0 < \alpha < 1/2$. This leads to explicit limit laws for the variables \eqref{eqn:intro_variables} in these special cases (Theorems \ref{thm:main_fixed} and \ref{thm:main_diverging}, respectively). The limiting distributions in these cases are those of classical extreme value theory: mostly Gumbel, except for the special case of $\max_{i=1}^N \ip{\mathbf{u}_i, A_N \mathbf{u}_i}$ when $A_N$ has finite rank and all eigenvalues negative, in which case the limit is Weibull. Our precise theorems have many cases to cover every possible situation, but here we just give one representative result for easy orientation:

\begin{thm}
\label{thm:intro_example}
Fix $k \in \N$, and $k$ nonzero real numbers ordered as $a_1 \leq \cdots \leq a_k \eqdef a$. Notate the multiplicity of the largest one as $m \defeq \#\{i : a_i = a\}$. 
 If 
\[
	A_N = \diag(a_1, \ldots, a_k, 0, \ldots, 0),
\]
and $a > 0$, then 
\[
	\frac{N}{2a} \max_{i=1}^N \ip{\mathbf{u}_i, A_N \mathbf{u}_i} - \log N + \left(1-\frac{m}{2}\right) \log \log N + \log\left(\Gamma\left(\frac{m}{2}\right)\sqrt{\prod_{j=1}^{k-m}\left(1-\frac{a_j}{a}\right)}\right) \overset{N \to \infty}{\to} \Lambda \quad \text{in distribution,}
\]
where $\Lambda$ is a Gumbel-distributed random variable, with distribution function $\P(\Lambda \leq x) = \exp(-e^{-x})$.
\end{thm}

This result is formulated for GOE for definiteness,  but 
 it clearly also applies to any random matrix whose eigenvectors are Haar distributed.
  In particular it applies to invariant ensembles, i.e., real-symmetric $N \times N$ random matrices $X_N$ 
  such that, for any deterministic orthogonal $O$, the matrices $X_N$ and $OX_NO^{-1}$ have the same distribution. 
  This includes smooth invariant ensembles, i.e., random matrices $X_N$ whose law has a density with respect to
   Lebesgue measure that is proportional to $\exp(- N \Tr V(X_N))$ for some mildly regular function $V : \R \to \R$. 
   
   While stated for real symmetric ensembles, 
   all our results as well as their proofs also hold for the complex Hermitian case, meaning they apply to Haar unitary matrices, and consequently
    to Gaussian Unitary Ensemble (GUE) and to complex Hermitian invariant ensembles. 
    In fact, the very special rank-one case for GUE has already appeared in the physics literature
 \cite{LakTomBohMaj2008}  in the context of detecting quantum chaos. 
 See Remark~\ref{physics} for more details and possible further applications in physics.
    
Finally, one can also ask about universality beyond invariant ensembles, e.g.
 in the class of Wigner matrices, which are real symmetric random matrices whose i.i.d. upper-triangular entries are not necessarily Gaussian. 
 We conjecture that Theorem \ref{thm:intro_example} holds at least when the $\mathbf{u}_i$'s are the eigenvectors of 
 a large class of Wigner matrices.


\subsection{Previous results on extremal statistics in random matrices.}\

Classical extremal statistics is primarily concerned with three named probability distributions -- Gumbel, Weibull, and Fr\'echet, which we will call \emph{classical} -- because of the Fisher-Tippett-Gnedenko theorem, which states that these three are the only possible limiting laws for the normalized maximum of independent and identically distributed random variables. But random-matrix eigen\emph{values} tend to be strongly correlated, so their extremes often (although not always) have non-classical limiting distributions. The most prominent of these is the Tracy-Widom distribution, which was first introduced by Tracy and Widom to describe the fluctuation of the top eigenvalue  of GOE or GUE matrices \cite{TraWid1993, TraWid1994, TraWid1996}, but later extended to a large class of models, demonstrating the universality phenomenon \cite{Sos1999, Ruz2006, TaoVu2010, Sod2010, BouErdYau2014, LeeYin2014, LeeSch2015, AltErdKruSch2020}. As another example, we mention results on the variables $X_N = \max_{\theta \in [0,2\pi]} \abs{\det(e^{\ii \theta} - U_N)}$, where $U_N$ is a random matrix from the so-called Circular Unitary Ensemble. These variables are conjectured \cite{FyoKea2007, FyoHiaKea2012} to satisfy
\[
	X_N - (\log N - \frac{3}{4}\log \log N) \overset{N \to \infty}{\to} \frac{1}{2}(G_1+G_2) \qquad \text{in distribution},
\]
where $G_1$ and $G_2$ are independent Gumbel variables, i.e., the limiting distribution is supposed to be \emph{almost} classical. Recent results have proved that $X_N - (\log N - \frac{3}{4} \log \log N)$ is tight \cite{ArgBelBou2017, PaqZei2018, ChhMadNaj2018}.

For non-Hermitian random matrices, however, classical extremal statistics
prevail: the spectral radius and the real part of the rightmost eigenvalue of
a Ginibre matrix follow the Gumbel distribution after an appropriate shift and rescaling
\cite{Kos1992, Rid2003, Ben2010, AkePhi2014, RidSin2014, CipErdSchXu2022directional}. The same limit behavior is conjectured for general entry distribution (universality); see \cite{CipErdSchXu2022rightmost} for the most recent result in this direction.

In contrast to the eigen\emph{values}, the eigen\emph{vectors} tend to be much less correlated, almost independent, so their
statistics tend to have the usual 
 limiting distributions. This has been demonstrated in various results on quantum unique ergodicity (that we will survey in Section \ref{subsec:delocalization}), and our current result is in line with these findings: GOE eigenvectors are so weakly dependent that even their extremes (at the level of low-rank quadratic forms) fall into the universality class of classical extremal statistics.


\subsection{Previous results on Gaussian approximations to GOE eigenvectors.}\label{sec:13}\

It is classical that GOE eigenvectors have the exact same distribution as the Gram--Schmidt orthonormalizations of i.i.d. standard Gaussian vectors. This simple fact has appeared in various forms in the literature, e.g. in \cite[Proposition 7.2]{Eat2007}, \cite[pp. 12-13]{Mec2019}, \cite[Fact 8]{GuiMai2005}. Here we recall its easy proof that goes via the Haar measure. Consider the $N \times N$ orthogonal matrix whose columns are $\ell^2$-normalized GOE eigenvectors. 
Since the GOE is invariant under conjugation by any deterministic orthogonal matrix, so is this eigenvector matrix. 
Thus it must be Haar distributed on the orthogonal group. On the other hand, one can easily check that the orthogonal matrix 
obtained by starting with columns that are independent standard Gaussian vectors and performing Gram--Schmidt orthonormalization 
on the columns has the same invariance property, so it must also be Haar distributed. But the Haar measure is unique, so they coincide in law.

However, roughly speaking, high-dimensional standard Gaussian vectors are almost orthogonal, and their norms are almost deterministically $\sqrt{N}$. This might seem to indicate that Gram--Schmidt does little beyond a rescaling, so that Haar entries are almost independent Gaussians. This principle first appeared in a 1906 result of Borel, showing that one entry of a Haar-distributed matrix behaves like a Gaussian variable \cite{Bor1906}. One can also ask for this approximation to hold \emph{jointly} in the entries, and as long as only a few Gram--Schmidt steps are performed, this essentially works. Not only the individual eigenvectors, but also their correlations and the joint distribution of not too many of them is almost Gaussian. But one has to be careful about applying this heuristic to a large number of columns: While each vector is close to Gaussian and their pair correlations are very weak, these may add up if we consider a family of too many of them. This can lead to a \emph{thresholding} phenomenon, where the ``Haar measure is almost independent Gaussian'' heuristic is good up to a certain number of joint entries, and fails beyond it.

Of course, this threshold can depend on both the observable being compared and the norms used to compare. We give several examples: First, if $Y = (y_{ij})_{i,j=1}^N$ is the matrix of independent Gaussians and $\Gamma = (\gamma_{ij})_{i,j=1}^N$ is the result after 
applying Gram--Schmidt on its columns,
 one can ask for the largest number of columns $m = m_N$ such that $\max_{i=1}^n \max_{j=1}^m |\sqrt{N}\gamma_{ij} - y_{ij}|$ still
 tends to zero in probability. 
 Jiang in  \cite{Jia2005}, \cite{Jia2006}
  identified this threshold by showing that one can take $m_N = \oo(N/(\log N))$ and that one \emph{cannot} take $m_N = CN/(\log N)$.

Second, one can ask for the largest $p_N$ and $q_N$ such that the total variation distance between the law of the $p_N \times q_N$ \emph{principal minors} of $Y$ and $\Gamma$ tends to zero. Diaconis, Eaton, and Lauritzen showed that one can take $p_N, q_N = \oo(N^{1/3})$ \cite{DiaEatLau1992}, which was improved to $p_N, q_N = \OO(N^{1/3})$ in the thesis of Collins \cite[Theorem 4.4.2]{Col2003}. Then Jiang showed both that one can take $p_N, q_N = \oo(N^{1/2})$, and the total-variation distance does \emph{not} tend to zero if $p_N, q_N = CN^{1/2}$ \cite{Jia2006}. 
Later, in simultaneous independent work, both Stewart \cite{Ste2020} and Jiang and Ma \cite{JiaMa2019} showed that one can take non-square $p_N \times q_N$ minors as long as $p_Nq_N = \oo(N)$. 
 In this context, our result says that $p_N = \OO(N^{1/2-\delta})$ and $q_N = N$ are still jointly Gaussian as far as the extremal statistics of the corresponding quadratic forms are concerned, even if they are not close in total variation. 
 
The papers \cite{Jia2005, Jia2006} are important for us; we comment on the relationship between them and our work in Remark \ref{rem:jiang_relationship} below. In fact, with a relatively simple argument, Jiang's result can be used to prove a weaker version of our main result, with the restriction \eqref{eqn:intro_rank} replaced with the non-optimal condition $\rank(A_N) \ll N^{1/3}$; see Remark \ref{rem:jiang} for details.

Finally, we mention some other related results. Jiang also showed the largest entry of a Haar-distributed matrix has Gumbel fluctuations \cite{Jia2005}, a result which had been conjectured on the basis of numerics by Donoho and Huo \cite{DonHuo2001}. Similar techniques  have been used by Guionnet and Ma\"ida in their study of rank-one Harish-Chandra--Itzykson--Zuber integrals \cite{GuiMai2005}.


\subsection{Previous results on quadratic forms of random-matrix eigenvectors.}\
\label{subsec:delocalization}

Owing to their quantum physics interpretation as measurable quantities, 
the variables $(\ip{\mathbf{u}_i, A_N \mathbf{u}_i})_{i=1}^N$ have a long history in random matrices, although to our knowledge this paper is the first study of their \emph{extremal} statistics. One motivation for these variables is
\emph{delocalization}, i.e.
the idea that for sufficiently mean-field random matrices the $\ell^2$-mass of the 
 eigenvector $\mathbf{u}_i$ should be approximately uniformly distributed across its components $\mathbf{u}_i(\alpha)$.
  For example, as an \emph{entrywise} bound one can prove that
\begin{equation}
\label{eqn:intro_delocalization_single}
	\abs{\mathbf{u}_i(\alpha)} \ll \frac{N^\epsilon}{\sqrt{N}} \quad \text{with high probability for each $i$ and $\alpha$},
\end{equation}
or even the stronger result
\begin{equation}
\label{eqn:intro_delocalization}
	\max_{i=1}^N \max_{\alpha = 1}^N \abs{\mathbf{u}_i(\alpha)} \ll \frac{N^\epsilon}{\sqrt{N}} \quad \text{with high probability},
\end{equation}
which is tight up to the factor $N^\epsilon$ since the normalized eigenvector has unit mass. 
  Of course, one can interpret $\mathbf{u}_i(\alpha) = \ip{\mathbf{u}_i,\mathbf{e}_\alpha}$ as a projection onto the $\alpha$th standard basis vector $\mathbf{e}_\alpha$, and from this perspective there is no reason to single out the standard basis. One could equally study $\ip{\mathbf{u}_i,\mathbf{q}}$ for some deterministic unit vector $\mathbf{q}$. This leads
 to the special case of our quadratic form  when $A_N$ has rank one, since
\[
	\abs{\ip{\mathbf{u}_i,\mathbf{q}}} = \sqrt{\ip{\mathbf{u}_i, (\mathbf{q}\mathbf{q}^T) \mathbf{u}_i}}.
\]
In the literature on random-matrix eigenvectors (and more generally on local laws 
for resolvents), results with generic $\mathbf{q}$ instead of $\mathbf{e}_\alpha$, which first appeared in \cite{KnoYin2013isotropic}, are called \emph{isotropic}. For example, 
 \cite{KnoYin2013isotropic} showed that 
\eqref{eqn:intro_delocalization_single} holds for $\ip{\mathbf{u}_i,\mathbf{q}}$ with any fixed deterministic
unit vector $\mathbf{q}$. 

From our perspective, high probability size estimates like \eqref{eqn:intro_delocalization_single} and \eqref{eqn:intro_delocalization} 
can be complemented and refined by distributional results of two types: 
\begin{itemize}
\item[(i)] On the one hand, in the spirit of \eqref{eqn:intro_delocalization_single}, one can show that any given $N^{1/2} \ip{\mathbf{u}_i,\mathbf{q}}$ is asymptotically Gaussian. This is a distributional result about typical behavior. 
\item[(ii)] On the other hand, in the spirit of \eqref{eqn:intro_delocalization}, one can show that 
$\max_{i=1}^N N \ip{\mathbf{u}_i,\mathbf{q}}^2$  after an appropriate shift 
 is asymptotically Gumbel. This is a distributional result about extremal behavior.
\end{itemize}
All of these results are short exercises in the special case of GOE eigenvectors, but they are highly nontrivial for other ensembles, such as general Wigner matrices, where orthogonal invariance is not present.
Generally speaking, in the literature
there are more size estimates like \eqref{eqn:intro_delocalization_single} and \eqref{eqn:intro_delocalization} than there are distributional results, and more distributional results on typical behavior than there are on extremal behavior. 

Size estimates like \eqref{eqn:intro_delocalization} were proved first for Wigner matrices in \cite{ErdSchYau2009}, then in \cite{ErdYauYin2012B} for a broader class called generalized Wigner matrices, as simple corollaries of \emph{optimal local laws}. The Gaussian fluctuations of $\sqrt{N}\ip{\mathbf{u}_i, \mathbf{q}}$ were established for generalized Wigner matrices first under a four-moment-matching condition in the bulk \cite{KnoYin2013eigenvector, TaoVu2012random} and a two-moment-matching condition at the edge \cite{KnoYin2013eigenvector} using 
comparison techniques for resolvents, then in full generality \cite{BouYau2017} using Dyson-Brownian motion techniques.

To the best of our knowledge, Gumbel fluctuations of $\max_i N  \ip{\mathbf{u}_i, \mathbf{q}}^2$ are still open beyond GOE and GUE.
 The corresponding size estimates have been obtained only recently for generalized Wigner matrices with subexponential entries, in work of Benigni and Lopatto \cite[Theorem 1.2]{BenLop2022}, who showed that for every $D > 0$ there exists $C_D > 0$ with
\begin{equation}
\label{eqn:benignilopatto}
	\sup_{\mathbf{q} \in \mathbb{S}^{N-1}} \P\left( \max_{i = 1}^N N \ip{\mathbf{u}_i,\mathbf{q}}^2 \geq C_D\log N \right) \leq C_DN^{-D}.
\end{equation}
In comparison, the rank-one case of our result shows, informally, that for GOE eigenvectors
\begin{equation}\label{rank1}
	\max_{i=1}^N N \ip{\mathbf{u}_i,\mathbf{q}}^2 \approx 2\log N - \log \log N - \log \pi + 2\Lambda
\end{equation}
with $\Lambda$ a Gumbel variable. The analogous result for
GUE is:
\begin{equation}\label{rank1gue}
	\max_{i=1}^N N \ip{\mathbf{u}_i,\mathbf{q}}^2 \approx \log N + \Lambda;
\end{equation}
see also \cite{LakTomBohMaj2008} and Remark~\ref{physics}.

Our goal is to understand the distributional analogue of the high probability bound \eqref{eqn:benignilopatto}
  in the special case of GOE and GUE, but for higher-rank observables. This extension is naturally related to the celebrated phenomena of \emph{quantum ergodicity} (QE) and \emph{quantum unique ergodicity} (QUE) that 
  appear in a variety of disordered quantum systems. They say roughly that, for $\psi_1, \psi_2, \ldots$ the eigenfunctions of some Hamiltonian operator, $A$ a deterministic operator (\emph{``observable''})
 in some appropriate class, and $f$ a model-dependent
 linear functional on this class, we have
\[
	\lim_{i, j \to \infty} \ip{\psi_i, A \psi_j} \to \delta_{ij} f(A).
\]
Results proving this for \emph{most} pairs $(i,j)$, sometimes by averaging in $i$ and $j$, are termed QE. 
Results proving this along  \emph{all} $(i,j)$ sequences are termed QUE.
 In the context of the Laplace-Beltrami operator on a Riemannian manifold with an ergodic geodesic flow,
  QUE was defined and conjectured by Rudnick and Sarnak \cite{RudSar1994}, following earlier QE results for the same model by Shnirelman \cite{Sni1974}, Colin de Verdi\`{e}re \cite{Col1985}, and Zelditch \cite{Zel1987}. 
  In this context  $A$ is from a suitable class of pseudodifferential operators and $f(A)$ is the integral of its symbol on the unit cotangent bundle.
  To date, the conjectures of Rudnick and Sarnak have been proved only in certain special cases of arithmetic surfaces \cite{Lin2006, Sou2010, HolSou2010}. In recent years, these phenomena have also been established in regular graphs, both deterministic (in work of Anantharaman and Le Masson \cite{AnaLeM2015}) and random (in work of Bauerschmidt, Huang, Knowles, and Yau \cite{BauHuaYau2019, BauKnoYau2017}).

In random matrices, QE and QUE naturally appear when considering $\ip{\mathbf{u}_i, A_N \mathbf{u}_j}$ for $(\mathbf{u}_i)_{i=1}^N$ the eigenvectors of the random matrix and $A_N$ a deterministic matrix. For generalized Wigner matrices, Bourgade and Yau 
in  \cite{BouYau2017} showed a local form of QUE in the bulk of the spectrum, namely that for $i$ any bulk index we have  
\[
	\frac{N}{\rank(A_N)} \Big(\ip{\mathbf{u}_i, A_N \mathbf{u}_i} - \frac{1}{N}\Tr A_N\Big) \to 0 \quad \text{in probability, as long as } \rank(A_N) \to \infty.
\]
The corresponding QUE result, showing this limit for all eigenvectors simultaneously, was established in \cite{BouYauYin2020} for matrices with a Gaussian component, strengthened in \cite{Ben2021}, and resolved in full generality for Wigner matrices 
 in \cite{CipErdSch2021} with the optimal speed of convergence.  These QUE statements are the analogues of  \eqref{eqn:intro_delocalization_single}
for general observables. Concerning the refinement of these QUE  results in the distributional direction for the typical behavior
(in the spirit of (i) above), we mention that
Gaussian fluctuations around the QUE  were proven in  \cite{BenLop2022QUE} for special low-rank observables for all $i$, in 
\cite{CipErdSch2022normal} for general full-rank observables, and finally for any observable in \cite{CipErdSch2022rankuniform} in the bulk.
 More precisely, Theorem 2.8 in \cite{CipErdSch2022rankuniform} asserts that for any Wigner eigenvector $\mathbf{u}_i$ with 
 eigenvalue away from the spectral edges we have
 \begin{equation}\label{CLTQUE}
  Q_N(i):=\frac{N}{\sqrt{2 \Tr (\mathring{A}_N)^2}} \Big(\ip{\mathbf{u}_i, A_N \mathbf{u}_i} - \frac{1}{N}\Tr A_N\Big) \overset{N \to \infty}{\to} \mathcal{N}
 \end{equation}
 as long as $\Tr (\mathring{A}_N)^2\to \infty$, 
 where $\mathring{A}_N: = A_N - \frac{1}{N} \Tr A_N$ is the traceless part of $A_N$ and $\mathcal{N}$  stands for the standard normal distribution. 
 
 In this paper we prove distributional results for the corresponding extremal statistics in the spirit of (ii) above.
 The asymptotically normal random variables $Q_N(i)$ are expected to be almost independent for different indices $i$, hence
 their extremal statistics should be classical. 
 Indeed, our Theorem~\ref{thm:main_diverging} below shows the Gumbel distribution emerging for $\max_i Q_N(i)$,
 at least for GOE/GUE eigenvectors  and $A_N=\diag(1,1, \ldots 1, 0, \ldots 0)$.  When $A_N$ has bounded rank, but is otherwise general, then
 \eqref{CLTQUE} does not hold since  $\ip{\mathbf{u}_i, A_N \mathbf{u}_i}$ is the sum of a few 
 asymptotically $\chi^2$ distributed random variables. Nevertheless the extremal statistics of the $Q_N(i)$'s are still classical: either Gumbel or Weibull,
 depending on the signature of $A_N$  (see Theorem~\ref{thm:main_fixed}). These results may be considered
 as yet another manifestation of the chaotic behavior of the GOE/GUE eigenvectors via a large class of statistics (parametrized by $A_N$) that are sensitive to \emph{all} eigenvectors at the same time.



\subsection{Organization and Notation.}\

The organization of the paper is as follows: In Section \ref{sec:results}, we give our main result comparing the variables \eqref{eqn:intro_variables} and their Gaussian counterparts, as well as some applications, and a sketch of the proof. The proof relies heavily on the key technical Proposition \ref{prop:undoing_gram_schmidt}, which essentially asserts that, for the purpose of our observable, applying Gram--Schmidt to Gaussian vectors is little more than rescaling; the proof of this result constitutes Section \ref{sec:undoing_gram_schmidt}. Finally, the applications of our main result to fixed-rank and diverging-rank problems rely on Gaussian computations, which are carried out in Sections \ref{sec:gaussian_fixed} and \ref{sec:gaussian_diverging}, respectively.  

Throughout the paper, we use the following notation. 
We write $\|\mathbf{v}\|$ for the Euclidean norm of a vector $\mathbf{v}$ (always writing vectors in boldface) and $\|M\|$ for the corresponding operator norm of a matrix $M$. If $M$ has size $a \times b$, we write
\[
	\vertiii{M} \defeq \max_{i=1}^a \max_{j=1}^b \abs{M_{ij}}
\] 
for the entrywise maximum norm (in particular, we can take $\vertiii{\cdot}$ of a vector $\mathbf{v}$). The standard notation $X \overset{d}{=} Y$ means that the random variables $X$ and $Y$ have the same distribution, and $a_N = \OO(b_N)$ means that there is some constant $C$ such that the sequences $(a_N)_{N=1}^\infty, (b_N)_{N=1}^\infty$ satisfy $a_N \leq C b_N$. The notation $\diag(m_1, \ldots, m_N)$ means the diagonal matrix with entries $m_1, \ldots, m_N$ along the diagonal.


\subsection*{Acknowledgements.}\

LE was supported by the ERC Advanced Grant ``RMTBeyond'' No. 101020331. BM was supported by Fulbright Austria and the Austrian Marshall Plan Foundation.


\section{Results}
\label{sec:results}


\subsection{Main result.}\

Our main result states that the extremal statistics of quadratic forms of GOE eigenvectors 
coincide with those of independent Gaussian vectors as long as the rank of $A_N$ is much smaller than $N^{1/2}$. 
To formulate it precisely, 
let $(\mathbf{u}_i)_{i=1}^N$ be the $\ell^2$-normalized eigenvectors of an $N \times N$ GOE matrix, and let $(\mathbf{y}_i)_{i=1}^N$ be i.i.d. standard Gaussian vectors of length $N$, meaning that $\mathbf{y}_i \sim \mc{N}(0,\Id_{N \times N})$ for each $i$ (so that $\|\mathbf{y}_i\|^2 \approx N$).

\begin{thm}
\label{thm:main}
Let $(A_N)_{N=1}^\infty$ be a sequence of real symmetric deterministic matrices such that, for some $\delta > 0$, 
\begin{equation}
\label{eqn:main_thm_rank}
	k_N \defeq \rank(A_N) = \OO(N^{1/2-\delta})
\end{equation}
and
\[
	\sup_N \|A_N\| < \infty.
\]
\begin{enumerate}
\item If the deterministic real sequences $(c_N)_{N=1}^\infty$ and $(d_N)_{N=1}^\infty$, and the probability distribution $\chi$, are such that
\[
	c_N \max_{i=1}^N \ip{\mathbf{y}_i, A_N \mathbf{y}_i} + d_N \overset{N \to \infty}{\to} \chi \qquad \text{in distribution}
\]
and
\[
	\sup_N \abs{c_N} < \infty,
\]
then
\[
	c_N \max_{i=1}^N N \ip{\mathbf{u}_i, A_N \mathbf{u}_i} + d_N \overset{N \to \infty}{\to} \chi \qquad \text{in distribution.}
\]
\item If the deterministic real sequences $(c'_N)_{N=1}^\infty$ and $(d'_N)_{N=1}^\infty$, and the probability distribution $\chi'$, are such that
\[
	c'_N \max_{i=1}^N \abs{\ip{\mathbf{y}_i, A_N \mathbf{y}_i}} + d'_N \overset{N \to \infty}{\to} \chi' \qquad \text{in distribution}
\]
and
\[
	\sup_N \abs{c'_N} < \infty,
\]
then
\[
	c'_N \max_{i=1}^N N \abs{\ip{\mathbf{u}_i, A_N \mathbf{u}_i}} + d'_N \overset{N \to \infty}{\to} \chi' \qquad \text{in distribution.}
\]
\end{enumerate}
\end{thm}

\begin{rem}
Our result is stated for the real case, but the same result holds for the complex case, meaning when the $(\mathbf{u}_i)_{i=1}^N$ are eigenvectors of the Gaussian \emph{Unitary} Ensemble (GUE), 
 the $(\mathbf{y}_i)_{i=1}^N$ are complex Gaussian vectors normalized so that $\|\mathbf{y}_i\|^2 \approx N$, and the $A_N$ are complex Hermitian matrices.
  For simplicity we write the proof only for the real case,
as the argument for the complex case is exactly the same with 
some natural minor adjustments.
\end{rem}

\begin{rem}
\label{rem:jiang_relationship}
We comment on the relationship between our work and that of Jiang \cite{Jia2005, Jia2006}, which is a very important precedent for us. 
Our paper and his consider the same underlying objects 
(what we will call $\sqrt{N}\gamma_{ij} - y_{ij}$ in, e.g., \eqref{eqn:for_jiang_remark}) 
to compare the outcome of a Gram--Schmidt orthonormalization with independent Gaussian.
 His work gives essentially tight estimates on their size. However, for the purpose of the extremal statistics
 our result boils down to estimating a certain sum of products of these objects. 
 One could estimate this sum by the absolute value of each summand  using Jiang's estimates; this would lead to our result but with the tighter and non-optimal
  rank restriction 
\[
	k_N = \OO(N^{1/3-\delta}).
\]
The details are given in Remark \ref{rem:jiang} below. To reach the exponent almost $1/2$, which we believe is a threshold (see Remark \ref{rem:rank}), one has to estimate the cancellations between terms in this sum; this is substantially more complicated than applying
 the term by term  estimates and it requires a quite different method. Our proof still uses some results from \cite{Jia2005}, namely in Proposition \ref{prop:jiang_bounds} and its corollary Lemma \ref{lem:delta_replacement_good_event} below. These are borrowed for convenience: by replacing the good event there with a high-moment expansion, like we give in the proof of Lemma \ref{lem:graphs}, we could have written a formally independent proof.
\end{rem}

\begin{rem}
\label{rem:rank}
Of course, one wonders to what extent the restriction \eqref{eqn:main_thm_rank} is tight for Theorem~\ref{thm:main} to hold.
 \emph{Some} restriction on $A_N$ is necessary: If we take $A_N = \Id = \Id_{N \times N}$, then $\max_{i=1}^N \ip{\mathbf{y}_i, \mathbf{y}_i}$ is the maximum of independent chi-squared random variables, so has Gumbel fluctuations, but $\max_{i=1}^N \ip{\mathbf{u}_i, \mathbf{u}_i} = 1$ deterministically.

If extremal quadratic forms of GOE eigenvectors cease to have Gumbel fluctuations as the rank of $A_N$ increases,
then the relationship between this phase transition and our result may be subtle, since our result relies on a comparison
between the GOE eigenvectors $(\mathbf{u}_i)_{i=1}^N$ and the \emph{independent} Gaussian vectors $(\mathbf{y}_i)_{i=1}^N$.
For simplicity, consider the case when $A_N = \diag(1,\ldots,1,0,\ldots,0)$ with $k_N = N^\alpha$ for some $\alpha \in (0,1)$. It may happen that there are two critical thresholds, $\alpha_{\textup{i.i.d.}}$ and $\alpha_{\textup{Gumbel}}$, such that 
\begin{enumerate}[label=(\roman*)]
\item $\max_{i=1}^N N \ip{\mathbf{u}_i, A_N\mathbf{u}_i}$ 
is well approximated by the independent-Gaussian analogue
$\max_{i=1}^N \ip{\mathbf{y}_i, A_N\mathbf{y}_i}$ for $\alpha < \alpha_{\textup{i.i.d.}}$ but not beyond,
\item and $\max_{i=1}^N N \ip{\mathbf{u}_i, A_N\mathbf{u}_i}$ has Gumbel fluctuations for $\alpha < \alpha_{\textup{Gumbel}}$ (these two already imply $\alpha_{\textup{i.i.d.}} \leq \alpha_{\textup{Gumbel}}$), 
\item but $\alpha_{\textup{i.i.d.}} < \alpha_{\textup{Gumbel}}$ with strict inequality.
\end{enumerate}

If such thresholds exist, we believe that probably $\alpha_{\textup{i.i.d.}} = 1/2$, but we leave open the value of $\alpha_{\textup{Gumbel}}$. It may happen that, for some $\alpha > \alpha_{\textup{i.i.d.}}$ (possibly even for any $\alpha<1$), one still has Gumbel limiting distribution,
but for the correct Gaussian approximation
one should take the vectors $\mathbf{y}_i$ to be \emph{appropriately correlated} Gaussians.
\end{rem}

\begin{rem}
\label{rem:universality}
It would be interesting to study \emph{universality}, i.e., to see whether Theorem~\ref{thm:main} is true when the $\mathbf{u}_i$'s are the
 eigenvectors of a non-invariant ensemble, such as generic Wigner matrices.
At the moment, the main difficulty is that existing universality proofs relying on the Dyson Brownian motion typically consider only finitely many eigenvectors simultaneously (see \cite{MarYau2020} for the strongest result to date), while $\max_{i=1}^N \ip{\mathbf{u}_i, A_N\mathbf{u}_i}$ involves \emph{all} eigenvectors at the same time.
\end{rem}

The value of  Theorem~\ref{thm:main} is that it reduces the study of $\max_{i=1}^N \ip{\mathbf{u}_i, A_N \mathbf{u}_i}$ and $\max_{i=1}^N \abs{\ip{\mathbf{u}_i, A_N \mathbf{u}_i}}$ to the corresponding Gaussian problems. The latter are easier, since they are of the form $\max_{i=1}^N f_{A_N}(\mathbf{y}_i)$ for some deterministic functions $f_{A_N}(\cdot)$, and especially since the variables $f_{A_N}(\mathbf{y}_i)$ are i.i.d. The following subsection gives several illustrative examples of how this theorem can be used in practice.


\subsection{Corollaries and physical application.}\

To illustrate the use of Theorem~\ref{thm:main}, we prove limit laws for $\max_{i=1}^N \ip{\mathbf{u}_i, A_N \mathbf{u}_i}$ and $\max_{i=1}^N \abs{\ip{\mathbf{u}_i, A_N \mathbf{u}_i}}$ for two special cases of $A_N$'s  via performing
the corresponding Gaussian calculation essentially explicitly. 
Theorem \ref{thm:main_fixed}, with proof in Section \ref{sec:gaussian_fixed}, completely characterizes the case when the matrices $A_N$ have a fixed rank and fixed eigenvalues; Theorem \ref{thm:main_diverging}, with proof in Section \ref{sec:gaussian_diverging}, considers the special case of diverging rank when $A_N = \diag(1,\ldots,1,0,\ldots,0)$ and $\rank(A_N) \approx N^\alpha$ for some $0 < \alpha < 1/2$. 
Similar calculations for more general observables with diverging ranks can also be carried out, but for the sake of brevity we refrain from doing so.

\begin{thm}
\label{thm:main_fixed}
Fix $k \in \N$ and nonzero real numbers $a_1, \ldots, a_k$. Set
\begin{equation}
\label{eqn:defmult}
\begin{split}
	&a \defeq \max_{i=1}^k a_i, \qquad m \defeq \#\{i : a_i = a\}, \\
	&a_\ast \defeq \max_{i=1}^k \abs{a_i}, \qquad m_+ \defeq \#\{i : a_i = a_\ast\}, \qquad m_- \defeq \#\{i : a_i = -a_\ast\},
\end{split}
\end{equation}
as well as
\begin{align}
	c_m(a_1, \ldots, a_k) &\defeq \log\left(\Gamma\left(\frac{m}{2}\right) \sqrt{\prod_{\substack{j \in \llbracket 1, k \rrbracket \\ a_j \neq a}} \left(1-\frac{a_j}{a}\right)} \right), \label{eqn:defc} \\
	c^\ast_{m_+,m_-}(a_1, \ldots, a_k) &\defeq \begin{cases} \log\left( \Gamma \left(\frac{m_+}{2}\right) \sqrt{\prod_{\substack{j \in \llbracket 1, k \rrbracket \\ a_j \neq a_\ast}} \left(1-\frac{a_j}{a_\ast}\right)} \right) & \text{if } m_+ > m_- \\ 
	\log\left( \Gamma \left(\frac{m_-}{2}\right) \sqrt{\prod_{\substack{j \in \llbracket 1, k \rrbracket \\ a_j \neq -a_\ast}} \left(1+\frac{a_j}{a_\ast}\right)} \right) & \text{if } m_+ < m_- \\
	\log\left( \Gamma \left(\frac{m_+}{2}\right) \left[ \left( \sqrt{\prod_{\substack{j \in \llbracket 1, k \rrbracket \\ a_j \neq a_\ast}} \left(1-\frac{a_j}{a_\ast}\right)} \right)^{-1} + \left(\sqrt{\prod_{\substack{j \in \llbracket 1, k \rrbracket \\ a_j \neq -a_\ast}} \left(1+\frac{a_j}{a_\ast}\right)} \right)^{-1} \right]^{-1} \right) & \text{if } m_+ = m_- \end{cases} \label{eqn:defcbar}
\end{align}
with the empty product interpreted as one as usual. Write $\Lambda$ for a Gumbel-distributed random variable, with distribution function $\P(\Lambda \leq x) = \exp(-e^{-x})$, and $\Psi_{k/2}$ for a $\frac{k}{2}$-Weibull-distributed random variable, with distribution function $\P(\Psi_{k/2} \leq x) = \max(\exp(-(-x)^{k/2}),1)$. Suppose that
\[
	\Spec(A_N) = \{a_1, a_2, \ldots, a_k, 0, \ldots, 0\}.
\]
\begin{enumerate}
\item Suppose that $a > 0$. Then we have
\begin{equation}
\label{eqn:main_apos_noabs}
	\frac{N}{2a} \max_{i=1}^N \ip{\mathbf{u}_i,A_N \mathbf{u}_i} - \log N + \left(1-\frac{m}{2}\right) \log \log N + c_m(a_1, \ldots, a_k) \overset{N \to \infty}{\to} \Lambda \quad \text{in distribution}.
\end{equation}
and
\begin{equation}
\label{eqn:main_apos_abs}
	\frac{N}{2a_\ast} \max_{i=1}^N \abs{\ip{\mathbf{u}_i,A_N \mathbf{u}_i}} - \log N + \left(1-\frac{\max(m_+,m_-)}{2}\right) \log\log N + c^\ast_{m_+,m_-}(a_1, \ldots, a_k) \overset{N \to \infty}{\to} \Lambda \quad \text{in distribution}.
\end{equation}
\item Suppose that $a < 0$. Writing $\gamma_k = \frac{1}{2}\left(\frac{2}{k\Gamma(k/2)}\right)^{2/k}$, we have
\begin{equation}
\label{eqn:main_aneg}
	\frac{\gamma_k N^{1+\frac{2}{k}}}{\left( \prod_{k=1}^k \abs{a_j} \right)^{1/k}} \max_{i=1}^N \ip{\mathbf{u}_i,A_N\mathbf{u}_i} \overset{N \to \infty}{\to} \Psi_{k/2} \quad \text{in distribution}.
\end{equation}
(If $a < 0$, then $\max_{i=1}^N \abs{\ip{\mathbf{u}_i, A_N \mathbf{u}_i}} = \max_{i=1}^N \ip{\mathbf{u}_i, (-A_N) \mathbf{u}_i}$, since $-A_N$ is positive semidefinite; thus this ``fourth case'' can be treated with \eqref{eqn:main_apos_noabs}.)
\end{enumerate}
\end{thm}

\begin{rem} Note that
the maximal statistics are all asymptotically Gumbel-distributed except for the case when all the $a_i$'s are negative, in which case they are asymptotically Weibull-distributed. This is not surprising, since the latter case is essentially the same as the \emph{minimum} of independent $\chi^2$ random variables. The latter are asymptotically Weibull-distributed, related to the fact that the minimum has no tails: it is deterministically nonnegative.
Also note that the  $N$-scaling in~\eqref{eqn:main_aneg} is different from~\eqref{eqn:main_apos_noabs} and \eqref{eqn:main_apos_abs}. This means that the Weibull results are not a direct corollary of Theorem \ref{thm:main} (roughly, they would require $c_N = N^{2/k}$ which is not bounded), but the elementary variant of Theorem \ref{thm:main} needed here is given in Section \ref{subsec:weibull}.
\end{rem}

\begin{rem}\label{physics}
The rank-one, complex-Hermitian special case of this result already appeared in the physics literature, in a dual form. Indeed, if $\mathbf{q}$ is a complex unit vector, $(\mathbf{u}_i)_{i=1}^N$ are GUE eigenvectors, and $\mathbf{e}_1 = (1,0,\ldots,0)$, then by rotation invariance 
\begin{align*}
	\max_{i=1}^N \ip{\mathbf{u}_i, \mathbf{q}\mathbf{q}^\ast \mathbf{u}_i} &\overset{d}{=} \max_{i=1}^N \ip{\mathbf{u}_i, \mathbf{e}_1\mathbf{e}_1^\ast \mathbf{u}_i} = \max_{i=1}^N \abs{\mathbf{u}_i(1)}^2 \overset{d}{=} \max_{i=1}^N \abs{\mathbf{u}_1(i)}^2, \\
	\max_{i=1}^N \ip{\mathbf{u}_i,(-\mathbf{q}\mathbf{q}^\ast) \mathbf{u}_i} &\overset{d}{=} - \min_{i=1}^N \abs{\mathbf{u}_1(i)}^2,
\end{align*}
i.e., we are just considering the largest and smallest squared components of a uniform random vector on the unit sphere. 
 For these quantities, Lakshminarayan, Tomsovic, Bohigas, and Majumdar \cite{LakTomBohMaj2008} rigorously showed that
\begin{align*}
	N \max_{i=1}^N \abs{\mathbf{u}_1(i)}^2 - \log N \overset{N \to \infty}{\to} \Lambda \quad \text{in distribution,} \\
	-N^2 \min_{i=1}^N \abs{\mathbf{u}_1(i)}^2 \overset{N \to \infty}{\to} \Psi_1 \quad \text{in distribution.}
\end{align*}
We mention that these formulas  look a bit different from Theorem~\ref{thm:main_fixed}, which is written for the real case, but they follow 
directly  from Theorem~\ref{thm:main} and easy classical computations for complex Gaussians.

The physical motivation of  the authors of \cite{LakTomBohMaj2008} was to detect quantum chaos via extremal statistics of eigenvectors. In 
particular, they performed extensive numerics for the eigenfunctions of the quantum kicked rotor model in the strongly chaotic parameter regime
and found Gumbel and Weibull distributions emerging.  As often happens in physics  of disordered quantum systems, 
they used GUE eigenvectors as phenomenological replacements for the actual eigenvectors and for this test case  they could actually 
prove the limiting behavior. Our result for higher rank observables $A_N$ may inspire an analogous numerical study for the 
kicked rotor with higher rank observables, but such investigation would go beyond the scope of this paper.
\end{rem}

\begin{thm}
\label{thm:main_diverging}
Fix $0 < \alpha < 1/2$, and let $(A_N)_{N=1}^\infty$ be any sequence of $N \times N$ matrices of the form
\[
	A_N = \diag(1,\ldots,1,0,\ldots,0)
\]
such that, for some $\epsilon > 0$, 
\[
	\abs{\rank(A_N) - N^\alpha} \leq N^{\frac{\alpha}{2}-\epsilon}.
\]
Then
\begin{equation}
\label{eqn:main_diverging}
	\left(\frac{\sqrt{\log N}}{N^{\alpha/2}}\right) \max_{i=1}^N N \ip{\mathbf{u}_i, A_N \mathbf{u}_i} - N^{\alpha/2}\sqrt{\log N} - 2\log N + \frac{\log \log N}{2} + \frac{\log (4\pi)}{2} \overset{N \to \infty}{\to} \Lambda \quad \text{in distribution.}
\end{equation}
\end{thm}


\subsection{Main steps in the proof of Theorem \ref{thm:main}.}\

The proof of Theorem \ref{thm:main} goes by coupling. 
Namely, recalling that $(\mathbf{y}_i)_{i=1}^N$ are i.i.d. standard Gaussian vectors of length $N$ with entries $(\mathbf{y}_i)_j = y_{ij}$, we will need both their Gram--Schmidt orthogonalizations $(\mathbf{w}_i)_{i=1}^N$ (with entries $(\mathbf{w}_i)_j = w_{ij}$), defined by $\mathbf{w}_1 = \mathbf{y}_1$ and 
\[
	\mathbf{w}_i \defeq \mathbf{y}_i -  \sum_{j=1}^{i-1} \frac{\ip{\mathbf{y}_i,\mathbf{w}_j}}{\|\mathbf{w}_j\|^2} \mathbf{w}_j, \quad i = 2, \ldots, N,
\]
as well as their orthonormalizations 
\[
	\bm{\gamma}_i \defeq \frac{\mathbf{w}_i}{\|\mathbf{w}_i\|},
\]
with entries $(\bm{\gamma}_i)_j = \gamma_{ij}$.

The following simple lemma is classical and its proof was already explained in Section~\ref{sec:13}.
\begin{lem}
\label{lem:recognizing_gram_schmidt}
Let $(\mathbf{u}_i)_{i=1}^N$ be $\ell^2$-normalized GOE eigenvectors. Then 
\[
	(\mathbf{u}_1, \ldots, \mathbf{u}_N) \overset{d}{=} (\bm{\gamma}_1, \ldots, \bm{\gamma}_N).
\]
\end{lem}
In other words, we realize the eigenvectors as the outcome of a Gram--Schmidt procedure 
on independent Gaussian vectors $\mathbf{y}_i$ and then we estimate the error. Since $\mathbf{y}$'s 
are almost orthogonal, the effect of the first few Gram--Schmidt steps is minor, but it adds up
when we perform $k$ of them.  The following proposition, which is at the heart of the proof of Theorem \ref{thm:main}, 
shows that $k\ll \sqrt{N}$ steps are still controllable as far as the extremal statistics of quadratic forms are concerned. Its
 proof takes up the entire Section \ref{sec:undoing_gram_schmidt}; some more intuition
 is explained in Section \ref{sec:prelim}.

\begin{prop}
\label{prop:undoing_gram_schmidt}
Fix $\delta > 0$, and take some sequence $(k = k_N)_{N=1}^\infty$ of positive integers satisfying
\[
	k = \OO(N^{1/2-\delta}).
\]
For each $k$, choose a $k$-tuple of deterministic real numbers $(a_1, \ldots, a_k) = (a_1^{(k)}, \ldots, a^{(k)}_k)$, and suppose there is some fixed $a > 0$, independent of $N$ and of $k$, with 
\begin{equation}
\label{eqn:as_bounded}
	\max_k \max_{i=1}^k \abs{a_i^{(k)}} \leq a.
\end{equation}
Then
\begin{equation}
\label{eqn:undoing_gram_schmidt_abs}
	\max_{j=1}^N \abs{\sum_{i=1}^k a_i (\sqrt{N}\gamma_{ij})^2 } - \max_{j=1}^N \abs{\sum_{i=1}^k a_i y_{ij}^2} \overset{N \to \infty}{\to} 0 \quad \text{in prob.}
\end{equation}
and
\begin{equation}
\label{eqn:undoing_gram_schmidt_noabs}
	\max_{j=1}^N \left( \sum_{i=1}^k a_i (\sqrt{N}\gamma_{ij})^2 \right) - \max_{j=1}^N \left( \sum_{i=1}^k a_i y_{ij}^2 \right) \overset{N \to \infty}{\to} 0 \quad \text{in prob.}
\end{equation}
\end{prop}

\begin{rem}
A close reading of the  proof gives effective estimates of the form $\P( | \ldots| \ge \epsilon) \le C_{D,\epsilon} N^{-D}$ for any fixed 
$\epsilon$ and $D$ for \eqref{eqn:undoing_gram_schmidt_abs} and \eqref{eqn:undoing_gram_schmidt_noabs}. 
\end{rem}

\begin{proof}[Proof of Theorem \ref{thm:main}]
We give the proof for the version without absolute values, the version with absolute values being similar. Write $\mathbf{Y}$ (respectively $\mathbf{U}$, respectively $\bm{\Gamma}$) for the matrix whose columns are $\mathbf{y}_1, \ldots, \mathbf{y}_N$ (respectively $\mathbf{u}_1, \ldots, \mathbf{u}_N$, respectively $\bm{\gamma}_1, \ldots, \bm{\gamma}_N$). Since the distributions of $\mathbf{Y}$ and $\mathbf{U}$ are each invariant under orthogonal conjugation, we can assume without loss of generality that $A_N$ is diagonal, and has the form
\[
	A_N = \diag(a_1, \ldots, a_k, 0, \ldots, 0)
\]
for some real, nonzero $a_1, \ldots, a_k$. Furthermore, since the distributions of $\mathbf{Y}$ and $\bm{\Gamma}$ are each invariant under switching rows and columns, we have 
\begin{align*}
	\max_{i=1}^N \ip{\mathbf{y}_i, A_N \mathbf{y}_i} &= \max_{i=1}^N \sum_{j=1}^k a_j y_{ij}^2 \overset{d}{=} \max_{j=1}^N \sum_{i=1}^k a_i y_{ij}^2, \\
	\max_{i=1}^N N \ip{\mathbf{u}_i, A_N \mathbf{u}_i} &\overset{d}{=} \max_{i=1}^N N \ip{\bm{\gamma}_i, A_N \bm{\gamma}_i} = \max_{i=1}^N \sum_{j=1}^k a_j (\sqrt{N}\gamma_{ij})^2 \overset{d}{=} \max_{j=1}^N \sum_{i=1}^k a_i (\sqrt{N}\gamma_{ij})^2,
\end{align*}
where the first equality in distribution in the second line follows from Lemma \ref{lem:recognizing_gram_schmidt}. The rest of the proof follows from Proposition \ref{prop:undoing_gram_schmidt}.
\end{proof}

\begin{rem}
The swapping of row and column indices in the above proof is very important: It allows one to consider, not the first few entries of all the $\bm{\gamma}_i$'s, but all the entries of the first few $\bm{\gamma}_i$'s. Since the effect of Gram--Schmidt is smaller on vectors considered earlier in the process, this is a beneficial switch.  We 
point out  that we do not know how to prove the analogue of Proposition \ref{prop:undoing_gram_schmidt} that considers $\max_{j=1}^k (\sum_{i=1}^N a_i  (\sqrt{N}\gamma_{ij})^2) - \max_{j=1}^k (\sum_{i=1}^N a_i y_{ij}^2)$.
For the remainder of the paper, the first index in terms like $\gamma_{ij}$, typically denoted by $i$ and subject to $i\le k$,
 will always track the order in the Gram--Schmidt orthogonalization procedure.
\end{rem}


\section{Undoing Gram--Schmidt: Proof of Proposition \ref{prop:undoing_gram_schmidt}}
\label{sec:undoing_gram_schmidt}


\subsection{Preliminaries.}\label{sec:prelim}\

The Gram--Schmidt vectors $\bm{\gamma}$ are complicated rational functions of the Gaussian vectors $\mathbf{y}$, but
to leading order $\bm{\gamma}_i \approx \mathbf{y}_i/\sqrt{N}$, at least when $i$ is not too big; this is the main intuition 
behind Proposition \ref{prop:undoing_gram_schmidt}. We need to estimate the effect of the error terms on the
extremal statistics and it will be done in several steps; see Lemma  \ref{lem:deltasdeltas} below.
In some easier steps, high probability Gaussian concentration bounds
on the $\mathbf{y}$'s suffice (Section~\ref{sec:easy}) -- this is similar to Jiang's method and it is sufficient up to $k\ll N^{1/3}$,
see Remark \ref{rem:jiang}.
However, to reach our $k\ll N^{1/2}$ threshold, we need to use a delicate cancellation mechanism
in a big sum of Gaussian monomials; in Section~\ref{sec:hard} we apply a fully nonlinear chaos expansion and bookkeep 
the various terms by graphs, reminiscent to a 
Feynman diagrammatic expansion. Here we crucially use the martingale structure built in the Gram--Schmidt procedure: for any index $i$, the
first $i$ Gram--Schmidt vectors depend only on the first $i$ Gaussian vectors.

 Define the error vectors $\bm{\Delta}_i$ (with entries $(\bm{\Delta}_i)_j = \Delta_{ij}$) given by
\[
	\bm{\Delta}_1 \defeq \bm{0}, \quad \bm{\Delta}_i \defeq \sum_{\ell=1}^{i-1} \ip{\mathbf{y}_i,\bm{\gamma}_\ell}\bm{\gamma}_\ell \qquad \text{for } i = 2, \ldots, N,
\]
and notice that
\[
	\mathbf{w}_i = \mathbf{y}_i - \sum_{\ell=1}^{i-1} \ip{\mathbf{y}_i,\bm{\gamma}_\ell} \bm{\gamma}_\ell = \mathbf{y}_i - \bm{\Delta}_i.
\]

The variables $\bm{\Delta}_i$ appear naturally in the Gram--Schmidt procedure, 
 but they include the normalization of $\mathbf{w}$ which is not convenient for expansions, so
 we will replace them by the variables $\widetilde{\bm{\Delta}}_i$, defined as
\begin{align*}
	\widetilde{\bm{\Delta}}_1 \defeq 0, \qquad \widetilde{\bm{\Delta}}_i &\defeq \frac{1}{N} \sum_{\ell=1}^{i-1} \ip{\mathbf{y}_i, \mathbf{y}_\ell} \mathbf{y}_\ell \quad \text{for} \quad i = 2, \ldots, N,
\end{align*}
which are polynomials in the $\mathbf{y}$'s. Note that $\widetilde{\bm{\Delta}}/\sqrt{N}$ is the next order correction 
to $\bm{\gamma}_i \approx \mathbf{y}_i/\sqrt{N}$ in Gram--Schmidt.
The main goal of Section \ref{sec:undoing_gram_schmidt} is to prove the following proposition, which essentially shows that higher-order corrections
can be ignored for our purpose. Proposition \ref{prop:undoing_gram_schmidt} will relatively easily follow from it in Section~\ref{sec:210}.
\begin{prop}
\label{prop:deltadeltatilde}
Fix $\delta > 0$. If   
\[
	k = \OO(N^{1/2 - \delta}),
\]
then for every $D > 0$ there exists $C_D > 0$ with
\begin{equation}
\label{eqn:prop_deltadeltatilde}
	\P\left( \max_{i=1}^k \max_{j=1}^N \abs{\Delta_{ij} - \widetilde{\Delta}_{ij}} \geq \frac{1}{N^{1/2}} \right) \leq C_DN^{-D}.
\end{equation}
\end{prop}

We introduce the intermediate quantities
\begin{align*}
	\bm{\Delta}^{(1)}_i &\defeq \frac{1}{N} \sum_{\ell=1}^{i-1} \ip{\mathbf{y}_i,\mathbf{w}_\ell} \mathbf{w}_\ell, \\
	\bm{\Delta}^{(2)}_i &\defeq \frac{1}{N} \sum_{\ell=1}^{i-1} \ip{\mathbf{y}_i,\mathbf{w}_\ell} \mathbf{y}_\ell = \frac{1}{N} \sum_{\ell=1}^{i-1} \ip{\mathbf{y}_i,\mathbf{y}_\ell - \bm{\Delta}_\ell} \mathbf{y}_\ell, \\
	\bm{\Delta}^{(3)}_i &\defeq \frac{1}{N} \sum_{\ell=1}^{i-1} \ip{\mathbf{y}_i,\mathbf{y}_\ell - \bm{\Delta}^{(2)}_\ell} \mathbf{y}_\ell, \\
	\bm{\Delta}^{(4)}_i &\defeq \frac{1}{N} \sum_{\ell=1}^{i-1} \ip{\mathbf{y}_i,\mathbf{y}_\ell - \bm{\widetilde{\Delta}}_\ell} \mathbf{y}_\ell.
\end{align*}
Proposition \ref{prop:deltadeltatilde} will follow from the following lemma (the constants $10$ below 
do not matter; they simply guarantee that we can use a few triangle inequalities). 

\begin{lem}
\label{lem:deltasdeltas}
Fix $\delta > 0$. If $k = \OO(N^{1/2-\delta})$, then  for every $D > 0$ there exists $C_D > 0$ with 
\begin{align}
	\P\left(\max_{i=1}^k \max_{j=1}^N \abs{\Delta_{ij} - \Delta^{(1)}_{ij}} \geq \frac{1}{10N^{1/2}} \right) &\leq C_D N^{-D}, \label{eqn:deltadelta1-prob} \\
	\P\left(\max_{i=1}^k \max_{j=1}^N \abs{\Delta^{(1)}_{ij} - \Delta^{(2)}_{ij}} \geq \frac{1}{10N^{1/2}} \right) &\leq C_D N^{-D}, \label{eqn:delta1delta2-prob} \\
	\P\left(\max_{i=1}^k \max_{j=1}^N \abs{\Delta^{(2)}_{ij} - \Delta^{(3)}_{ij}} \geq \frac{1}{10N^{1/2}} \right) &\leq C_D N^{-D}, \label{eqn:delta2delta3-prob} \\
	\P\left(\max_{i=1}^k \max_{j=1}^N \abs{\Delta^{(3)}_{ij} - \Delta^{(4)}_{ij}} \geq \frac{1}{10N^{1/2}} \right) &\leq C_D N^{-D}, \label{eqn:delta3delta4-prob} \\
	\P\left(\max_{i=1}^k \max_{j=1}^N \abs{\Delta^{(4)}_{ij} - \widetilde{\Delta}_{ij}} \geq \frac{1}{10N^{1/2}} \right) &\leq C_D N^{-D}. \label{eqn:delta4deltatilde-prob}
\end{align}
\end{lem}

The estimates \eqref{eqn:deltadelta1-prob}, \eqref{eqn:delta1delta2-prob}, \eqref{eqn:delta2delta3-prob}, and \eqref{eqn:delta3delta4-prob} are easier; the technically delicate one is \eqref{eqn:delta4deltatilde-prob}, where a full diagrammatic expansion is used. The proofs of these estimates share the same starting steps, which we now describe in the case of \eqref{eqn:deltadelta1-prob}. One can rewrite
\begin{align*}
	\bm{\Delta}_i &= \left( \sum_{\ell=1}^{i-1} \bm{\gamma}_\ell \bm{\gamma}_\ell^T \right) \mathbf{y}_i \eqdef M^{(i)} \mathbf{y}_i, \\
	\bm{\Delta}^{(1)}_i &= \left( \frac{1}{N} \sum_{\ell=1}^{i-1} \mathbf{w}_\ell \mathbf{w}_\ell^T \right) \mathbf{y}_i \eqdef M^{(1,i)} \mathbf{y}_i,
\end{align*}
and then observe that the matrices $M^{(i)}$ and $M^{(1,i)}$ depend only on $\mathbf{y}_1, \ldots, \mathbf{y}_{i-1}$. In particular they are independent of $\mathbf{y}_i$, so that for each $i$ and $j$, we have the equality in distribution
\[
	\Delta_{ij} - \Delta^{(1)}_{ij} \overset{d}{=} \sigma^{(0,1)}_{ij} Z,
\]
where $Z$ is a standard Gaussian variable independent of everything else, and 
\[
	(\sigma_{ij}^{(0,1)})^2 \defeq ((M^{(i)} - M^{(1,i)})(M^{(i)} - M^{(1,i)})^T)_{jj}.
\]
With this information, the union bound gives
\begin{equation}
\label{eqn:delta_replacement_starting_point}
\begin{split}
	\P\left( \max_{i=1}^k \max_{j=1}^N \abs{\Delta_{ij} - \Delta^{(1)}_{ij}} \geq \frac{1}{10N^{1/2}}\right) &\leq kN \max_{i=1}^k \max_{j=1}^N \P\left(\abs{\Delta_{ij} - \Delta^{(1)}_{ij}} \geq \frac{1}{10N^{1/2}} \right) \\
	&= kN \max_{i=1}^k \max_{j=1}^N \P\left(N(\sigma_{ij}^{(0,1)})^2 Z^2 \geq \frac{1}{100} \right).
\end{split}
\end{equation}
Thus it remains only to understand $N(\sigma_{ij}^{(0,1)})^2$, which we will see is $\oo(1)$. For the proofs of the other estimates, we write
\begin{align*}
	\bm{\Delta}^{(2)}_i = M^{(2,i)}\mathbf{y}_i \qquad &\text{where} \qquad M^{(2,i)} \defeq \frac{1}{N} \sum_{\ell=1}^{i-1} \mathbf{y}_\ell \mathbf{w}_\ell^T, \\
	\bm{\Delta}^{(3)}_i = M^{(3,i)}\mathbf{y}_i \qquad &\text{where} \qquad M^{(3,i)} \defeq \frac{1}{N} \sum_{\ell=1}^{i-1} \mathbf{y}_\ell (\mathbf{y}_\ell - \bm{\Delta}_\ell^{(2)})^T, \\
	\bm{\Delta}^{(4)}_i = M^{(4,i)}\mathbf{y}_i \qquad &\text{where} \qquad M^{(4,i)} \defeq \frac{1}{N} \sum_{\ell=1}^{i-1} \mathbf{y}_\ell (\mathbf{y}_\ell - \widetilde{\bm{\Delta}}_\ell)^T, \\
	\widetilde{\bm{\Delta}}_i = \widetilde{M^{(i)}} \mathbf{y}_i \qquad &\text{where} \qquad \widetilde{M^{(i)}} \defeq \frac{1}{N} \sum_{\ell=1}^{i-1} \mathbf{y}_\ell \mathbf{y}_\ell^T, \\
\end{align*}
so that 
\begin{align*}
	\Delta^{(1)}_{ij} - \Delta^{(2)}_{ij} \overset{d}{=} \sigma_{ij}^{(1,2)} Z, \qquad &\text{where} \qquad (\sigma_{ij}^{(1,2)})^2 \defeq ((M^{(1,i)} - M^{(2,i)})(M^{(1,i)} - M^{(2,i)})^T)_{jj}, \\
	\Delta^{(2)}_{ij} - \Delta^{(3)}_{ij} \overset{d}{=} \sigma_{ij}^{(2,3)} Z, \qquad &\text{where} \qquad (\sigma_{ij}^{(2,3)})^2 \defeq ((M^{(2,i)} - M^{(3,i)})(M^{(2,i)} - M^{(3,i)})^T)_{jj}, \\
	\Delta^{(3)}_{ij} - \Delta^{(4)}_{ij} \overset{d}{=} \sigma_{ij}^{(3,4)} Z, \qquad &\text{where} \qquad (\sigma_{ij}^{(3,4)})^2 \defeq ((M^{(3,i)} - M^{(4,i)})(M^{(3,i)} - M^{(4,i)})^T)_{jj}, \\
	\Delta^{(4)}_{ij} - \widetilde{\Delta}_{ij} \overset{d}{=} \sigma_{ij}^{(4,\infty)} Z, \qquad &\text{where} \qquad (\sigma_{ij}^{(4,\infty)})^2 \defeq ((M^{(4,i)} - \widetilde{M^{(i)}})(M^{(4,i)} - \widetilde{M^{(i)}})^T)_{jj}.
\end{align*}
Here the $Z$'s denote standard Gaussian variables, not the same ones from line to line. In the proof of each estimate, we will jump straightaway into estimating the relevant $N\sigma_{ij}^2$, showing that it is $\oo(1)$. Sometimes this will be in the sense of high moments (typically order $1/\delta$ moments, where $k = \OO(N^{1/2-\delta})$), but other times we will just bound it on a certain
high probability ``good'' event which we now define, using the error variables
\[
	L_i \defeq \abs{\sqrt{\frac{N}{\|\mathbf{w}_i\|^2}} - 1} \qquad \text{for } i = 1, \ldots, N.
\]

\begin{lem}
\label{lem:delta_replacement_good_event}
Fix $\delta > 0$, $k = \OO(N^{1/2-\delta})$, and consider the event
\begin{equation}
\label{eqn:delta_replacement_good_event}
\begin{split}
	\mc{E}_{k,N}: =& \, \left\{ \max_{a,b=1}^N \abs{y_{ab}} \leq \log N \right\} \cap \left\{ \max_{a,b=1}^N \frac{\abs{\ip{\mathbf{y}_a,\mathbf{y}_b}}}{(\sqrt{N})^{1+\delta_{ab}}} \leq \log N \right\} \cap \left\{\max_{a=1}^k L_a \leq \frac{\log N}{\sqrt{N}} \right\} \\
	&\cap \left\{\max_{a,b=1}^N \abs{w_{ab}} \leq 2\log N \right\} \cap \left\{ \max_{a=1}^N \|\mathbf{w}_a\|^2 \leq N \log N\right\} \cap \left\{\max_{a=1}^k \max_{b=1}^N \abs{\Delta_{ab}} \leq (\log N)^2 \sqrt{\frac{k}{N}} \right\}.
\end{split}
\end{equation}
Then the complement of $\mc{E}_{k,N}$ has very small probability: For some $c > 0$,
\[
	\P((\mc{E}_{k,N})^c) = \OO(e^{-c(\log N)^2}).
\]
\end{lem}

The proof of this will be given after the following proposition, which collects some estimates of Jiang \cite{Jia2005}, and uses the error variable
\[
	\epsilon_N(k) \defeq \max_{1 \leq i \leq k, 1 \leq j \leq N} \abs{\sqrt{N}\gamma_{ij} - y_{ij}}.
\]
\begin{prop}
\label{prop:jiang_bounds}
Fix some sequence $(k=k_N)_{N=1}^\infty$ of positive integers satisfying $k = \OO(N^{1/2})$. 
Then for some small $c > 0$ we have
\begin{align}
	\P\left(\max_{i=1}^k \vertiii{\bm{\Delta}_i} > (\log N)^2 \sqrt{\frac{k}{N}} \right) &= \OO(e^{-c(\log N)^2}), \label{eqn:jiang_delta_bound} \\
	\P\left(\max_{i=1}^k L_i > \frac{\log N}{\sqrt{N}}\right) &= \OO(e^{-c(\log N)^2}), \label{eqn:jiang_L_bound} \\
	\P\left(\epsilon_N(k) > (\log N)^3 \sqrt{\frac{k}{N}} \right) &= \OO(e^{-c(\log N)^2}). \label{eqn:jiang_epsilon_bound}
\end{align}
\end{prop}
\begin{proof}
\begin{itemize}
\item[\eqref{eqn:jiang_delta_bound}] From \cite[Lemma 3.5]{Jia2005}, we have
\begin{equation}
\label{eqn:jiang_complete_delta_bound}
	\P\left(\max_{i=1}^k \vertiii{\bm{\Delta}_i} \geq t\right) \leq \frac{3kN}{t} \left(1+\frac{t^2}{3(k+t\sqrt{N})}\right)^{-N/2}
\end{equation}
for any $t > 0$. From our choice $t = (\log N)^2(k/N)^{1/2}$ in \eqref{eqn:jiang_delta_bound}, this is at most 
\[
	CN^2\left(1+\frac{(\log N)^4(k/N)}{3(k+k(\log N)^2)}\right)^{-N/2} \leq CN^2 \left(1+\frac{(\log N)^2}{CN}\right)^{-N/2} = \OO(e^{-c(\log N)^2})
\]
for some $C > 0$.
\item[\eqref{eqn:jiang_L_bound}] From \cite[Lemma 3.6]{Jia2005}, we have
\[
	\P\left(\max_{i=1}^k L_i \geq r\right) \leq 4ke^{-Nr^2/16}
\]
whenever $r \in (0,1/4)$, $k \leq (r/2)N$. We take $r=N^{-1/2}(\log N)$.
\item[\eqref{eqn:jiang_epsilon_bound}] From \cite[Theorem 5]{Jia2005}, we have\footnote{
In writing Theorem 5, Jiang writes the last denominator as $3(k+\sqrt{N})$ instead of $3(k+t\sqrt{N})$, which suffices for his purposes. But indeed he proves the sharper result with $3(k+t\sqrt{N})$, as can be seen by comparing his proof of this result with the statement of his Lemma 3.5.
}
\begin{equation}
\label{eqn:jiang_epsilonNk_bound}
	\P(\epsilon_N(k) \geq rs+2t) \leq 4ke^{-Nr^2/16} + 3kN\left(\frac{1}{s}e^{-s^2/2} + \frac{1}{t}\left(1+\frac{t^2}{3(k+t\sqrt{N})}\right)^{-N/2}\right)
\end{equation}
whenever $r \in (0,1/4)$, $s > 0$, $t > 0$, and $k \leq (r/2)N$. 
We take, say, $r = \frac{\log N}{\sqrt{N}}$, $s = \log N$, and 
$t = (\log N)^2 \sqrt{k/N}$.
\end{itemize}
\end{proof}

\begin{proof}[Proof of Lemma \ref{lem:delta_replacement_good_event}]
The $L$ and $\Delta$ terms follow immediately from Proposition \ref{prop:jiang_bounds}. The two events in \eqref{eqn:delta_replacement_good_event} involving $w$'s are actually included in the other four events: this is true for the vectors since $\mathbf{w}_a$ is a projection of $\mathbf{y}_a$ and hence $\|\mathbf{w}_a\| \leq \|\mathbf{y}_a\|$, and true for the scalars since $\abs{w_{ab}} \leq \abs{y_{ab}} + \abs{\Delta_{ab}}$, both deterministically. So only the basic Gaussian events remain to be proven, which are routine: Of course
\[
	\P\left(\max_{a,b=1}^N \abs{y_{ab}} \geq \log N\right) \leq N^2\P(\abs{y_{11}} \geq \log N) \leq N^2\exp(-(\log N)^2/2),
\]
but also
\begin{align*}
	\P\left(\max_{\substack{a,b=1 \\ a \neq b}}^N \abs{\ip{\mathbf{y}_a,\mathbf{y}_b}} \geq N^{1/2} \log N\right) &\leq 2N^2 \P\left( \frac{\sum_{b=1}^N y_{1b}y_{2b}}{N} \geq \frac{\log N}{N^{1/2}} \right), \\
	\P\left(\max_{a=1}^N \|\mathbf{y}_a\|^2 \geq N \log N \right) &\leq N\P\left(\frac{\sum_{b=1}^N y_{1b}^2}{N} \geq \log N \right).
\end{align*}
The two probabilities on the right-hand sides of the above display are each superpolynomially small (precisely, the first is at most $2\exp(-(\log N)^2)$, and the second is at most $2\exp(-N\log N/2)$). Indeed, they describe large deviations events for the sample averages of $N$ independent variables (distributed as the product of two independent Gaussians in the first case, and as $\chi^2_1$ in the second case), which fall under Cram\'er's theorem, the upper bound of which is valid at finite $N$. The rate function in the first case is 
\[
I(x) = \frac{1}{2}\Big(-1+\sqrt{1+4x^2}+\log\big(\frac{-1+\sqrt{1+4x^2}}{2x^2}\big)\Big) \geq x^2,
\] which is convex, vanishes at zero, and is at least $N^{-1} (\log N)^2$ when evaluated at $x = N^{-1/2}\log N$.
 The rate function in the second case is $I(x) = \frac{1}{2}(x-1-\log x)$, evaluated at $x = \log N$.
\end{proof}


\subsection{Proofs of the simpler bounds \eqref{eqn:deltadelta1-prob}, \eqref{eqn:delta1delta2-prob}, \eqref{eqn:delta2delta3-prob}, and \eqref{eqn:delta3delta4-prob}}\label{sec:easy}

\begin{proof}[Proof of \eqref{eqn:deltadelta1-prob}]
Since
\[
	\gamma_{\ell j} \gamma_{\ell q} - \frac{1}{N} w_{\ell j} w_{\ell q} = w_{\ell j} w_{\ell q} \left(\frac{1}{\|\mathbf{w}_\ell\|^2} - \frac{1}{N}\right),
\]
and the $\mathbf{w}_i$'s are orthogonal, we compute
\begin{align*}
	N(\sigma_{ij}^{(0,1)})^2 &= N((M^{(i)} - M^{(1,i)})(M^{(i)} - M^{(1,i)})^T)_{jj} = N \sum_{q=1}^N ((M^{(i)} - M^{(1,i)})_{jq})^2 \\
	&= N \sum_{q=1}^N \sum_{\ell=1}^{i-1} \sum_{\ell'=1}^{i-1} w_{\ell j} w_{\ell q} w_{\ell' j} w_{\ell' q} \left( \frac{1}{\|\mathbf{w}_\ell\|^2} - \frac{1}{N} \right) \left( \frac{1}{\|\mathbf{w}_\ell'\|^2} - \frac{1}{N} \right) \\
	&= N \sum_{\ell=1}^{i-1} w_{\ell j}^2 \|\mathbf{w}_\ell\|^2 \left(\frac{1}{\|\mathbf{w}_\ell\|^2} - \frac{1}{N} \right)^2.
\end{align*}
Notice that 
\[
	\abs{\frac{1}{\|\mathbf{w}_{\ell}\|^2} - \frac{1}{N}} = \frac{1}{N} \abs{\frac{N}{\|\mathbf{w}_\ell\|^2} - 1} = \frac{1}{N} \abs{ \sqrt{\frac{N}{\|\mathbf{w}_\ell\|^2}} - 1 } \abs{ \sqrt{\frac{N}{\|\mathbf{w}_\ell\|^2}} + 1 } \leq \frac{L_\ell(L_\ell+2)}{N},
\]
so that on the good event $\mc{E}_{k,N}$ from \eqref{eqn:delta_replacement_good_event}, we have
\[
	\max_{i=1}^k \max_{j=1}^N (N(\sigma_{ij}^{(0,1)})^2) \leq Nk(2\log N)^2 (\log N) N \left(\frac{3\log N}{N^{3/2}}\right)^2 \leq \frac{100k (\log N)^5}{N} \leq N^{-1/2},
\]
so that
\begin{equation}
\label{eqn:sigma_on_a_good_event}
	\max_{i=1}^k \max_{j=1}^N \P\left(N(\sigma_{ij}^{(0,1)})^2 Z^2 \geq \frac{1}{100}\right) \leq \P(N^{-1/2} Z^2 \geq 1/100) + \P((\mc{E}_{k,N})^c)).
\end{equation}
The first term on the right-hand side is $\OO(\exp(-c\sqrt{N}))$, and the second term is $\OO(e^{-c(\log N)^2})$ from Lemma \ref{lem:delta_replacement_good_event}; combined with the general estimate \eqref{eqn:delta_replacement_starting_point}, this finishes the proof (and gives a better estimate than we need).
\end{proof}

\begin{proof}[Proof of \eqref{eqn:delta1delta2-prob}]
We compute
\begin{align*}
	N(\sigma_{ij}^{(1,2)})^2 &= N\sum_{q=1}^N ((M^{(1,i)}-M^{(2,i)})_{jq})^2 = \frac{1}{N} \sum_{q=1}^N \left(\sum_{\ell=1}^{i-1} (w_{\ell j} - y_{\ell j})w_{\ell q} \right)^2 \\
	&= \frac{1}{N} \sum_{\ell=1}^{i-1} \sum_{\ell'=1}^{i-1} \sum_{q=1}^N \Delta_{\ell j} \Delta_{\ell' j} w_{\ell q} w_{\ell' q} = \frac{1}{N} \sum_{\ell=1}^{i-1} \sum_{\ell'=1}^{i-1} \Delta_{\ell j} \Delta_{\ell' j} \ip{\mathbf{w}_{\ell},\mathbf{w}_{\ell'}} = \frac{1}{N} \sum_{\ell=1}^{i-1} \Delta_{\ell j}^2 \|\mathbf{w}_\ell\|^2.
\end{align*}
On the good event $\mc{E}_{k,N}$ from \eqref{eqn:delta_replacement_good_event}, we have
\[
	\max_{i=1}^k \max_{j=1}^N N(\sigma_{ij}^{(1,2)})^2 \leq N^{-1} k (N^{-1/4})^2 N \log N \leq \left(\frac{k}{\sqrt{N}}\right) \log N.
\]
The rest of the proof is as in \eqref{eqn:sigma_on_a_good_event}.
\end{proof}

\begin{proof}[Proof of \eqref{eqn:delta2delta3-prob}]
We compute
\begin{align*}
	N(\sigma_{ij}^{(2,3)})^2 &= N\sum_{q=1}^N ((M^{(2,i)}-M^{(3,i)})_{jk})^2 = \frac{1}{N} \sum_{q=1}^N \left(\sum_{\ell=1}^{i-1} y_{\ell j} (\Delta_{\ell q} - \Delta^{(2)}_{\ell q}) \right)^2 \\
	&= \frac{1}{N} \sum_{\ell=1}^{i-1} \sum_{\ell'=1}^{i-1} \sum_{q=1}^N y_{\ell j} y_{\ell' j} (\Delta_{\ell q} - \Delta^{(2)}_{\ell q}) (\Delta_{\ell' q} - \Delta^{(2)}_{\ell' q}) \\
	&= \frac{1}{N} \sum_{\ell=1}^{i-1} \sum_{\ell'=1}^{i-1} y_{\ell j} y_{\ell' j} \ip{\bm{\Delta}_{\ell} - \bm{\Delta}^{(2)}_{\ell},\bm{\Delta}_{\ell'} - \bm{\Delta}^{(2)}_{\ell'}}.
\end{align*}
Consider the good event
\[
	\mc{E}^{(0,2)}_{k,N} := \left\{\max_{i=1}^k \max_{j=1}^N \abs{\Delta_{ij} - \Delta^{(2)}_{ij}} \leq \frac{1}{N^{1/2}}\right\}.
\]
From \eqref{eqn:deltadelta1-prob} and \eqref{eqn:delta1delta2-prob} we see that $P((\mc{E}^{(0,2)}_{k,N})^c) = \OO_D(N^{-D})$ for any $D$. On $\mc{E}^{(0,2)}_{k,N}$ we have 
\[
	\max_{\ell,\ell'=1}^k \abs{ \ip{\bm{\Delta}_{\ell} - \bm{\Delta}^{(2)}_{\ell},\bm{\Delta}_{\ell'} - \bm{\Delta}^{(2)}_{\ell'}} } \leq 1,
\]
so that, on $\mc{E}^{(0,2)}_{k,N} \cap \mc{E}_{k,N}$, we have 
\[
	\max_{i=1}^k \max_{j=1}^N N(\sigma_{ij}^{(2,3)})^2 \leq \left(\frac{k}{\sqrt{N}}\right)^2 (\log N)^2.
\]
The rest of the proof is as in \eqref{eqn:sigma_on_a_good_event}.
\end{proof}

The proof of \eqref{eqn:delta3delta4-prob} relies on the following result.

\begin{lem}
\label{lem:delta2deltatilde}
If $\delta > 0$, $k = \OO(N^{1/2-\delta})$, and $p \geq 1$ is an integer, then
\[
	\max_{\ell, \ell'=1}^{k} \E\left[\ip{\bm{\Delta}^{(2)}_\ell - \widetilde{\bm{\Delta}}_\ell, \bm{\Delta}^{(2)}_{\ell'} - \widetilde{\bm{\Delta}}_{\ell'}}^{2p}\right] \leq C_p \left(\frac{k}{\sqrt{N}}\right)^{4p} \le C_p N^{-4p\delta}.
\]
\end{lem}

\begin{proof}[Proof of Lemma \ref{lem:delta2deltatilde}]
First we compute
\begin{align*}
	 \ip{\bm{\Delta}^{(2)}_{\ell} - \widetilde{\bm{\Delta}}_{\ell},\bm{\Delta}^{(2)}_{\ell'} - \widetilde{\bm{\Delta}}_{\ell'}} &= \frac{1}{N^2} \ip{ \sum_{m=1}^{\ell-1} \ip{\mathbf{y}_\ell,-\bm{\Delta}_m} \mathbf{y}_m, \sum_{m'=1}^{\ell'-1} \ip{\mathbf{y}_{\ell'}, -\bm{\Delta}_{m'}} \mathbf{y}_{m'}} \\
	 &= \frac{1}{N^2} \sum_{m=1}^{\ell-1} \sum_{m'=1}^{\ell'-1} \ip{\mathbf{y}_\ell,\bm{\Delta}_m} \ip{\mathbf{y}_{\ell'}, \bm{\Delta}_{m'}} \ip{\mathbf{y}_m, \mathbf{y}_{m'}} .
\end{align*}
To estimate this, we need to bound $\ip{\mathbf{y}_\ell, \bm{\Delta}_m}$ in high moments, when $m < \ell$. For any fixed $p$ we have
\begin{align*}
	\E[\ip{\mathbf{y}_\ell, \bm{\Delta}_m}^{2p}] = \E\left[\left(\sum_{s=1}^{m-1} \ip{\mathbf{y}_m,\bm{\gamma}_s} \ip{\mathbf{y}_\ell,\bm{\gamma}_s}\right)^{2p}\right] &= \sum_{s_1, \ldots, s_{2p} = 1}^{m-1} \E\left[ \prod_{a=1}^{2p} (\ip{\mathbf{y}_m, \bm{\gamma}_{s_a}} \ip{\mathbf{y}_{\ell}, \bm{\gamma}_{s_a}}) \right]
\end{align*}
and since $\ell > m > \max_{a=1}^{2p} s_a$, we can do the Wick theorem separately on $m$ and $\ell$. Write $P_{2p}$ for the set of partitions $\pi$ of $\{1, \ldots, 2p\}$ into $p$ pairs, and notate $\pi = (b_1, \ldots, b_p)$ with blocks $b_i = (s_1(b_i),s_2(b_i))$. Since $\ip{\bm{\gamma}_s,\bm{\gamma}_{s'}} = \delta_{ss'}$, we find
\[
	\E[\ip{\mathbf{y}_\ell, \bm{\Delta}_m}^{2p}] = \sum_{s_1, \ldots, s_{2p}=1}^{m-1} \left( \sum_{\pi = (b_1,\ldots,b_p) \in P_{2p}} \prod_{j=1}^p \delta_{s_1(b_j),s_2(b_j)} \right)^2 \eqdef \sum_{s_1, \ldots, s_{2p}=1}^{m-1} F_{s_1, \ldots, s_{2p}}.
\]
Notice that $F = F_{s_1, \ldots, s_{2p}}$ depends only on $p$ and on the cardinality of the set $\{s_1, \ldots, s_{2p}\}$, but not on the actual values of the $s_i$ themselves. Since this cardinality takes values in $\{1, \ldots, 2p\}$, there is some upper bound $F \leq C_p$ regardless of $m$ and $\ell$. Furthermore, $F_{s_1, \ldots, s_{2p}}$ vanishes unless this cardinality is at most $p$, i.e., it vanishes if any $s_i$ is alone. Thus $\E[\ip{\mathbf{y}_\ell,\bm{\Delta}_m}^{2p}]$ is at most $C_p$ times the number of ordered tuples $(s_1, \ldots, s_{2p})$ such that each $s_i$ takes values in $\{1, \ldots, m-1\}$ and such that the set $\mc{S} = \{s_1, \ldots, s_{2p}\}$ has cardinality at most $p$ (forgetting the requirement that no $s_i$ be alone, to get an upper bound). We can pick such a tuple in the following way: First choose $p$ distinct elements of $\{1, \ldots, m-1\}$, and fix this as an alphabet. Then select one element of this alphabet, $2p$ times: the first is called $s_1$, the second called $s_2$, and so on. This procedure produces $\binom{m-1}{p} p^{2p}$ tuples, which overcounts (since, e.g. when $p = 2$, the ordered tuple of all ones can be created both by the alphabet $\{1,2\}$ and by the alphabet $\{1, 3\}$), so
\[
	\E[\ip{\mathbf{y}_\ell,\bm{\Delta}_m}^{2p}] \leq C_p \binom{m-1}{p} \leq C_p m^p.
\]
Standard Gaussian estimates show
\begin{align*}
	\E[\|\mathbf{y}_m\|^{2p}] &\leq C_p N^p
\end{align*}
and, for $m \neq m'$,
\begin{align*}
	\E[\ip{\mathbf{y}_m,\mathbf{y}_{m'}}^{2p}] = \sum_{\pi = (b_1, \ldots, b_p) \in P_{2p}} \E[\|\mathbf{y}_m\|^{p}] \leq C_p N^p.
\end{align*}
Thus
\begin{align*}
	&\E\left[ \ip{\bm{\Delta}^{(2)}_{\ell} - \widetilde{\bm{\Delta}}_{\ell},\bm{\Delta}^{(2)}_{\ell'} - \widetilde{\bm{\Delta}}_{\ell'}}^{2p} \right] \\
	&= N^{-4p} \sum_{m_1, \ldots, m_{2p} = 1}^{\ell-1} \sum_{m'_1, \ldots, m'_{2p} = 1}^{\ell'-1} \E\left[\prod_{a=1}^{2p} \ip{\mathbf{y}_\ell,\bm{\Delta}_{m_a}} \ip{\mathbf{y}_{\ell'}, \bm{\Delta}_{m'_a}} \ip{\mathbf{y}_{m_a}, \mathbf{y}_{m'_a}} \right] \\
	&\leq \left( N^{-2} \sum_{m=1}^{\ell-1} \sum_{m'=1}^{\ell'-1} \E[\ip{\mathbf{y}_\ell, \bm{\Delta}_m}^{6p}]^{\frac{1}{6p}} \E[\ip{\mathbf{y}_{\ell'},\bm{\Delta}_{m'}}^{6p}]^{\frac{1}{6p}} \E[\ip{\mathbf{y}_m,\mathbf{y}_{m'}}^{6p}]^{\frac{1}{6p}} \right)^{2p} \\
	&\leq \left( N^{-2} \left(\sum_{m,m'=1}^{\min(\ell,\ell')-1} \delta_{mm'} + \sum_{m \neq m'}\right) \E[\ip{\mathbf{y}_\ell, \bm{\Delta}_m}^{6p}]^{\frac{1}{6p}} \E[\ip{\mathbf{y}_{\ell'},\bm{\Delta}_{m'}}^{6p}]^{\frac{1}{6p}} \E[\ip{\mathbf{y}_m,\mathbf{y}_{m'}}^{6p}]^{\frac{1}{6p}} \right)^{2p} \\
	&\leq C_p \left( N^{-2} (k k^{1/2} k^{1/2} N + k^2 k^{1/2} k^{1/2} N^{1/2}) \right)^{2p} = C_p \left(\frac{k}{\sqrt{N}}\right)^{4p}.
\end{align*}
\end{proof}

\begin{proof}[Proof of \eqref{eqn:delta3delta4-prob}]
As in the proof of \eqref{eqn:delta2delta3-prob}, we compute
\begin{align*}
	N(\sigma_{ij}^{(3,4)})^2 &= \frac{1}{N} \sum_{\ell=1}^{i-1} \sum_{\ell'=1}^{i-1} y_{\ell j} y_{\ell' j} \ip{\bm{\Delta}^{(2)}_{\ell} - \widetilde{\bm{\Delta}}_{\ell},\bm{\Delta}^{(2)}_{\ell'} - \widetilde{\bm{\Delta}}_{\ell'}}.
\end{align*}
so that, for integer $p \geq 1$, using (a weaker consequence of) Lemma \ref{lem:delta2deltatilde},
\begin{align*}
	\max_{i=1}^k \max_{j=1}^N \E[(N(\sigma_{ij}^{(3,4)})^2)^p] &\leq \left( N^{-1} \sum_{\ell,\ell'=1}^{i-1} \E[\abs{y_{\ell j}}^{4p}]^{\frac{1}{4p}} \E[\abs{y_{\ell' j}}^{4p}]^{\frac{1}{4p}} \E\left[\ip{\bm{\Delta}^{(2)}_\ell - \widetilde{\bm{\Delta}}_\ell, \bm{\Delta}^{(2)}_{\ell'} - \widetilde{\bm{\Delta}}_{\ell'}}^{2p}\right]^{\frac{1}{2p}} \right)^p \leq C_p \left( \frac{k^2}{N} \right)^p.
\end{align*}
From (the analogue of) \eqref{eqn:delta_replacement_starting_point} and Markov's inequality, we find
\begin{align*}
	\P\left( \max_{i=1}^k \max_{j=1}^N \abs{\Delta^{(3)}_{ij} - \Delta^{(4)}_{ij}} \geq \frac{1}{10N^{1/2}}\right) &\leq kN \max_{i=1}^k \max_{j=1}^N \P\left(N(\sigma_{ij}^{(3,4)})^2 Z^2 \geq \frac{1}{100} \right) \\
	&\leq kN (100)^p \E[Z^{2p}] \max_{i=1}^k \max_{j=1}^N \E[(N(\sigma_{ij}^{(3,4)})^2)^p] \leq C_p kN \left(\frac{k^2}{N}\right)^p
\end{align*}
which suffices by taking $p$ large enough.
\end{proof}


\subsection{Proof of  \eqref{eqn:delta4deltatilde-prob} via graphical expansion.}\label{sec:hard}\

The proof of \eqref{eqn:delta4deltatilde-prob} relies on the following lemma, which we prove after.
\begin{lem}
\label{lem:graphs}
There exist constants $C_p$ such that, for each integer $p \geq 1$, we have
\[
	\max_{i=1}^k \max_{j=1}^N \E[(N(\sigma_{ij}^{(4,\infty)})^2)^p] \leq C_p \left(\frac{k}{\sqrt{N}}\right)^p.
\]
\end{lem}

\begin{proof}[Proof of \eqref{eqn:delta4deltatilde-prob}]
From (the analogue of) \eqref{eqn:delta_replacement_starting_point} and Markov's inequality, we find
\begin{align*}
	\P\left( \max_{i=1}^k \max_{j=1}^N \abs{\Delta^{(4)}_{ij} - \widetilde{\Delta}_{ij}} \geq \frac{1}{10N^{1/2}}\right) &\leq kN \max_{i=1}^k \max_{j=1}^N \P\left(N(\sigma_{ij}^{(4,\infty)})^2 Z^2 \geq \frac{1}{100} \right) \\
	&\leq kN (100)^p \E[Z^{2p}] \max_{i=1}^k \max_{j=1}^N \E[(N(\sigma_{ij}^{(4,\infty)})^2)^p] \leq C_p kN \left(\frac{k}{\sqrt{N}}\right)^p,
\end{align*}
which suffices by taking $p$ large enough.
\end{proof}

\begin{proof}[Proof of Lemma \ref{lem:graphs}]
We start by computing
\begin{align*}
	N(\sigma_{ij}^{(4,\infty)})^2 &= N\sum_{k=1}^N ((M^{(4,i)}-\widetilde{M^{(i)}})_{jk})^2 = \frac{1}{N} \sum_{k=1}^N \left(\sum_{\ell=1}^{i-1} y_{\ell j} (- \widetilde{\Delta}_{\ell k}) \right)^2 \\
	&= \frac{1}{N} \sum_{\ell=1}^{i-1} \sum_{\ell'=1}^{i-1} \sum_{k=1}^N y_{\ell j} y_{\ell' j} \widetilde{\Delta}_{\ell k} \widetilde{\Delta}_{\ell' k} = \frac{1}{N} \sum_{\ell=1}^{i-1} \sum_{\ell'=1}^{i-1} y_{\ell j} y_{\ell' j} \ip{\widetilde{\bm{\Delta}}_{\ell},\widetilde{\bm{\Delta}}_{\ell'}} 
\end{align*}
and
\begin{align*}
	 \ip{\widetilde{\bm{\Delta}}_{\ell}, \widetilde{\bm{\Delta}}_{\ell'}} &= \frac{1}{N^2} \ip{ \sum_{m=1}^{\ell-1} \ip{\mathbf{y}_\ell,\mathbf{y}_m} \mathbf{y}_m, \sum_{m'=1}^{\ell'-1} \ip{\mathbf{y}_{\ell'}, \mathbf{y}_{m'}} \mathbf{y}_{m'}} = \frac{1}{N^2} \sum_{m=1}^{\ell-1} \sum_{m'=1}^{\ell'-1} \ip{\mathbf{y}_\ell,\mathbf{y}_m} \ip{\mathbf{y}_{\ell'}, \mathbf{y}_{m'}} \ip{\mathbf{y}_m, \mathbf{y}_{m'}},
\end{align*}
so that
\begin{equation}
\label{eqn:sigma4infty}
\begin{split}
	&\E[(N(\sigma_{ij}^{(4,\infty)})^2)^p] \\
	&= N^{-3p} \sum_{\ell_1, \ell'_1, \ldots, \ell_p, \ell'_p=1}^{i-1} \sum_{m_1=1}^{\ell_1-1} \sum_{m'_1 = 1}^{\ell'_1-1} \cdots \sum_{m_p=1}^{\ell_p-1} \sum_{m'_p=1}^{\ell'_p-1} \E\left[ \prod_{a=1}^p y_{\ell_a j} y_{\ell'_a j} \ip{\mathbf{y}_{\ell_a}, \mathbf{y}_{m_a}} \ip{\mathbf{y}_{\ell'_a}, \mathbf{y}_{m'_a}} \ip{\mathbf{y}_{m_a}, \mathbf{y}_{m'_a}} \right].
\end{split}
\end{equation}
Notice that everything is now written in terms of the $\mathbf{y}$ variables, which are independent Gaussians, so in computing this expectation via the Wick theorem, the only question is which indices coincide. Write $A$ for an \emph{assignment} of values to these indices, formally a function from $\{1, \ldots, 4p\}$ to $\{1, \ldots, i-1\}$, interpreted such that, for $k = 1, \ldots, p$,  $A(2k-1)$ is the value of $\ell_k$, $A(2k)$ is the value of $\ell'_k$, $A(2p+2k-1)$ is the value of $m_k$, and $A(2p+2k)$ is the value of $m'_k$. Not all functions from $\{1, \ldots, 4p\}$ to $\{1, \ldots, i-1\}$ are valid assignments; a \emph{legal} assignment is one in which $m_k < \ell_k$ and $m'_k < \ell'_k$ for each $k$, and legal assignments are the only kind that appear in \eqref{eqn:sigma4infty}. Write $\ms{A}$ for the set of all legal assignments. Then define
\[
	\mathbf{y}_A \defeq \prod_{a=1}^p y_{\ell_a j} y_{\ell'_a j} \ip{\mathbf{y}_{\ell_a}, \mathbf{y}_{m_a}} \ip{\mathbf{y}_{\ell'_a}, \mathbf{y}_{m'_a}} \ip{\mathbf{y}_{m_a}, \mathbf{y}_{m'_a}},
\]
where the indices have the values assigned to them by $A$, so that
\[
	\E[(N(\sigma_{ij}^{(4,\infty)})^2)^p] = N^{-3p} \sum_{A \in \ms{A}} \E\left[ \mathbf{y}_A \right].
\]
Say that a \emph{pattern} $P = P(A)$ is the knowledge of which $\ell$ and $m$ variables coincide in $A$ (formally, for each assignment $A$, create a graph $G_A$ whose vertices carry the variable names ``$\ell_1$,'' ``$\ell'_1$,'' $\ldots$ ``$m_p$,'' ``$m'_p$,'' and such that two vertices are connected in $G_A$ if $A$ assigns them the same value; then a pattern $P$ is an equivalence class of assignments $A$, where the equivalence relation is $A \cong A'$ if $G_A = G_{A'}$). Each assignment $A$ belongs to exactly one pattern, and in fact $\E[\mathbf{y}_A]$ depends only on the pattern $P$ to which $A$ belongs; thus we abuse notation by writing $\E[\mathbf{y}_P]$ instead of $\E[\mathbf{y}_A]$. If $\ms{P}$ denotes the set of all legal patterns, and 
\[
	a(P) \defeq \#\{A \in \ms{A} : P(A) = P\},
\]
then we can write equivalently
\begin{equation}
\label{eqn:pattern_sum}
	\E[(N(\sigma_{ij}^{(4,\infty)})^2)^p] = N^{-3p} \sum_{P \in \ms{P}} a(P) \E\left[ \mathbf{y}_P \right].
\end{equation}
Now our goal is to estimate $a(P)$ and $\abs{\E[\mathbf{y}_P]}$. Notice that $\mathbf{y}_P$
 (technically $\mathbf{y}_A$, where $A$ is a member of the equivalence class represented by $P$) is a product of some terms of the form $y_{aj}$ and some terms of the form $\ip{\mathbf{y}_a,\mathbf{y}_b}$. In principle its mean could be exactly computed for each $P$, 
 by doing all of the integrations, but this is more work than necessary; on the other end of the spectrum, one could do none of the integrations, which roughly means estimating each $\ip{\mathbf{y}_a,\mathbf{y}_b}$ by $N^{(1+\delta_{ab})/2}$ and each $y_{aj}$ by one, but this is not good enough. We need to do some of the integrations, and it turns out that it is sufficient to only do integrations of the following simple type: Whenever $a \neq b \neq c$ and $X$ is independent of $\mathbf{y}_b$, we have 
\begin{equation}
\label{eqn:star_rules}
\begin{split}
	\E_{\mathbf{y}_b}[\ip{\mathbf{y}_a,\mathbf{y}_b} \ip{\mathbf{y}_b,\mathbf{y}_c} X] &= \ip{\mathbf{y}_a,\mathbf{y}_c} X, \\
	\E_{\mathbf{y}_b}[y_{bj} \ip{\mathbf{y}_a,\mathbf{y}_b} X] &= y_{aj} X.
\end{split}
\end{equation}
That is, in the estimation, we pay careful attention \emph{only} to indices which appear exactly twice.

At the end, we will carry out a procedure we call a \emph{straight estimate}, which we now explain. Consider expectations of the form
\begin{equation}
\label{eqn:v_form}
	\E\left[\left( \prod_{a=1}^\alpha v_a\right) \left( \prod_{b=1}^\beta \mathbf{v}_b \right) \left( \prod_{c=1}^\gamma \mathbf{V}_c \right) \right]
\end{equation}
with each $v_a$ of the form $y_{r_a j}$ for some $r_a$, each $\mathbf{v}_b$ of the form $\ip{\mathbf{y}_{s_b},\mathbf{y}_{t_b}}$ for some $s_b \neq t_b$, and each $\mathbf{V}_c$ of the form $\|\mathbf{y}_{u_c}\|^2$ for some $u_c$, and with $\alpha, \beta, \gamma \geq 0$, including allowing different indices to coincide. A \emph{straight estimate}, which amounts to just a power counting, is the following:
\begin{align*}
	&\abs{\E\left[\left( \prod_{a=1}^\alpha v_a\right) \left( \prod_{b=1}^\beta \mathbf{v}_b \right) \left( \prod_{c=1}^\gamma \mathbf{V}_c \right) \right]} \\
	&\leq \left( \prod_{a=1}^\alpha \E[\abs{v_a}^{\alpha+\beta+\gamma}]^{\frac{1}{\alpha+\beta+\gamma}} \right) \left( \prod_{b=1}^\beta \E[\abs{\mathbf{v}_b}^{\alpha+\beta+\gamma}]^{\frac{1}{\alpha+\beta+\gamma}} \right) \left( \prod_{c=1}^\gamma \E[\abs{\mathbf{V}_c}^{\alpha+\beta+\gamma}]^{\frac{1}{\alpha+\beta+\gamma}} \right) \\
	&\leq C_{\alpha,\beta,\gamma} (N^{1/2})^{\beta} N^{\gamma}.
\end{align*}
In our case, $\alpha$, $\beta$, and $\gamma$ will take values in some finite set depending on $p$, so we can write $\max_{\alpha,\beta,\gamma} C_{\alpha,\beta,\gamma}$ as some $C_p$. This means it suffices to just count how many norms and how many inner products of distinct $\mathbf{y}$'s we have. 

So our procedure is to use \eqref{eqn:star_rules} repeatedly in 
a certain precise way we describe below, then apply a straight estimate for the remaining factors.
Roughly speaking, we will count  how many times we can apply \eqref{eqn:star_rules}. Note that
 each application gains a factor $N^{-1/2}$
over the straight estimate.  A final power counting shows that we can apply \eqref{eqn:star_rules}  sufficiently many
times to obtain our target bound.
 
 We track the integrations using a graphical notation that we now introduce. The basic idea is that we are going to define one fixed graph, which at the beginning has no colors. For each assignment $A$, we will obtain a graph $\mf{G}_A$ by coloring the vertices of this graph in up to $k$ distinct colors; multiple vertices with the same color correspond to multiple appearances in $A$ of the same index. Then applying the rules \eqref{eqn:star_rules} corresponds to ``contracting'' an edge, in the sense of erasing an edge from some $v$ to some $v'$ and also erasing an edge from $v'$ to some $v''$, and drawing a new edge from $v$ to $v''$. 

The basic graph has $4p+1$ vertices and $5p$ edges. We name the vertices by index names in the assignment, with a slight abuse of notation: We have vertices ``$\ell_1$,'' ``$\ell'_1$,'' $\ldots$, ``$\ell_p$,'' ``$\ell'_p$,'' ``$m_1$,'' ``$m'_1$,'' $\ldots$ ``$m_p$,'' ``$m'_p$,'' plus one special vertex called ``$\mathbf{e}_j$.'' Later we will assign values to these variables, but these will appear in the graph as colors. If $\ell_1 = 4$, then we have a vertex called $\ell_1$ with the ``color'' $4$, not a vertex called $4$. The edge set consists of the cycles $\mathbf{e}_j - \ell_a - m_a - m'_a - \ell'_a - \mathbf{e}_j$ for each $a \in \llbracket 1, p \rrbracket$, and we will use the term \emph{unit} to refer to a vertex subset of the form $\{\ell_a, m_a, m'_a, \ell'_a\}$ for some $a \in \llbracket 1, p \rrbracket$. Occasionally we will use the term abusively to denote these four vertices plus the three edges between them; or to denote these four vertices, the three edges between them, plus the two edges between them and $\mathbf{e}_j$. The $p=3$ graph is shown in Figure~\ref{fig:basic} (although the basic graph ``has no colors,'' we draw the $\mathbf{e}_j$ vertex in white to remember that it is special).

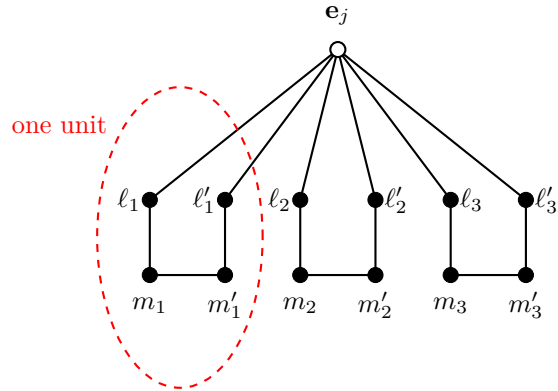
\begin{figure}[h!]
\begin{center}
\begin{tikzpicture}
	\draw [black, fill=black] (-2.5,1) circle [radius=0.1];
	\draw [black, fill=black] (-2.5,0) circle [radius=0.1];
	\draw [black, fill=black] (-1.5,1) circle [radius=0.1];
	\draw [black, fill=black] (-1.5,0) circle [radius=0.1];
	\draw [thick](0,3) -- (-2.5,1) -- (-2.5,0) -- (-1.5,0) -- (-1.5,1) -- (0,3);
	\draw [black, fill=black] (0.5,1) circle [radius=0.1];
	\draw [black, fill=black] (0.5,0) circle [radius=0.1];
	\draw [black, fill=black] (-0.5,1) circle [radius=0.1];
	\draw [black, fill=black] (-0.5,0) circle [radius=0.1];
	\draw [thick](0,3) -- (0.5,1) -- (0.5,0) -- (-0.5,0) -- (-0.5,1) -- (0,3);
	\draw [black, fill=black] (1.5,1) circle [radius=0.1];
	\draw [black, fill=black] (1.5,0) circle [radius=0.1];
	\draw [black, fill=black] (2.5,1) circle [radius=0.1];
	\draw [black, fill=black] (2.5,0) circle [radius=0.1];
	\draw [thick](0,3) -- (1.5,1) -- (1.5,0) -- (2.5,0) -- (2.5,1) -- (0,3);
	\draw [black, thick, fill=white] (0,3) circle [radius=0.1];
	\node [above] at (0,3.2) {$\mathbf{e}_j$};
	\node [left] at (-2.5,1) {$\ell_1$};
	\node [left] at (-1.5,1) {$\ell'_1$};
	\node [left] at (-0.5,1) {$\ell_2$};
	\node [right] at (0.5,1) {$\ell'_2$};
	\node [right] at (1.5,1) {$\ell_3$};
	\node [right] at (2.5,1) {$\ell'_3$};
	\node at (-2.5,-0.4) {$m_1$};
	\node at (-1.5,-0.4) {$m'_1$};
	\node at (-0.5,-0.4) {$m_2$};
	\node at (0.5,-0.4) {$m'_2$};
	\node at (1.5,-0.4) {$m_3$};
	\node at (2.5,-0.4) {$m'_3$};
	\draw [dashed, thick, red] (-2.1,0.5) ellipse (1.1 and 2);
	\node [red] at (-3.7,2) {one unit};
\end{tikzpicture}
\end{center}
\caption{The basic graph when $p = 3$, before coloring, with one of the three ``units'' labelled.}
\label{fig:basic}
\end{figure}

The point of this graph is that edges correspond to inner products that are present in $\mathbf{y}_A$; for example, $\mathbf{y}_A$ contains $\ip{\mathbf{y}_{\ell_1},\mathbf{y}_{m_1}}$ and $\ip{\mathbf{y}_{m_1},\mathbf{y}_{m'_1}}$ but not $\ip{\mathbf{y}_{\ell_1},\mathbf{y}_{\ell'_1}}$. Scalars like $y_{\ell_1j}$ are thought of as $\ip{\mathbf{y}_{\ell_1},\mathbf{e}_j}$ where $\mathbf{e}_j$ the $j$th vector in the standard basis, and this is the point of the vertex $\mathbf{e}_j$; it is special because the naive size of $\ip{\mathbf{y}_a,\mathbf{y}_b}$ is order $N^{(1+\delta_{ab})/2}$ but the naive size of $\ip{\mathbf{y}_a,\mathbf{e}_j}$ is order one. 

To encode assignments $A$ in the graph, we color the vertices other than $\mathbf{e}_j$, recording different values by different colors. Fix once and for all $k$ colors, which we will sometimes number. For each assignment $A$, we start with the basic graph whose $p=3$ case is shown in Figure~\ref{fig:basic}, then obtain a graph we call $\mf{G}_A$ by coloring the vertices according to the assignment. For example, if $A$ assigns $\ell_1$ to $7$, $\ell'_1$ to $8$, $m_1$ to $9$, and $m'_1$ to $9$, then the vertex $\ell_1$ gets the ``color'' $7$, the vertex $\ell'_1$ gets the ``color'' $8$, and so on. For example, when $p = 3$, the assignment $A = \{\ell_1 = 1, \ell'_1 = 1, \ell_2 = 3, \ell'_2 = 5, \ell_3 = 8, \ell'_3 = 9, m_1 = 2, m'_1 = 3, m_2 = 4, m'_2=6, m_3 = 7, m'_3 = 7\}$ is drawn in Figure \ref{fig:assignment} (choosing, for the sake of the picture, actual colors).

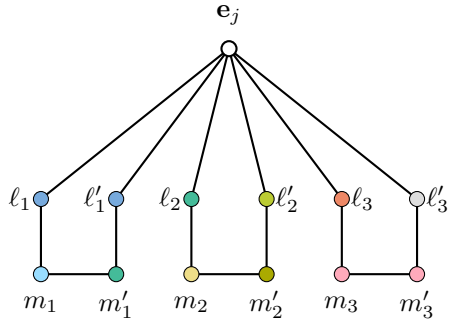
\begin{figure}[h!]
\begin{center}
\begin{tikzpicture}
	\draw [thick](0,3) -- (-2.5,1) -- (-2.5,0) -- (-1.5,0) -- (-1.5,1) -- (0,3);
	\draw [black, fill=color1] (-2.5,1) circle [radius=0.1];
	\draw [black, fill=color2] (-2.5,0) circle [radius=0.1];
	\draw [black, fill=color1] (-1.5,1) circle [radius=0.1];
	\draw [black, fill=color3] (-1.5,0) circle [radius=0.1];
	
	\draw [thick](0,3) -- (0.5,1) -- (0.5,0) -- (-0.5,0) -- (-0.5,1) -- (0,3);
	\draw [black, fill=color4] (0.5,1) circle [radius=0.1];
	\draw [black, fill=color5] (0.5,0) circle [radius=0.1];
	\draw [black, fill=color3] (-0.5,1) circle [radius=0.1];
	\draw [black, fill=color6] (-0.5,0) circle [radius=0.1];
	
	\draw [thick](0,3) -- (1.5,1) -- (1.5,0) -- (2.5,0) -- (2.5,1) -- (0,3);
	\draw [black, fill=color7] (1.5,1) circle [radius=0.1];
	\draw [black, fill=color8] (1.5,0) circle [radius=0.1];
	\draw [black, fill=color9] (2.5,1) circle [radius=0.1];
	\draw [black, fill=color8] (2.5,0) circle [radius=0.1];

	\draw [black, thick, fill=white] (0,3) circle [radius=0.1];
	\node [above] at (0,3.2) {$\mathbf{e}_j$};
	\node [left] at (-2.5,1) {$\ell_1$};
	\node [left] at (-1.5,1) {$\ell'_1$};
	\node [left] at (-0.5,1) {$\ell_2$};
	\node [right] at (0.5,1) {$\ell'_2$};
	\node [right] at (1.5,1) {$\ell_3$};
	\node [right] at (2.5,1) {$\ell'_3$};
	\node at (-2.5,-0.4) {$m_1$};
	\node at (-1.5,-0.4) {$m'_1$};
	\node at (-0.5,-0.4) {$m_2$};
	\node at (0.5,-0.4) {$m'_2$};
	\node at (1.5,-0.4) {$m_3$};
	\node at (2.5,-0.4) {$m'_3$};
\end{tikzpicture}
\end{center}
\caption{A sample graph $\mf{G}_A$, for the case $p = 3$ and the assignment $A = \{\ell_1 = 1, \ell'_1 = 1, \ell_2 = 3, \ell'_2 = 5, \ell_3 = 8, \ell'_3 = 9, m_1 = 2, m'_1 = 3, m_2 = 4, m'_2=6, m_3 = 7, m'_3 = 7\}$. For example, the vertices $\ell_1$ and $\ell'_1$ are both assigned the ``color'' numbered $1$, which is drawn in light blue for the figure, and the vertex $m_1$ is assigned the ``color'' $2$, which is drawn in light cyan for the figure. This graph has exactly one \emph{monochromatic edge}, namely the edge between $m_3$ and $m'_3$, which are both colored pink. (This name is slightly abusive, since we do not color the edges, unlike in the unit graphs we introduce later; ``monochromatic'' refers just to the colors on the vertices which the edge connects.)}
\label{fig:assignment}
\end{figure}

This corresponds to
\begin{equation}
\label{eqn:example_yA}
	\mathbf{y}_A = y_{1j}^2 y_{3j}y_{5j}y_{7j}y_{9j} \ip{\mathbf{y}_1,\mathbf{y}_2} \ip{\mathbf{y}_1,\mathbf{y}_3} \ip{\mathbf{y}_2,\mathbf{y}_3} \ip{\mathbf{y}_3,\mathbf{y}_4} \ip{\mathbf{y}_5,\mathbf{y}_6} \ip{\mathbf{y}_4,\mathbf{y}_6} \ip{\mathbf{y}_7,\mathbf{y}_8} \ip{\mathbf{y}_9,\mathbf{y}_8} \|\mathbf{y}_8\|^2.
\end{equation}

(If $A'$ also has pattern $P$, then $\mf{G}_{A'}$ will be related to $\mf{G}_A$ just by replacing the colors -- for example, the pink vertices in Figure~\ref{fig:assignment} could instead be colored brown.)

One of the main enemies in our counting is the presence of \emph{monochromatic edges}, meaning edges connecting vertices of the same color, like the edge between $m_3$ and $m'_3$ in Figure~\ref{fig:assignment}, since they correspond to norms like $\|\mathbf{y}_7\|^2$ which are large, of order $N$. 
 It will be very important for us that only $p$ of the edges can be monochromatic, namely those between $m_a$ and $m'_a$ for $a = 1, \ldots, p$. The edges between $\ell_a$ and $m_a$, or between $\ell'_a$ and $m'_a$, can \emph{never} be monochromatic, since in \eqref{eqn:sigma4infty}, we always have $m_a < \ell_a$ and $m'_a < \ell'_a$. 

It will be useful to keep track of which colors are shared across units, which we will do using a second auxiliary 
 graph, that we will call the \emph{unit graph} and denote by $\mf{U}_A$. This graph has just $p$ vertices, each vertex representing a unit.
  The vertices have no color, but if two units each contain a vertex of a certain color, then we place an edge \emph{of that color} between them. Two vertices can be connected by multiple edges (if each has a red vertex and a blue vertex, say), but only of distinct colors (if two units each have two red vertices, we put one red edge between them, not two). For example, Figure~\ref{fig:unit} shows one example graph $\mf{G}_A$ and the corresponding unit graph $\mf{U}_A$ when $p = 4$.

\begin{figure}[h!]
\begin{center}
\begin{tikzpicture}
	\draw [thick](1,3) -- (-2.5,1) -- (-2.5,0) -- (-1.5,0) -- (-1.5,1) -- (1,3);
	\draw [black, fill=color7] (-2.5,1) circle [radius=0.1];
	\draw [black, fill=color4] (-2.5,0) circle [radius=0.1];
	\draw [black, fill=color7] (-1.5,1) circle [radius=0.1];
	\draw [black, fill=color1] (-1.5,0) circle [radius=0.1];
	
	\draw [thick](1,3) -- (0.5,1) -- (0.5,0) -- (-0.5,0) -- (-0.5,1) -- (1,3);
	\draw [black, fill=color4] (0.5,1) circle [radius=0.1];
	\draw [black, fill=color2] (0.5,0) circle [radius=0.1];
	\draw [black, fill=color1] (-0.5,1) circle [radius=0.1];
	\draw [black, fill=color3] (-0.5,0) circle [radius=0.1];
	
	\draw [thick](1,3) -- (1.5,1) -- (1.5,0) -- (2.5,0) -- (2.5,1) -- (1,3);
	\draw [black, fill=color7] (1.5,1) circle [radius=0.1];
	\draw [black, fill=color6] (1.5,0) circle [radius=0.1];
	\draw [black, fill=color7] (2.5,1) circle [radius=0.1];
	\draw [black, fill=color6] (2.5,0) circle [radius=0.1];
	
	\draw [thick](1,3) -- (3.5,1) -- (3.5,0) -- (4.5,0) -- (4.5,1) -- (1,3);
	\draw [black, fill=color5] (3.5,1) circle [radius=0.1];
	\draw [black, fill=color8] (3.5,0) circle [radius=0.1];
	\draw [black, fill=color9] (4.5,1) circle [radius=0.1];
	\draw [black, fill=color10] (4.5,0) circle [radius=0.1];

	\draw [black, thick, fill=white] (1,3) circle [radius=0.1];
	\node [above] at (1,3.2) {$\mathbf{e}_j$};
	\node [left] at (-2.5,1) {$\ell_1$};
	\node [left] at (-1.5,1) {$\ell'_1$};
	\node [left] at (-0.5,1) {$\ell_2$};
	\node [left] at (0.5,1) {$\ell'_2$};
	\node [right] at (1.5,1) {$\ell_3$};
	\node [right] at (2.5,1) {$\ell'_3$};
	\node [right] at (3.5,1) {$\ell_4$};
	\node [right] at (4.5,1) {$\ell'_4$};
	\node at (-2.5,-0.4) {$m_1$};
	\node at (-1.5,-0.4) {$m'_1$};
	\node at (-0.5,-0.4) {$m_2$};
	\node at (0.5,-0.4) {$m'_2$};
	\node at (1.5,-0.4) {$m_3$};
	\node at (2.5,-0.4) {$m'_3$};
	\node at (3.5,-0.4) {$m_4$};
	\node at (4.5,-0.4) {$m'_4$};
	
	\node[left] at (-4,1) {$\mf{G}_A$};
	
	\draw (-6,-2) -- (8,-2);
	
	\draw [ultra thick, color1] (0,-3) arc (45:135:1.41);
	\draw [ultra thick, color4] (0,-3) arc (-45:-135:1.41);
	\draw [ultra thick, color7] (2,-3) arc (-45:-135:2.87);
	\draw [black, fill=black] (-2,-3) circle [radius=0.1];
	\draw [black, fill=black] (0,-3) circle [radius=0.1];
	\draw [black, fill=black] (2,-3) circle [radius=0.1];
	\draw [black, fill=black] (4,-3) circle [radius=0.1];
	
	\node[left] at (-4,-3) {$\mf{U}_A$};
\end{tikzpicture}
\end{center}
\caption{A $p=4$ sample graph $\mf{G}_A$ and the corresponding unit graph $\mf{U}_A$. Notice that $\mf{G}_A$ has colored vertices and colorless edges, while $\mf{U}_A$ has colorless vertices and colored edges. The unit graph has a light blue edge between the first and second vertices because $m'_1$ and $\ell_2$ are both light blue, and so on. Notice also that, although the first and third units each have two orange vertices, we only put one orange edge, not two, between the first and third vertices in the unit graph.}
\label{fig:unit}
\end{figure}

The unit graph will eventually tell us how many integrations we will do, i.e., in which units we will apply \eqref{eqn:star_rules}, in the following way: Say that a vertex $v$ in the unit graph is \emph{monochromatic} if all edges touching it have the same color, otherwise \emph{polychromatic}. By convention, we take isolated vertices to be monochromatic. In the unit graph in Figure~\ref{fig:unit}, the third and fourth vertices are monochromatic.

Suppose the number of polychromatic vertices in $\mf{U}_A$ is $q \in \{0, 1, \ldots, p\}$. For $j = 2, 3, 4$, let $n^{\textup{pc}}_j$ (respectively, $n^{\textup{mc}}_j$) be the number of polychromatic (respectively, monochromatic) vertices in $\mf{U}_A$ whose corresponding units in $\mf{G}_A$ have exactly $j$ distinct colors. For example, in Figure~\ref{fig:unit}, we have $q = 2$, $n^{\textup{pc}}_2 = 0$, $n^{\textup{pc}}_3 = 1$, $n^{\textup{pc}}_4 = 1$, $n^{\textup{mc}}_2 = 1$, $n^{\textup{mc}}_3 = 0$, $n^{\textup{mc}}_4 = 1$. Notice that $q$, the $n^{\textup{pc}}_j$'s, and the $n^{\textup{mc}}_j$'s depend only on the pattern, not on the assignment. Notice also that 
\[
	n^{\textup{pc}}_2 + n^{\textup{pc}}_3 + n^{\textup{pc}}_4 = q, \qquad n^{\textup{mc}}_2 + n^{\textup{mc}}_3 + n^{\textup{mc}}_4 = p-q.
\]
We then make the following claims:
\begin{itemize}
\item \textbf{Claim 1:} For each $p$, there exists $C_p$ with
\begin{equation}
\label{eqn:graphs_claim_1}
	\abs{\E[\mathbf{y}_P]} \leq C_p (N^2)^{n^{\textup{pc}}_2 + n^{\textup{pc}}_3} (N^{3/2})^{n^{\textup{pc}}_4} (N^2)^{n^{\textup{mc}}_2} (N^{3/2})^{n^{\textup{mc}}_3} N^{n^{\textup{mc}}_4}.
\end{equation}
\item \textbf{Claim 2:}
\begin{equation}
\label{eqn:graphs_claim_2}
	a(P) \leq k^{2n^{\textup{pc}}_2 + 3n^{\textup{pc}}_3 + 4n^{\textup{pc}}_4 - q + 2n^{\textup{mc}}_2 + 3n^{\textup{mc}}_3 + 4n^{\textup{mc}}_4}.
\end{equation}
\end{itemize}
Assume these claims for a moment; then since
\[
	2n^{\textup{pc}}_2 + 3n^{\textup{pc}}_3 + 4n^{\textup{pc}}_4 - q + 2n^{\textup{mc}}_2 + 3n^{\textup{mc}}_3 + 4n^{\textup{mc}}_4 \geq 2(n^{\textup{pc}}_2 + n^{\textup{pc}}_3 + n^{\textup{pc}}_4) - q + 2(n^{\textup{mc}}_2 + n^{\textup{mc}}_3 + n^{\textup{mc}}_4) = q+2(p-q) \geq p
\]
and
\begin{align*}
	&\frac{2n^{\textup{pc}}_2 + 3n^{\textup{pc}}_3 + 4n^{\textup{pc}}_4 - q + 2n^{\textup{mc}}_2 + 3n^{\textup{mc}}_3 + 4n^{\textup{mc}}_4}{2} + 2(n^{\textup{pc}}_2 + n^{\textup{pc}}_3 + n^{\textup{mc}}_2) + \frac{3}{2}(n^{\textup{pc}}_4 + n^{\textup{mc}}_3) + n^{\textup{mc}}_4 \\
	&= 3n^{\textup{pc}}_2 + \frac{7}{2}n^{\textup{pc}}_3 + \frac{7}{2}n^{\textup{pc}}_4 - \frac{q}{2} + 3(n^{\textup{mc}}_2+n^{\textup{mc}}_3 + n^{\textup{mc}}_4) \leq 3q + 3(p-q) = 3p,
\end{align*}
we have
\[
	\sup_{P \in \ms{P}} N^{-3p} a(P)\abs{\E[\mathbf{y}_P]} \leq C_p \left(\frac{k}{\sqrt{N}}\right)^p.
\]
We plug this into \eqref{eqn:pattern_sum}; since the number of legal patterns (i.e., the cardinality of the set $\ms{P}$) depends on $p$ but not on $N$, this completes the proof of Lemma \ref{lem:graphs}, modulo the claims above.

We start with Claim 1. Notice that $\E[\mathbf{y}_P]$ is of the form \eqref{eqn:v_form}, and that applying the integration rules \eqref{eqn:star_rules} to an equation of the form \eqref{eqn:v_form} produces another equation of the form \eqref{eqn:v_form}. In fact, we will start with $\E[\mathbf{y}_P]$, apply \eqref{eqn:star_rules} some number of times, and eventually take the straight estimate on what remains.

Since the straight estimate is just power counting, we can consider it one unit at a time. The straight estimate for any unit is at most $C_p N^2$, since there is at most one monochromatic edge (namely, the one between the $m$'s), which has size $N$, and two bichromatic edges, which each have size $\sqrt{N}$. The straight estimate for a unit with four distinct colors is $N^{3/2}$, since there cannot be any monochromatic edge. From \eqref{eqn:graphs_claim_1}, we see that we are taking the straight estimate for all units corresponding to polychromatic vertices in $\mf{U}_A$, and for all units with two distinct colors corresponding to a monochromatic vertex in $\mf{U}_A$. We treat the other units as follows:
\begin{itemize}
\item If the unit has exactly three distinct colors and corresponds to a monochromatic vertex in $\mf{U}_A$, then there are two distinct colors that appear in the unit exactly one time (and another that appears twice); since at most one of these also appears in other units, at least one of these only appears in this unit. Suppose it is the color $b$; then $\E[\mathbf{y}_P]$ either has the form $\E[\ip{\mathbf{y}_a,\mathbf{y}_b}\ip{\mathbf{y}_b,\mathbf{y}_c}X]$ where $a$, $b$, and $c$ are distinct and $X$ is independent of $\mathbf{y}_b$, or has the form $\E[y_{bj} \ip{\mathbf{y}_a,\mathbf{y}_b} X]$ where $a$ and $b$ are distinct and $X$ is independent of $\mathbf{y}_b$. Either way, we apply \eqref{eqn:star_rules} to reduce by one the count of inner products of distinct $y$'s, while leaving the counts of norm-squares unchanged. Then the straight estimate on the result will be at most $N^{3/2}$, which is exactly what we want for \eqref{eqn:graphs_claim_1}.
\item If the unit has exactly four distinct colors and corresponds to a monochromatic vertex in $\mf{U}_A$, then there are two distinct colors that appear in the unit exactly one time, and appear in no other unit. We integrate each of these as in the case of three distinct colors; each gains a factor of $\sqrt{N}$, so the straight estimate on the result is $N$, which is exactly what we want for \eqref{eqn:graphs_claim_1}.
\end{itemize}
This finishes the proof of Claim 1 \eqref{eqn:graphs_claim_1}.

Now we prove Claim 2. Without the factor of $-q$, \eqref{eqn:graphs_claim_2} would be trivial, and in fact usually an overcount: We just choose two colors ($k^2$) for the units with two colors, three colors ($k^3$) for the units with three colors, and four colors ($k^4$) for the units with four colors, to obtain
\begin{equation}
\label{eqn:colors_reference}
	a(P) \leq k^{2n^{\textup{pc}}_2 + 3n^{\textup{pc}}_3 + 4n^{\textup{pc}}_4 + 2n^{\textup{mc}}_2 + 3n^{\textup{mc}}_3 + 4n^{\textup{mc}}_4},
\end{equation}
which overcounts because it does not impose the requirement in some patterns that colors overlap between units. To account for this, we introduce 
a new auxiliary graph, called the \emph{path graph} $\mf{P}_A$, which is obtained by removing some edges from the unit graph by the following procedure:

Fix a color appearing in $\mf{U}_A$, say blue, and notice that the ``blue subgraph'' of $\mf{U}_A$ consisting of all blue edges and incident vertices is the complete graph on its vertex set $V_b$, by definition. Remove edges as necessary until the blue subgraph is just a single path spanning $V_b$. Repeat for every color that appears in $\mf{U}_A$.

Notice that the path graph can be the same as the unit graph, if no edges needed removing. One example is in Figure~\ref{fig:path}.

\begin{figure}[h!]
\begin{center}
\begin{tikzpicture}
	\draw [ultra thick, color1] (0,1) -- (1,0) -- (-1,0) -- (0,1);
	\draw [ultra thick, color7] (1,0) -- (0,-1) -- (-1,0);
	\draw [ultra thick, color7] (0,1) arc (90:180:1);
	\draw [ultra thick, color7] (1,0) arc (0:90:1);
	\draw [ultra thick, color7] (-1,0) arc [start angle = 150, end angle=390, x radius = 1.15, y radius = 1];
	\draw [ultra thick, color7] (0,-1) -- (0,1);
	\draw [ultra thick, color4] (1,2) -- (0,1);
	\draw [ultra thick, color10] (-1,2) -- (0,2);
	\draw [black, fill=black] (1,0) circle [radius=0.1];
	\draw [black, fill=black] (-1,0) circle [radius=0.1];
	\draw [black, fill=black] (0,1) circle [radius=0.1];
	\draw [black, fill=black] (-1,2) circle [radius=0.1];
	\draw [black, fill=black] (1,2) circle [radius=0.1];
	\draw [black, fill=black] (0,2) circle [radius=0.1];
	\draw [black, fill=black] (0,-1) circle [radius=0.1];
	
	\node at (0,3) {This unit graph ...};
	
	\draw [ultra thick] (4,-2) -- (4,4);
	
	\draw [ultra thick, color1] (7,1) -- (8,0) -- (6,0);
	\draw [ultra thick, color7] (7,1) arc (90:180:1);
	\draw [ultra thick, color7] (7,1) -- (7,-1) -- (8,0);
	\draw [ultra thick, color4] (8,2) -- (7,1);
	\draw [ultra thick, color10] (6,2) -- (7,2);
	\draw [black, fill=black] (8,0) circle [radius=0.1];
	\draw [black, fill=black] (6,0) circle [radius=0.1];
	\draw [black, fill=black] (7,1) circle [radius=0.1];
	\draw [black, fill=black] (6,2) circle [radius=0.1];
	\draw [black, fill=black] (8,2) circle [radius=0.1];
	\draw [black, fill=black] (7,2) circle [radius=0.1];
	\draw [black, fill=black] (7,-1) circle [radius=0.1];
	
	\node at (7,3) {... could lead to this path graph ...};
	
	\draw [ultra thick] (10,-2) -- (10,4);
	
	\draw [ultra thick, color1] (12,0) -- (13,1) -- (14,0);
	\draw [ultra thick, color7] (14,0) arc (0:90:1);
	\draw [ultra thick, color7] (12,0) -- (13,-1) -- (14,0);
	\draw [ultra thick, color4] (14,2) -- (13,1);
	\draw [ultra thick, color10] (12,2) -- (13,2);
	\draw [black, fill=black] (14,0) circle [radius=0.1];
	\draw [black, fill=black] (12,0) circle [radius=0.1];
	\draw [black, fill=black] (13,1) circle [radius=0.1];
	\draw [black, fill=black] (12,2) circle [radius=0.1];
	\draw [black, fill=black] (14,2) circle [radius=0.1];
	\draw [black, fill=black] (13,2) circle [radius=0.1];
	\draw [black, fill=black] (13,-1) circle [radius=0.1];
	
	\node at (13,3) {... or to this path graph.};
\end{tikzpicture}
\end{center}
\caption{An example of the multiple path graphs that can be obtained from a given unit graph. The middle and the right graphs have no extraneous edges, so if they were themselves unit graphs, they would be the same as their corresponding path graphs.}
\label{fig:path}
\end{figure}
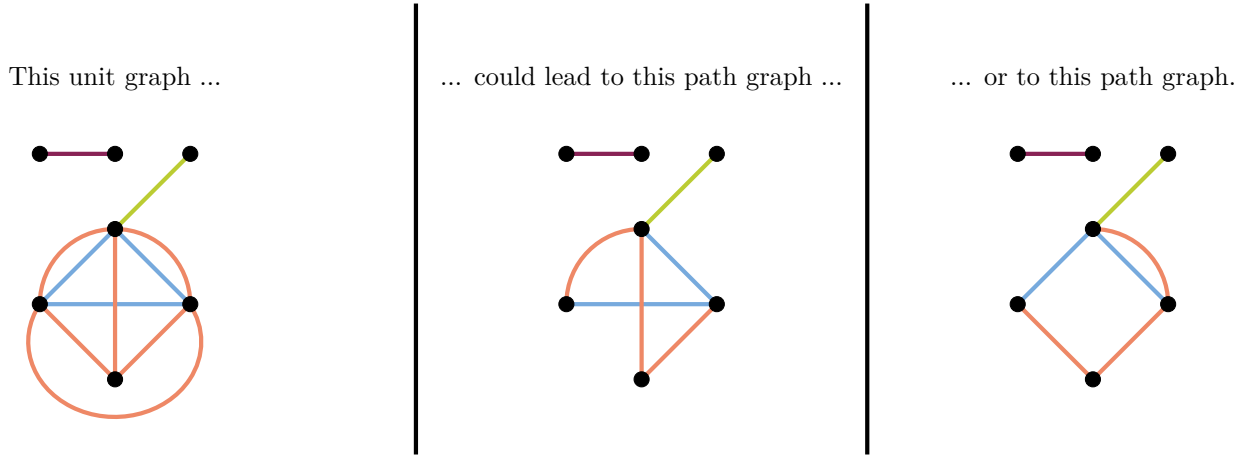

It is certainly not unique (see Figure~\ref{fig:path}), but it has the good property, by construction, that its edge set can be partitioned into paths of different colors: exactly one blue path, exactly one red path, and so on. It is easy to see that each edge of the path graph saves one color with respect to the count that ignores inter-unit connections, meaning that if $\mf{P}_A$ has $r$ edges, then
\begin{equation}
\label{eqn:graphs_claim_2_strongest}
	a(P) \leq k^{2n^{\textup{pc}}_2 + 3n^{\textup{pc}}_3 + 4n^{\textup{pc}}_4 - r + 2n^{\textup{mc}}_2 + 3n^{\textup{mc}}_3 + 4n^{\textup{mc}}_4}.
\end{equation}
However, if $v$ is polychromatic in $\mf{U}_A$, then it is clearly also polychromatic in $\mf{P}_A$, hence has at least two incident edges in $\mf{P}_A$. Thus $\mf{P}_A$ has $q$ vertices with at least two incident edges; hence $r \geq q$, which means that \eqref{eqn:graphs_claim_2_strongest} implies \eqref{eqn:graphs_claim_2} and finishes the proof of Lemma \ref{lem:graphs}.
\end{proof}


\subsection{Proof of Proposition \ref{prop:cumulant_expansion}.}\

The goal of this subsection is to prove the following:

\begin{prop}
\label{prop:cumulant_expansion}
If 
\[
	k = \OO(N^{1/2-\delta}),
\]
then for every fixed $D, \epsilon > 0$ there exists $C_{D,\epsilon}$ with 
\[
	\P\left(\max_{j=1}^N \abs{\sum_{i=1}^k a_i \widetilde{\Delta}_{ij} y_{ij}} > \epsilon \right) \leq C_{D,\epsilon} N^{-D}.
\]
\end{prop}

\begin{proof}
We start with Markov's inequality
\[
	\P\left(\max_{j=1}^N \abs{\sum_{i=1}^k a_i \widetilde{\Delta}_{ij} y_{ij}} > \epsilon \right) \leq N\epsilon^{-p} \max_{j=1}^N \E\left[\abs{\sum_{i=1}^k a_i \widetilde{\Delta}_{ij} y_{ij}}^p\right],
\]
and finish by applying Lemma \ref{lem:graphs_easier} below for some large $p$. 
\end{proof}

\begin{lem}
\label{lem:graphs_easier}
If $p \geq 1$ is an integer and $a$ is as in \eqref{eqn:as_bounded}, then
\[
	\max_{j=1}^N \E\left[ \left( \sum_{i=1}^k a_i \widetilde{\Delta}_{ij} y_{ij} \right)^p \right] \leq C_p a^p \left(\frac{k}{\sqrt{N}}\right)^p.
\]
\end{lem}
\begin{proof}
This is essentially an easier analogue of the proof of Lemma \ref{lem:graphs}, so we skip the details in some steps. We expand
\[
	\E\left[ \left( \sum_{i=1}^k a_i \widetilde{\Delta}_{ij} y_{ij} \right)^p \right] = N^{-p} \sum_{i_1, \ldots, i_p=1}^k \sum_{m_1=1}^{i_1-1} \sum_{m_2=1}^{i_2-1} \cdots \sum_{m_p=1}^{i_p-1} \left( \prod_{b=1}^p a_{i_b} \right) \E\left[ \prod_{b=1}^p y_{i_b j} y_{m_b j} \ip{\mathbf{y}_{i_b},\mathbf{y}_{m_b}} \right]
\]
Inside the expectation, the only question is which indices coincide, so we write $A$ for an \emph{assignment of values} to these indices; $P$ for the associated pattern, i.e., the knowledge of which $i$ and $m$ variables coincide in $A$; $\ms{A}$ for the set of all legal assignments; and $\ms{P}$ for the set of all legal patterns. Write also
\begin{align*}
	a_A &\defeq \prod_{b=1}^p a_{i_b}, \\
	\mathbf{y}_A &\defeq \prod_{b=1}^p y_{i_b j} y_{m_b j} \ip{\mathbf{y}_{i_b},\mathbf{y}_{m_b}}.
\end{align*}
Then
\[
	\E\left[ \left( \sum_{i=1}^k a_i \widetilde{\Delta}_{ij} y_{ij} \right)^p \right] = N^{-p} \sum_{A \in \ms{A}} a_A \E[\mathbf{y}_A].
\]
As before, $\E[\mathbf{y}_A]$ only depends on the pattern of $A$, but now $a_A$ actually depends on $A$ itself; however, since $\max_{i=1}^k \abs{a_i} \leq a$, we have $\max_{A \in \ms{A}} \abs{a_A} \leq a^p$, so that
\begin{equation}
\label{eqn:easy_pattern_sum}
	\abs{\E\left[ \left( \sum_{i=1}^k a_i \widetilde{\Delta}_{ij} y_{ij} \right)^p \right]} \leq a^p N^{-p} \sum_{A \in \ms{A}}  \abs{\E[\mathbf{y}_A]} = a^p N^{-p} \sum_{P \in \ms{P}} a(P) \abs{\E[\mathbf{y}_P]}
\end{equation}
where $a(P)$ is the number of assignments with pattern $P$. We will again use graphs $\mf{G}_A$, but the basic graph now has $2p+1$ vertices and $3p$ edges: Vertices named ``$i_1$,'' ..., ``$i_p$,'' ``$m_1$,'' ..., ``$m_p$,'' plus one special vertex called ``$\mathbf{e}_j$,'' with edges along the cycle $\mathbf{e}_j - i_b - m_b - \mathbf{e}_j$ for each $b = 1, \ldots, p$. A sample colored graph with $p = 4$ is given in Figure~\ref{fig:two},
 along with a pictorial definition of unit, consisting of only two vertices this time, and the associated unit graph $\mf{U}_A$.

\begin{figure}[h!]
\begin{center}
\begin{tikzpicture}
	\draw [thick](1,3) -- (-2.5,1) -- (-2.5,0) -- (1,3);
	\draw [black, fill=color7] (-2.5,1) circle [radius=0.1];
	\draw [black, fill=color1] (-2.5,0) circle [radius=0.1];
	
	\draw [thick](1,3) -- (-0.5,0) -- (-0.5,1) -- (1,3);
	\draw [black, fill=color1] (-0.5,1) circle [radius=0.1];
	\draw [black, fill=color2] (-0.5,0) circle [radius=0.1];
	
	\draw [thick](1,3) -- (2.5,0) -- (2.5,1) -- (1,3);
	\draw [black, fill=color7] (2.5,1) circle [radius=0.1];
	\draw [black, fill=color3] (2.5,0) circle [radius=0.1];
	
	\draw [thick](1,3) -- (4.5,0) -- (4.5,1) -- (1,3);
	\draw [black, fill=color4] (4.5,1) circle [radius=0.1];
	\draw [black, fill=color6] (4.5,0) circle [radius=0.1];

	\draw [black, thick, fill=white] (1,3) circle [radius=0.1];
	\node [above] at (1,3.2) {$\mathbf{e}_j$};
	\node [left] at (-2.5,1) {$i_1$};
	\node [left] at (-0.5,1) {$i_2$};
	\node [right] at (2.5,1) {$i_3$};
	\node [right] at (4.5,1) {$i_4$};
	\node at (-2.5,-0.4) {$m_1$};
	\node at (-0.5,-0.4) {$m_2$};
	\node at (2.5,-0.4) {$m_3$};
	\node at (4.5,-0.4) {$m_4$};
	
	\draw [dashed, thick, red] (4.5,0.5) ellipse (1.1 and 2);
	\node [red] at (6.3,2) {one unit};
	
	\node[left] at (-4,1) {$\mf{G}_A$};
	
	\draw (-6,-2) -- (8,-2);
	
	\draw [rounded corners, ultra thick, color1] (-2.5,-3) -- (-.5,-3);
	\draw [ultra thick, color7] (2.5,-3) arc (-45:-135:3.55);
	\draw [black, fill=black] (-2.5,-3) circle [radius=0.1];
	\draw [black, fill=black] (-0.5,-3) circle [radius=0.1];
	\draw [black, fill=black] (2.5,-3) circle [radius=0.1];
	\draw [black, fill=black] (4.5,-3) circle [radius=0.1];
	
	\node[left] at (-4,-3) {$\mf{U}_A$};
\end{tikzpicture}
\end{center}
\caption{A $p=4$ sample graph $\mf{G}_A$ and associated unit graph $\mf{U}_A$.}
\label{fig:two}
\end{figure}
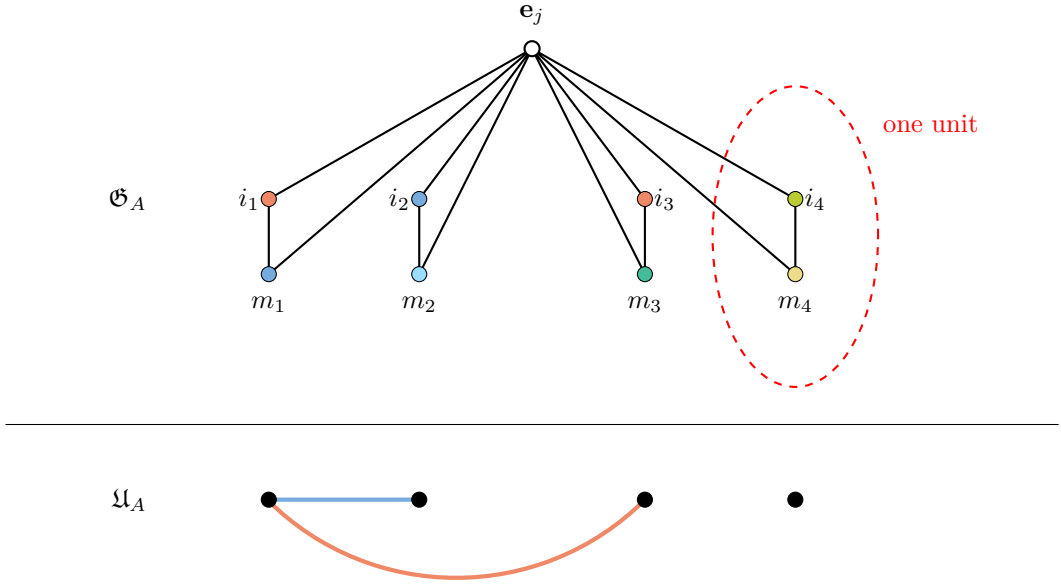

Write $q$ for the number of polychromatic vertices in $\mf{U}_A$. All units have exactly two distinct colors, so there is no more need to separate out $q = n^{\textup{pc}}_2 + n^{\textup{pc}}_3 + n^{\textup{pc}}_4$; we distinguish only polychromatic from monochromatic vertices. We make the following claims:

\begin{itemize}
\item \textbf{Claim 1:} For each $p$, there exists $C_p$ with
\begin{equation}
\label{eqn:easy_graphs_claim_1}
	\abs{\E[\mathbf{y}_P]} \leq C_p (N^{1/2})^q
\end{equation}
\item \textbf{Claim 2:} 
\begin{equation}
\label{eqn:easy_graphs_claim_2}
	a(P) \leq k^{2p-q}.
\end{equation}
\end{itemize}
Assuming these claims, we find
\[
	\sup_{P \in \ms{P}} N^{-p} a(P) \abs{\E[\mathbf{y}_P]} \leq C_p \left(\frac{k}{\sqrt{N}}\right)^p 
\]
which, in light of \eqref{eqn:easy_pattern_sum}, finishes the proof.

Now we prove Claim 1. Here the straight estimate for any unit is $\sqrt{N}$, and this is what we take for all the units corresponding to polychromatic vertices in $\mf{U}_A$. In the units corresponding to monochromatic vertices in $\mf{U}_A$, we integrate: If the vertex corresponding a unit is only incident to edges of a single color in $\mf{U}_A$, then there is one color that appears in that unit but in no other unit; we apply \eqref{eqn:star_rules} on this color to get rid of the inner product for this unit, so that the straight-estimate size of the result is $C_p$. If the vertex corresponding to a unit is isolated in $\mf{U}_A$, then each of its colors appears in that unit but in no other unit; then we apply \eqref{eqn:star_rules} as above, so that the straight-estimate size of the result is $C_p$. This proves \eqref{eqn:easy_graphs_claim_1}. 

The proof of Claim 2 here is a near-verbatim copy of the analogous argument in the proof of Lemma \ref{lem:graphs}: The estimate $k^{2p}$ would be trivial and an overcount, so we again take a path graph $\mf{P}_A$ of the unit graph $\mf{U}_A$, each edge of which genuinely reduces the entropy in colors, and which has at least $q$ edges.
\end{proof}


\subsection{Proof of Proposition \ref{prop:undoing_gram_schmidt}}\label{sec:210}

\begin{proof}[Proof of Proposition \ref{prop:undoing_gram_schmidt}]
We give the proof of \eqref{eqn:undoing_gram_schmidt_abs} (the version with absolute values), the proof of \eqref{eqn:undoing_gram_schmidt_noabs} (the version without absolute values) being similar. Whenever $(A_j)_{j=1}^N$ and $(B_j)_{j=1}^N$ are deterministic real numbers we have
\[
	\abs{\max_{j=1}^N \abs{A_j+B_j} - \max_{j=1}^N \abs{A_j}} \leq \max_{j=1}^N \abs{B_j}.
\]
In particular
\begin{equation}
\label{eqn:for_jiang_remark}
\begin{split}
	\abs{\max_{j=1}^N \abs{\sum_{i=1}^k a_i (\sqrt{N}\gamma_{ij})^2 } - \max_{j=1}^N \abs{\sum_{i=1}^k a_i y_{ij}^2}} &\leq \max_{j=1}^N \abs{2\sum_{i=1}^k a_i (\sqrt{N}\gamma_{ij} - y_{ij})y_{ij}} + \max_{j=1}^N \abs{\sum_{i=1}^k a_i(\sqrt{N}\gamma_{ij} - y_{ij})^2} \\
	&\leq \max_{j=1}^N \abs{2\sum_{i=1}^k a_i (\sqrt{N}\gamma_{ij} - y_{ij})y_{ij}} + ak\epsilon_N(k)^2.
\end{split}
\end{equation}
For any $\delta > 0$, \eqref{eqn:jiang_epsilon_bound} gives $\lim_{N \to \infty} \P(ak\epsilon_N(k)^2 > \delta) = 0$, which handles the second term on the right-hand side. The first term on the right-hand side requires more work. It is naively bounded by $2a\epsilon_N(k) \max_{j=1}^N \sum_{i=1}^k \abs{y_{ij}}$, but this is essentially order $k\sqrt{k/N}$ from \eqref{eqn:jiang_epsilon_bound}, so we need something more sophisticated to take $k$ almost order $N^{1/2}$ as desired. The idea will be to replace $\sqrt{N}\gamma_{ij} - y_{ij}$ with $\widetilde{\Delta}_{ij}$, up to small error; these two quantities have the same naive size, but the gain is that for $\widetilde{\Delta}_{ij}$ one can apply Proposition \ref{prop:cumulant_expansion} (we do not know how to prove the analogue of Proposition \ref{prop:cumulant_expansion} with $\widetilde{\Delta}_{ij}$ replaced with $\sqrt{N}\gamma_{ij} - y_{ij}$).

Since
\[
	\sqrt{N}\gamma_{ij} - y_{ij} = -\Delta_{ij} + (y_{ij}-\Delta_{ij})\left(\sqrt{\frac{N}{\|\mathbf{w}_i\|^2}} - 1\right)
\]
(even when $i = 1$), we have
\begin{align*}
	\abs{\sum_{i=1}^k a_i(\sqrt{N}\gamma_{ij} - y_{ij})y_{ij}} &\leq \abs{\sum_{i=1}^k a_i \Delta_{ij}y_{ij}} + ak \left(\max_{i=1}^k L_i\right)  \left( \max_{i=1}^k \vertiii{\mathbf{y}_i} \right)  \left[ \left( \max_{i=1}^k \vertiii{\mathbf{y}_i} \right) + \left( \max_{i=1}^k \vertiii{\bm{\Delta}_i} \right) \right]
\end{align*}
A union bound and standard Gaussian tail estimates give
\begin{equation}
\label{eqn:jiang_gaussian_bound}
	\P\left(\max_{i=1}^N \vertiii{\mathbf{y}_i} \geq s\right) \leq \frac{N^2}{s\sqrt{2\pi}} e^{-s^2/2},
\end{equation}
which tends to zero if we take, say, $s = \log N$. For any $\delta > 0$, by our choice of $k$ and Proposition \ref{prop:jiang_bounds}, we thus have
\begin{align*}
	&\P\left(ak\left(\max_{i=1}^k L_i\right) \left(\max_{i=1}^N \vertiii{\mathbf{y}_i}\right) \left[ \left(\max_{i=1}^N \vertiii{\bm{\Delta}_i}\right) + \left(\max_{i=1}^N \vertiii{\mathbf{y}_i}\right) \right] > \delta \right) \\
	&\leq \P\left(\max_{i=1}^k L_i > \frac{\log N}{\sqrt{N}}\right) + \P\left(\max_{i=1}^N \vertiii{\mathbf{y}_i} \geq \log N \right) + \P\left(\max_{i=1}^k \vertiii{\bm{\Delta}_i} > \frac{\log N}{N^{1/4}}\right) \to 0.
\end{align*}
Furthermore, we have
\[
	\abs{\sum_{i=1}^k a_i \Delta_{ij}y_{ij}} \leq \abs{\sum_{i=1}^k a_i \widetilde{\Delta}_{ij}y_{ij}} + \abs{\sum_{i=1}^k a_i (\Delta_{ij}-\widetilde{\Delta}_{ij}) y_{ij}}.
\]
If 
\[
	\mc{E}_N := \left\{\max_{i=1}^k \max_{j=1}^N \abs{\Delta_{ij} - \widetilde{\Delta}_{ij}} \leq \frac{1}{N^{1/2}} \right\},
\]
then
\begin{align*}
	\P\left(\max_{j=1}^N \abs{\sum_{i=1}^k a_i (\Delta_{ij} - \widetilde{\Delta}_{ij}) y_{ij}} \geq \delta\right) \leq \P(akN^{-1/2}\max_{i=1}^k \vertiii{\mathbf{y}_i} > \delta) + \P(\mc{E}_N^c).
\end{align*}
The first term on the right-hand side tends to zero for any $\delta > 0$ by \eqref{eqn:jiang_gaussian_bound}, and the second term tends to zero by Proposition \ref{prop:deltadeltatilde}.

Finally, an application of Proposition \ref{prop:cumulant_expansion} finishes the proof.
\end{proof}

\begin{rem}
\label{rem:jiang}
As mentioned above, adapting arguments of Jiang \cite{Jia2005} gives an easier proof of Theorem \ref{thm:main} in the restricted case when $k_N = \OO(N^{1/3-\delta})$, which we now sketch, ignoring things like log factors and ``with high probability.''

It suffices to prove the analogue of Proposition \ref{prop:undoing_gram_schmidt} with $k_N = \OO(N^{1/3-\delta})$. As explained just following \eqref{eqn:for_jiang_remark}, the difficult task is to show
\[
	\max_{j=1}^N \abs{2\sum_{i=1}^k a_i(\sqrt{N}\gamma_{ij} - y_{ij})y_{ij}} \to 0 \quad \text{in probability.}
\]
The easy route is to move the absolute value inside, ignoring cancellations. This turns out to work for $k = \OO(N^{1/3-\delta})$, as we are currently explaining, but to get up to $k = \OO(N^{1/2-\delta})$ one needs to take the cancellations into account; in fact, this is the main novelty of Section \ref{sec:undoing_gram_schmidt}. Taking the absolute values inside, one bounds this by
\[
	\epsilon_N(k) \max_{i=1}^N \sum_{j=1}^k \abs{y_{ij}} \lesssim k \epsilon_N(k).
\]
Since Jiang's work essentially shows $\epsilon_N(k) \lesssim \sqrt{\frac{k}{N}}$ (see \eqref{eqn:jiang_epsilon_bound}), we want $k\sqrt{\frac{k}{N}}$ to tend to zero, which leads us to a threshold $k \ll N^{1/3}$.
\end{rem}


\subsection{The Weibull case.}\
\label{subsec:weibull}

The Weibull results we claimed in \eqref{eqn:main_aneg}, i.e. the case of $\max_{i=1}^N \ip{\mathbf{u}_i,A_N\mathbf{u}_i}$ with $A_N = \diag(a_1,\ldots,a_k)$ and $a_1,\ldots,a_k < 0$, do not directly fit into the framework elsewhere in this section, since the scaling is different. (Up to log factors, the maximum of $N$ independent $\chi^2$ variables with $k$ degrees of freedom is order one, but the minimum is order $N^{-2/k}$.) In this section we give the necessary variant on our arguments. To simplify the statement and proof, we discuss minima and take all $a_i$ positive, rather than maxima with all $a_i$ negative. The proof of \eqref{eqn:main_aneg} is just a combination of the following lemma with the Gaussian computations given in Lemma \ref{lem:rankk_allneg}.

\begin{lem}
\label{lem:weibull}
Fix $k$, and fix $a_1, \ldots, a_k$ all positive. Then 
\[
    N^{2/k}\left[ \min_{j=1}^N \left( \sum_{i=1}^k a_i (\sqrt{N}\gamma_{ij})^2\right) - \min_{j=1}^N \left( \sum_{i=1}^k a_i y_{ij}^2 \right)\right] \overset{N \to \infty}{\to} 0 \quad \text{in probability.}
\]
\end{lem}

This result follows from Lemmas \ref{lem:weibull_hard_part} and \ref{lem:weibull_easy_part} below. The former is longer than the latter, but the latter relies in a crucial way on the former. 

\begin{lem}
\label{lem:weibull_hard_part}
For fixed $k$, we have
\[
	N^{2/k} \left[ \min_{j=1}^N \left(\sum_{i=1}^k a_i w_{ij}^2 \right) - \min_{j=1}^N \left( \sum_{i=1}^k a_iy_{ij}^2\right)\right] \overset{N \to \infty}{\to} 0 \quad \text{in probability}.
\]
\end{lem}
\begin{proof}
Since $\mathbf{w}_1 = \mathbf{y}_1$, when $k = 1$ there is nothing to prove. In the remainder we assume $k \geq 2$.

Consider the good event 
\[
	\mc{E}_N = \{\text{for all } j = 2, \ldots, k \text{ and all } i = 1, \ldots, j-1, \abs{\ip{\mathbf{y}_j, \bm{\gamma}_i}} \leq (\log N)^4 \} \cap \left\{N^{2/k} \min_{j=1}^N \left( \sum_{i=1}^k a_i y_{ij}^2 \right) \leq  \log N\right\}.
\]
We make two claims: First, on the event $\mc{E}_N$, for every $\epsilon > 0$ there exists $N_0(\epsilon)$ such that for all $N \geq N_0(\epsilon)$ we have
\begin{equation}
\label{eqn:weibull_in_prob}
	\abs{N^{2/k} \left[ \min_{j=1}^N \left(\sum_{i=1}^k a_i w_{ij}^2 \right) - \min_{j=1}^N \left( \sum_{i=1}^k a_iy_{ij}^2\right)\right]} \leq \epsilon.
\end{equation}
Second, $\lim_{N \to \infty} \P(\mc{E}_N) = 1$. Notice that these two claims suffice to prove the lemma. 

We prove the first claim, and assume tacitly in the following that we work on $\mc{E}_N$. Throughout, we write $C_{k,a}$ for positive constants which depend on $(a_1,\ldots,a_k)$ and/or on $k$, but not on $N$, and which may change from line to line; we write $C$ for absolute constants which change from line to line. Write $j_\ast$ for the (possibly non-unique) minimizing index for the $y$ problem, i.e., $\min_{j=1}^N \sum_{i=1}^k a_i y_{ij}^2 = \sum_{i=1}^k a_i y_{ij_\ast}^2$. From the second set in the definition of $\mc{E}_N$  and the 
strict positivity of all $a_i$ this implies
\begin{equation}
\label{eqn:weibully}
	\max_{i=1}^k \abs{y_{ij_\ast}} \leq C_{k,a} \frac{(\log N)^C}{N^{1/k}}.
\end{equation}
From the first set in the definition of $\mc{E}_N$, we find
\begin{equation}
\label{eqn:weibull_delta_initial}
	\abs{\Delta_{ij_\ast}} \leq \frac{1}{\sqrt{N}} \sum_{\ell=1}^{i-1} \abs{\ip{\mathbf{y}_i,\bm{\gamma}_\ell} w_{\ell j_\ast}} \leq \frac{(\log N)^C}{\sqrt{N}} \sum_{\ell=1}^{i-1} \abs{w_{\ell j_\ast}},
\end{equation}
and hence, via \eqref{eqn:weibully},
\[
	\abs{w_{ij_\ast}} \leq \abs{y_{ij_\ast}} + \abs{\Delta_{ij_\ast}} \leq C_a \frac{(\log N)^C}{N^{1/k}} + \frac{(\log N)^C}{\sqrt{N}} \sum_{\ell=1}^{i-1} \abs{w_{\ell j_\ast}}.
\]
From the triangular structure of this inequality, we find
\[
	\max_{i=1}^k \abs{w_{i j_\ast}} \leq C_{k,a} \frac{(\log N)^C}{N^{1/k}}.
\]
Plugging this back into \eqref{eqn:weibull_delta_initial} yields
\begin{equation}
\label{eqn:weibull_delta}
	\max_{i=1}^k \abs{\Delta_{ij_\ast}} \leq C_{k,a} \frac{(\log N)^{C}}{N^{\frac{1}{2}+\frac{1}{k}}}.
\end{equation}
We use \eqref{eqn:weibull_delta} along with \eqref{eqn:weibully} to estimate
\begin{align}
	N^{2/k} \left[ \min_{j=1}^N \left(\sum_{i=1}^k a_i w_{ij}^2 \right) - \min_{j=1}^N \left( \sum_{i=1}^k a_iy_{ij}^2\right)\right] &\leq N^{2/k} \left[ \sum_{i=1}^k a_i w_{ij_\ast}^2 - \sum_{i=1}^k a_iy_{ij_\ast}^2\right]  \nonumber  \\
	&= N^{2/k} \left[ \sum_{i=1}^k a_i (-2y_{ij_\ast} \Delta_{ij_\ast} + \Delta_{ij_\ast}^2)\right]  \nonumber  \\
	&\leq C_{k,a} N^{2/k}(\log N)^C \left( \frac{1}{N^{\frac{1}{2}+\frac{2}{k}}}+ \frac{1}{N^{1+\frac{2}{k}}}\right).
	\label{eqn:upper}
\end{align}
This tends to zero, which is a one-sided inequality towards \eqref{eqn:weibull_in_prob} (in the sense that there is no absolute value on the left-hand side of \eqref{eqn:upper}). For the other side, we write $j_\#$ for the (possibly non-unique) minimizing index for the $w$ problem, i.e., $\min_{j=1}^N \sum_{i=1}^k a_i w_{ij}^2 = \sum_{i=1}^k a_i w_{ij_\#}^2$. To estimate this minimum from above we use~\eqref{eqn:upper} together with the second set in the definition of $\mc{E}_N$. This implies
\begin{equation}
\label{eqn:weibullw}
	\max_{i=1}^k \abs{w_{ij_\#}} \leq C_a \frac{(\log N)^C}{N^{1/k}},
\end{equation}
thus
\begin{align*}
	\abs{\Delta_{ij_\#}} &\leq C_{k,a} \frac{(\log N)^{C}}{N^{\frac{1}{2}+\frac{1}{k}}} \quad \text{for each} \quad i = 1, \ldots, k, \\
	\abs{y_{ij_\#}} &\leq \abs{w_{ij_\#}} + \abs{\Delta_{ij_\#}} \leq C_{k,a} \frac{(\log N)^C}{N^{1/k}} \quad \text{for each} \quad i = 1, \ldots, k.
\end{align*}
Like before, these estimates give
\begin{align*}
	N^{2/k} \left[ \min_{j=1}^N \left(\sum_{i=1}^k a_i y_{ij}^2 \right) - \min_{j=1}^N \left( \sum_{i=1}^k a_iw_{ij}^2\right)\right] &\leq N^{2/k} \left[ \sum_{i=1}^k a_i y_{ij_\#}^2 - \sum_{i=1}^k a_iw_{ij_\#}^2\right] \\
	&= N^{2/k} \left[ \sum_{i=1}^k a_i (2y_{ij_\#} \Delta_{ij_\#} - \Delta_{ij_\#}^2)\right] \\
	&\leq C_{k,a} N^{2/k} (\log N)^C \left( \frac{1}{N^{\frac{1}{2}+\frac{2}{k}}}+ \frac{1}{N^{1+\frac{2}{k}}}\right),
\end{align*}
which gives the other side of, and therefore completes the proof of, \eqref{eqn:weibull_in_prob}.

It remains only to show $\P(\mc{E}_N) \to 1$. We have 
\[
\P\Big(N^{2/k} \min_{j=1}^N \sum_{i=1}^k a_i y_{ij}^2 \leq \log N\Big) \to 1,
\]
 since the Gaussian computations in Lemma \ref{lem:rankk_allneg} show that the random variable $N^{2/k} \min_{j=1}^N \sum_{i=1}^k a_i y_{ij}^2$ has a limit in distribution. For the inner products, we note that on the good event $\mc{E}_{k,N}$ from \eqref{eqn:delta_replacement_good_event}, we have $\|\mathbf{w}_i\|^{-1} \leq N^{-1/2}(1+L_i) \leq CN^{-1/2}$ and
 \[
 	\abs{\ip{\mathbf{y}_j,\bm{\gamma}_i}} = \frac{1}{\|\mathbf{w}_i\|} \abs{\ip{\mathbf{y}_j,\mathbf{w}_i}} \leq CN^{-1/2}(\abs{\ip{\mathbf{y}_j,\mathbf{y}_i}} + \|\mathbf{y}_j\|\|\bm{\Delta}_i\|) \leq C(\log N)^3
 \]
 for all $j = 2, \ldots, k$ and all $i = 1, \ldots, j-1$; since Lemma \ref{lem:delta_replacement_good_event} showed $\P(\mc{E}_{k,N}) = 1- \oo(1)$, this completes the proof that $\P(\mc{E}_N) = 1 - \oo(1)$. 
\end{proof}

\begin{lem}
\label{lem:weibull_easy_part}
For fixed $k$, we have
\[
	N^{2/k} \left[ \min_{j=1}^N \left(\sum_{i=1}^k a_i (\sqrt{N} \gamma_{ij})^2 \right) - \min_{j=1}^N \left( \sum_{i=1}^k a_iw_{ij}^2\right)\right] \overset{N \to \infty}{\to} 0 \quad \text{in probability}.
\]
\end{lem}
\begin{proof}
Using the positivity of $a_i$'s, one can compute
\begin{align*}
	\abs{N^{2/k} \left[ \min_{j=1}^N \left(\sum_{i=1}^k a_i (\sqrt{N} \gamma_{ij})^2 \right) - \min_{j=1}^N \left( \sum_{i=1}^k a_iw_{ij}^2\right)\right]} &\leq \left( \max_{i=1}^k \abs{ \frac{N}{\|\mathbf{w}_i\|^2}-1 } \right) N^{2/k} \min_{j=1}^N \left(\sum_{i=1}^k a_i w_{ij}^2\right) \\
	&\eqdef \left( \max_{i=1}^k \abs{ \frac{N}{\|\mathbf{w}_i\|^2}-1 } \right) X_N.
\end{align*}
Since Lemma \ref{lem:rankk_allneg} below shows that the random variables $Y_N \defeq N^{2/k} \min_{j=1}^N \left(\sum_{i=1}^k a_i y_{ij}^2 \right)$ have a limit in distribution, and Lemma \ref{lem:weibull_hard_part} showed that $Y_N - X_N$ tends to zero in probability, the random variables $X_N$ also have a limit in distribution. Thus it suffices to show $\max_{i=1}^k \abs{N\|\mathbf{w}_i\|^{-2} - 1}$ tends to zero in probability, and since the maximum has only $k$ terms, we can show that $\abs{ N\|\mathbf{w}_i\|^{-2} - 1}$ tends to zero in probability for any fixed $i$. But the variable $\|\mathbf{w}_i\|^2$ is distributed as a $\chi^2$ with $N-i+1$ degrees of freedom (see, e.g., the proof of \cite[Lemma 3.6]{Jia2005}); hence the second moment of $N\|\mathbf{w}_i\|^{-2} - 1$ has an explicit form, which tends to zero for every fixed $i$.
\end{proof}


\section{Gaussian Computations: Fixed Rank}
\label{sec:gaussian_fixed}

The punchline of Section \ref{sec:undoing_gram_schmidt} is that, in order to compute certain extremal statistics related to Haar orthogonal entries, it suffices to compute the same extremal statistics with the Haar entries replaced by rescaled independent Gaussians.  
  These Gaussian calculations will be done slightly differently depending on the signature structure of $A$,
  and by orthogonal invariance we may assume that $A$ is already diagonal.
 We thus start our proof of Theorem \ref{thm:main_fixed} by making Gaussian computations in the following three lemmas, which handle, respectively, the case where all eigenvalues of $A$ are positive; the case where some are positive and some are negative; and the case where all are negative. 

\begin{lem}
\label{lem:rankk_allpos}
Fix $k \in \N$ and real numbers $a_1, \ldots, a_k$, which are all positive, and write $A_N = \diag(a_1, \ldots, a_k, 0, \ldots, 0)$. With the same constants defined in \eqref{eqn:defmult}, \eqref{eqn:defc}, and \eqref{eqn:defcbar}, and with $\Lambda$ a Gumbel-distributed random variable, we have
\[
	\frac{1}{2a} \max_{i=1}^N \ip{\mathbf{y}_i, A_N \mathbf{y}_i} - \log N + \left(1-\frac{m}{2}\right) \log \log N + c_m(a_1, \ldots, a_k) \overset{N \to \infty}{\to} \Lambda \quad \text{in distribution}.
\]
\end{lem}
\begin{proof}
Although the cases of different multiplicity $m$ (the number of $a_i$'s equaling the largest value, $a$) could be treated simultaneously, for pedagogical reasons we treat first the case $m = k$, then the case $m = 1$, then the case of intermediate $m$. 

The case $m = k$, i.e. $a_1 = \cdots = a_k = a$, is easy; the result is then equivalent to the maximal fluctuations of $\chi^2_k$ variables, i.e., to the statement $\frac{1}{2}\max_{i=1}^N \{X_i\} - \log N + (1-k/2) \log \log N + \log \Gamma(k/2)\to \Lambda$ where the $X_i$'s are i.i.d. $\chi^2_k$'s, and this is a very classical result (see, e.g., \cite[Section 3.4]{EmbKluMik1997}). 

We now assume $m = 1$, i.e., there is a unique largest $a_i$, which we assume without loss of generality to be the last one $a_k$. Momentarily writing $y_j = y_{j1}$, we first need large-$t$ asymptotics for $\P(\sum_{i=1}^k a_iy_i^2 \geq t)$. Write $E_t = \{(z_1,\ldots,z_k) \in \R^k : \sum_j a_j z_j^2 \leq t\}$ for the relevant ellipsoid and its interior; then this probability is of course the integral of Gaussian measure over $E_t^c$, which we estimate by working $k$-dimensional hyperspherical coordinates $(r, \varphi_1, \ldots, \varphi_{k-1})$, with $\varphi_j = \arccot\frac{z_j}{\sqrt{z_k^2+z_{k-1}^2+\cdots+z_{j+1}^2}}$ for $j = 1, \ldots, k-2$ and $\varphi_{k-1} = 2\arccot\frac{z_{k-1}+\sqrt{z_k^2+z_{k-1}^2}}{z_k}$. In these coordinates, we have
\begin{equation}
\label{eqn:sphericalcoords}
\begin{split}
	\P\left(\sum_{i=1}^k a_iy_i^2 \geq t\right) &= \frac{1}{(2\pi)^{k/2}} \int_{S_\varphi} \left( \int_{r_t(\vec{\varphi})}^\infty e^{-\frac{r^2}{2}} r^{k-1} \diff r \right) \sin^{k-2}(\varphi_1) \sin^{k-3}(\varphi_2) \cdots \sin(\varphi_{k-2}) \diff \vec{\varphi} \\
	&=  \frac{t^{\frac{k}{2}}}{(2\pi)^{k/2}} \int_{S_\varphi} \left( \int_1^\infty e^{-t\frac{u^2}{2f(\vec{\varphi})}} u^{k-1} \diff u \right) f(\vec{\varphi})^{-\frac{k}{2}} \sin^{k-2}(\varphi_1) \sin^{k-3}(\varphi_2) \cdots \sin(\varphi_{k-2}) \diff \vec{\varphi} \\
	&= \frac{2t^{\frac{k}{2}}}{(2\pi)^{k/2}} \int_{[0,\pi]^{k-1}} \left( \int_1^\infty e^{-t\frac{u^2}{2f(\vec{\varphi})}} u^{k-1} \diff u \right) f(\vec{\varphi})^{-\frac{k}{2}} \sin^{k-2}(\varphi_1) \sin^{k-3}(\varphi_2) \cdots \sin(\varphi_{k-2}) \diff \vec{\varphi}
\end{split}
\end{equation}
where $\diff \vec{\varphi} = \prod_{j=1}^{k-1} \diff \varphi_j$, where $S_\varphi = \{(\varphi_1, \ldots, \varphi_{k-1}) : 0 \leq \varphi_j \leq \pi \text{ for } j=1, \ldots, k-2, \text{ and } 0 \leq \varphi_{k-1} \leq 2\pi\}$, and where 
\begin{align*}
	r_t(\vec{\varphi}) =& \, r_t(\varphi_1, \ldots, \varphi_{k-1}) = \sqrt{\frac{t}{f(\vec{\varphi})}}, \\
	f(\vec{\varphi}) =& \, a_1\cos^2(\varphi_1) + a_2\sin^2(\varphi_1)\cos^2(\varphi_2) + a_3\sin^2(\varphi_1)\sin^2(\varphi_2)\cos^2(\varphi_3) \\
	&+ \cdots + a_{k-1}\sin^2(\varphi_1)\cdots \sin^2(\varphi_{k-2})\cos^2(\varphi_{k-1}) + a_k\sin^2(\varphi_1) \cdots \sin^2(\varphi_{k-2})\sin^2(\varphi_{k-1}).
\end{align*}
(The last equality in \eqref{eqn:sphericalcoords} holds since $f(\vec{\varphi})$ is $\pi$-periodic in $\varphi_{k-1}$, which does not appear anywhere else, i.e., $\int_{[0,\pi]^{k-2} \times [0,\pi]} = \int_{[0,\pi]^{k-2} \times [\pi,2\pi]}$.)

The right-hand side of \eqref{eqn:sphericalcoords} is now amenable to the Laplace method. The integral in $u$ is slightly unusual, since the exponential is maximized at the endpoint $u = 1$; this variant on the Laplace method gives, for any $\vec{\varphi}$ with $f(\vec{\varphi}) \neq 0$ (i.e., Lebesgue-a.e. $\vec{\varphi}$),
\[
	\int_1^\infty e^{-t\frac{u^2}{2f(\vec{\varphi})}} u^{k-1} \diff u \sim f(\vec{\varphi}) t^{-1} e^{-\frac{t}{2f(\vec{\varphi})}},
\]
which leaves us with an integral of the form
\[
	\int_{[0,\pi]^{k-1}} e^{tg(\vec{\varphi})} h(\vec{\varphi}) \diff \vec{\varphi},
\]
with $g(\vec{\varphi}) = -\frac{1}{2f(\vec{\varphi})}$ and $h(\vec{\varphi}) = f(\vec{\varphi})^{1-\frac{k}{2}} \sin^{k-2}(\varphi_1) \sin^{k-3}(\varphi_2) \cdots \sin(\varphi_{k-2})$. Because of how we ordered the $a_i$'s, the function $g$ takes a unique maximum on $[0,\pi]^{k-1}$ at $\vec{\varphi}^\ast = (\pi/2,\ldots,\pi/2)$, where it takes the value $-\frac{1}{2a_k}$, and where it has Hessian $\left. \nabla^2 g(\vec{\varphi})\right|_{\vec{\varphi} = \vec{\varphi}^\ast} = \frac{1}{a_k^2} \diag(a_1-a_k, a_2-a_k, \ldots, a_{k-1}-a_k)$. Since $h(\vec{\varphi}^\ast) = f(\vec{\varphi}^\ast)^{1-\frac{k}{2}} = a_k^{1-\frac{k}{2}}$, we obtain
\begin{align*}
	\P\left(\sum_{i=1}^k a_iy_i^2 \geq t\right) &\sim \frac{2t^{\frac{k}{2}}}{(2\pi)^{k/2}} \frac{1}{t} \left(\frac{2\pi}{t}\right)^{\frac{k-1}{2}} a_k^{1-\frac{k}{2}} \frac{1}{\sqrt{\left(\frac{1}{a_k^2}\right)^{k-1} \prod_{j=1}^{k-1} (a_k-a_j)}} e^{-\frac{t}{2a_k}} \\
	&= \sqrt{\frac{2}{\pi}} \frac{a_k^{k/2}}{\sqrt{\prod_{j=1}^{k-1}(a_k-a_j)}} t^{-\frac{1}{2}} e^{-\frac{t}{2a_k}}.
\end{align*}
Now we will use the following classical result in extreme value theory (see, e.g., \cite[Proposition 3.3.28]{EmbKluMik1997}): Suppose that $F$ and $G$ are distribution functions with infinite right endpoint, i.e. $\max(F(x),G(x)) < 1$ for every $x$, such that for some sequences $(c_N)_{N=1}^\infty$, $(d_N)_{N=1}^\infty$ with $c_N > 0$ we have 
\begin{equation}
\label{eqn:gumbeltailF}
	\lim_{N \to \infty} F^N(c_Nx+d_N) = \Lambda(x).
\end{equation}
Then
\begin{equation}
\label{eqn:gumbeltailG}
	\lim_{N \to \infty} G^N(c_Nx+d_N) = \Lambda(x+b)
\end{equation}
if and only if
\[
	\lim_{x \to +\infty} \frac{1-F(x)}{1-G(x)} = e^b.
\]
We apply this with $G$ the distribution function of $\sum_{i=1}^k a_iy_i^2$ and $F$ the distribution function of $a_k$ times a $\chi^2_1$ variable, i.e., 
\[
	F(x) = \int_{-\sqrt{\frac{x}{a_k}}}^{\sqrt{\frac{x}{a_k}}} \frac{e^{-\frac{y^2}{2}}}{\sqrt{2\pi}} \diff y \quad \text{which has} \quad 1-F(x) \sim \sqrt{\frac{2a_k}{\pi}} x^{-1/2} e^{-\frac{x}{2a_k}}.
\]
Then
\[
	\lim_{x \to +\infty} \frac{1-F(x)}{1-G(x)} = \sqrt{\prod_{j=1}^{k-1} \left(1-\frac{a_j}{a_k}\right)}
\]
and since \eqref{eqn:gumbeltailF} holds with $c_N = 2a_k$ and $d_N = 2a_k(\log N - \frac{1}{2}\log\log N - \frac{1}{2}\log \pi)$ (see, e.g., \cite[Section 3.4]{EmbKluMik1997}), we obtain
\[
	\P\left(\max_{j=1}^N \left\{\sum_{i=1}^k a_i y_{ij}^2\right\} \leq c_N\left(x-\frac{1}{2}\log\left(\prod_{j=1}^{k-1}\left(1-\frac{a_j}{a_k}\right)\right)\right)+d_N\right) \to \Lambda(x)
\]
which is what we want.

Finally we handle the case of $1 < m < k$, and without loss of generality we assume the $a_i$'s are ordered as $a_1 \leq \cdots \leq a_{k-m} < a_{k-m+1} = \cdots = a_k = a$. Notice that in this case
\[
	f(\vec{\varphi}) = a_1\cos^2(\varphi_1) + a_2 \sin^2(\varphi_1)\cos^2(\varphi_2) + \cdots + a_{k-m}\sin^2(\varphi_1)\cdots\cos^2(\varphi_{k-m}) + a\sin^2(\varphi_1)\sin^2(\varphi_2) \cdots \sin^2(\varphi_{k-m})
\]
i.e., $f(\vec{\varphi})$ actually does not depend on $\varphi_{k-m+1}, \ldots, \varphi_{k-1}$. Thus we can integrate these out first in \eqref{eqn:sphericalcoords}, and we pick up a factor
\[
	\int_{[0,\pi]^{m-1}} \sin^{m-2}(\varphi_{k-m+1}) \cdots \sin(\varphi_{k-2}) \prod_{j=k-m+1}^{k-1} \diff \varphi_j = \prod_{j=0}^{m-2} \left[ \int_0^\pi \sin^j(x) \diff x \right] = \prod_{j=0}^{m-2} \frac{\sqrt{\pi}\Gamma\left(\frac{1+j}{2}\right)}{\Gamma \left(1+\frac{j}{2}\right)} = \pi^{\frac{m-1}{2}} \frac{\Gamma\left(\frac{1}{2}\right)}{\Gamma\left(\frac{m}{2}\right)}
\]
so that, writing $\underline{\varphi} = (\varphi_1, \ldots, \varphi_{k-m})$ and $\diff \underline{\varphi} = \prod_{j=1}^{k-m} \diff \varphi_j$, and then using the Laplace method on the remaining $k-m$ variables just as above, we find
\begin{align*}
	&\P\left(\sum_{i=1}^k a_iy_i^2 \geq t\right) \\
	&= \frac{2t^{\frac{k}{2}}}{(2\pi)^{k/2}} \frac{\pi^{\frac{m}{2}}}{\Gamma\left(\frac{m}{2}\right)} \int_{[0,\pi]^{k-m}} \left( \int_1^\infty e^{-t\frac{u^2}{2f(\underline{\varphi})}} u^{k-1} \diff u \right) f(\underline{\varphi})^{-\frac{k}{2}} \sin^{k-2}(\varphi_1) \sin^{k-3}(\varphi_2) \cdots \sin^{m-1}(\varphi_{k-m}) \diff \underline{\varphi} \\
	&\sim \frac{2t^{\frac{k}{2}}}{t(2\pi)^{k/2}} \frac{\pi^{\frac{m}{2}}}{\Gamma\left(\frac{m}{2}\right)} \left(\frac{2\pi}{t}\right)^{\frac{k-m}{2}} a^{1-\frac{k}{2}} \frac{1}{\sqrt{\left(\frac{1}{a^2}\right)^{k-m} \prod_{j=1}^{k-m}(a-a_j)}} e^{-\frac{t}{2a}} = \frac{2^{1-\frac{m}{2}}}{\Gamma\left(\frac{m}{2}\right)} \frac{a^{1+\frac{k}{2}-m}}{\sqrt{\prod_{j=1}^{k-m} (a-a_j)}} t^{\frac{m}{2}-1} e^{-\frac{t}{2a}}.
\end{align*}
We use the same tail-equivalence result from extreme value theory, still with $G$ the distribution function of $\sum_{i=1}^k a_i y_i^2$, but now with $F$ the distribution function of $a$ times a $\chi^2_m$ variable, i.e., 
\[
	F(x) = \frac{\gamma\left(\frac{m}{2},\frac{x}{2a}\right)}{\Gamma\left(\frac{m}{2}\right)} \quad \text{which has} \quad 1-F(x) \sim \frac{(2a)^{1-\frac{m}{2}}}{\Gamma\left(\frac{m}{2}\right)} t^{\frac{m}{2}-1} e^{-\frac{t}{2a}}
\]
which is known to satisfy \eqref{eqn:gumbeltailF} with $c_N = 2a$ and $d_N = 2a(\log N + (\frac{m}{2}-1)\log\log N - \log \Gamma(\frac{m}{2}))$ (see, e.g., \cite[Section 3.4]{EmbKluMik1997}). We obtain
\[
	\P\left( \max_{j=1}^N \left\{ \sum_{i=1}^k a_i y_{ij}^2\right\} \leq c_N\left(x-\frac{1}{2}\log \left( \prod_{j=1}^{k-m} \left(1-\frac{a_j}{a}\right) \right) \right) + d_N \right) \to \Lambda(x),
\]
which is what we wanted.
\end{proof}

\begin{lem}
\label{lem:rankk_somepossomeneg}
Fix $k \in \N$ and nonzero real numbers $a_1, \ldots, a_k$, at least one of which is positive, and write $A_N = \diag(a_1,\ldots,a_k,0,\ldots,0)$. With the same constants defined in \eqref{eqn:defmult}, \eqref{eqn:defc}, and \eqref{eqn:defcbar}, and with $\Lambda$ a Gumbel-distributed random variable, we have
\[
	\frac{1}{2a} \max_{i=1}^N \ip{\mathbf{y}_i, A_N \mathbf{y}_i} - \log N + \left(1-\frac{m}{2}\right) \log \log N + c_m(a_1, \ldots, a_k) \overset{N \to \infty}{\to} \Lambda \quad \text{in distribution}
\]
and
\[
	\frac{1}{2a} \max_{i=1}^N \abs{\ip{\mathbf{y}_i, A_N \mathbf{y}_i}} - \log N + \left(1-\frac{\max(m_+,m_-)}{2}\right) \log\log N + c^\ast_{m_+,m_-}(a_1, \ldots, a_k) \overset{N \to \infty}{\to} \Lambda \quad \text{in distribution.}
\]
\end{lem}
\begin{proof}
As before, we split the proof into cases for pedagogical reasons. In the first half we treat the case without absolute values (when we assume without loss of generality the ordering $a_1 \leq \cdots \leq a_k$), treating first the case $m = 1$ and then the case $m > 1$. In the second half we treat the case with absolute values, when we assume without loss of generality (up to flipping the signs of all the $a_i$'s) the ordering
\[
	\underbrace{a_k = a_{k-1} = \cdots = a_{k-m_+ + 1}}_{m_+ \text{ many}} = a, \quad \underbrace{a_{k-m_+} = \cdots = a_{k-m+1}}_{m_- \text{ many}} = -a, \quad a > \max(0,\max(\abs{a_1}, \ldots, \abs{a_{k-m}})).
\]
Here we treat first the case when $m_\ast \defeq m_+ + m_- = 1$, then the case $m_\ast > 1$. 

The case without absolute values when $m = 1$ is very similar to the case when all $a_i$'s are positive; the difference here is that the shape $\{(z_1, \ldots, z_k) \in \R^k : \sum_j a_j z_j^2 \leq t\}$ is no longer an ellipsoid, but rather a type of conic section. Its complement can be parametrized in our $k$-dimensional hyperspherical coordinates as
\[
	\left\{(r, \varphi_1, \ldots, \varphi_{k-1}) : f(\vec{\varphi}) > 0 \text{ and }  r > \sqrt{t/f(\vec{\varphi})} \right\}.
\]
(This is also true in the case when all $a_i$'s are positive, but there the condition $\{f(\vec{\varphi}) > 0\}$ is satisfied for every $\vec{\varphi}$, which is no longer true when some $a_i$'s are negative.) Then 
\[
	\P\left(\sum_{i=1}^k a_i y_i^2 \geq t\right) = \frac{2t^{\frac{k}{2}}}{(2\pi)^{k/2}} \int_{G_{\vec{\varphi}}} \left( \int_1^\infty e^{-t\frac{u^2}{2f(\vec{\varphi})}} u^{k-1} \diff u \right) f(\vec{\varphi})^{-\frac{k}{2}} \sin^{k-2}(\varphi_1) \sin^{k-3}(\varphi_2) \cdots \sin(\varphi_{k-2}) \diff \vec{\varphi}
\]
using the crucial set $G_{\vec{\varphi}} = \{\varphi_1, \ldots, \varphi_{k-1} \in [0,\pi]^{k-1} : f(\vec{\varphi}) > 0\}$. But, ordering $a_1 \leq \ldots \leq a_k$ without loss of generality as above, notice that this set is open and contains $\vec{\varphi}^\ast = (\pi/2,\ldots,\pi/2)$; hence the Laplace method, expanding in a small neighborhood about this point, works in the same way as before.

Next we consider the case without absolute values but with $m > 1$. We still have the crucial property that $f(\vec{\varphi})$ does not depend on $\varphi_{k-m+1}, \ldots, \varphi_{k-1}$, so we can write $f(\vec{\varphi}) = f(\underline{\varphi})$, naturally define $G_{\underline{\varphi}}$, and then note that $G_{\vec{\varphi}}$ has the product structure $G_{\vec{\varphi}} = G_{\underline{\varphi}} \times [0,\pi]^{m-1}$; hence we can integrate out $\varphi_{k-m+1}, \ldots, \varphi_{k-1}$ cleanly and then proceed in the same way as before.

Next we consider the case with absolute values and $m_\ast = 1$. Up to flipping the sign of every $a_i$, we can assume that $m_+ = 1$ and $m_- = 0$. Here the relevant sets are
\[
	\left\{(z_1, \ldots, z_k) \in \R^k : \abs{\sum_{i=1}^k a_i z_i^2} > t\right\} = \{(r,\varphi_1, \ldots, \varphi_{k-1}) : \abs{f(\vec{\varphi})} > 0 \text{ and } r > \sqrt{t/\abs{f(\vec{\varphi})}}\}.
\]
so that
\[
	\P\left(\abs{\sum_{i=1}^k a_iy_i^2} \geq t\right) = \frac{2t^{\frac{k}{2}}}{(2\pi)^{k/2}} \int_{\abs{G_{\vec{\varphi}}}} \left( \int_1^\infty e^{-t\frac{u^2}{2\abs{f(\vec{\varphi})}}} u^{k-1} \diff u \right) \abs{f(\vec{\varphi})}^{-\frac{k}{2}} \sin^{k-2}(\varphi_1) \sin^{k-3}(\varphi_2) \cdots \sin(\varphi_{k-2}) \diff \vec{\varphi}
\]
where $\abs{G_{\vec{\varphi}}} = \{\varphi_1, \ldots, \varphi_{k-1} \in [0,\pi]^{k-1} : \abs{f(\vec{\varphi})} > 0\}$, which still contains the crucial point $\vec{\varphi}^\ast = (\pi/2,\ldots,\pi/2)$ at which the exponential argument $\frac{-1}{2\abs{f(\vec{\varphi})}}$ is uniquely maximized (notice this is still the maximum since we assumed $a_k$ also had the largest absolute value). Furthermore, the function $f$ is positive in a neighborhood of $\vec{\varphi}^\ast$; hence for the purpose of the Laplace method we can drop the absolute values, and obtain the same asymptotics as before.

Finally we consider the case of absolute values and $m_\ast > 1$, i.e. when there are multiple largest $\abs{a_i}$'s, with maximal value labelled $a_\ast$, of which there are $m_+$ many positive ones and $m_-$ many negative ones. We handle this by writing
\[
	\P\left(\abs{\sum_{i=1}^k a_iy_i^2} \geq t\right) = \P\left(\sum_{i=1}^k a_iy_i^2 \geq t\right) + \P\left(\sum_{i=1}^k (-a_i)y_i^2 \geq t\right).
\]
Earlier arguments show that the terms on the right-hand side behave asymptotically, respectively, like $c_+t^{\frac{m_+}{2}-1}e^{-\frac{t}{2a_\ast}}$ and $c_-t^{\frac{m_-}{2}-1}e^{-\frac{t}{2a_\ast}}$ for some constants $c_+$ and $c_-$. Thus the asymptotics of the sum are dominated by term corresponding to $\max(m_+,m_-)$. When $m_+ = m_-$, the terms have equal size and the constants add.
\end{proof}

\begin{lem}
\label{lem:rankk_allneg}
Suppose $a_1, \ldots, a_k < 0$ are fixed, write $A_N = \diag(a_1, \ldots, a_k, 0,\ldots,0)$, and write $\Psi_{k/2}$ for a $\frac{k}{2}$-Weilbull-distributed random variable (i.e. with distribution function $\Psi_{k/2}(x) = \min(\exp(-(-x)^{k/2}),1)$). Then with 
\[
	\gamma_k = \frac{1}{2} \left(\frac{2}{k\Gamma(k/2)}\right)^{\frac{2}{k}}
\]
we have
\[
	\frac{\gamma_k N^{2/k}}{\left(\prod_{j=1}^k \abs{a_j}\right)^{\frac{1}{k}}} \max_{i=1}^N \ip{\mathbf{y}_i, A_N \mathbf{y}_i}  \overset{N \to \infty}{\to} \Psi_{k/2} \quad \text{in distribution.}
\] 
\end{lem}
\begin{proof}
In the case $a_1 = \cdots = a_k = a$, we note that whenever $x < 0$ we have
\[
	\P\left(\frac{\gamma_k N^{2/k}}{(-a)} \max_{j=1}^N \left\{a \sum_{i=1}^k y_{ij}^2\right\} \leq x\right) = \P\left( \sum_{i=1}^k y_{i1}^2 \geq -\frac{x}{\gamma_kN^{2/k}} \right)^N \to \exp(-(-x)^{k/2}).
\]
When the $a_i$'s are not all the same, we will use the classic result (see, e.g., \cite[Proposition 3.3.14]{EmbKluMik1997}) that the domain of attraction of $\Psi_{k/2}$ is closed under tail-equivalence; precisely, if $F$ and $G$ are distributions with right endpoints at zero and there exists a positive sequence $(c_N)_{N=1}^\infty$ with
\[
	\lim_{N \to \infty} F^N(c_Nx) = \Psi_{k/2}(x), \quad x < 0,
\]
then
\begin{equation}
\label{eqn:weibullG_k}
	\lim_{N \to \infty} G^N\left(b^{2/k}c_Nx\right) = \Psi_{k/2}(x), \quad x < 0,
\end{equation}
if and only if 
\begin{equation}
\label{eqn:weibull_ratio_k}
	\lim_{x \uparrow 0} \frac{1-F(x)}{1-G(x)} = b.
\end{equation}
We use this with $F(x) = \exp\left(-\left(- x \right)^{\frac{k}{2}}\right)$ and $c_N = \frac{1}{N^{2/k}}$, which has $\lim_{x \uparrow 0} \frac{1-F(x)}{(-x)^{k/2}} = 1$, and with $G$ the distribution function of $\sum_{i=1}^k a_i Z_i^2$, where the $Z_i$'s are independent Gaussians. That is, for $x < 0$ we have
\[
	1-G(x) = \frac{2(-x)^{\frac{k}{2}}}{(2\pi)^{k/2}} \int_{[0,\pi]^{k-1}} \left( \int_0^1 e^{x\frac{u^2}{2(-f(\vec{\varphi}))}} u^{k-1} \diff u \right) (-f(\vec{\varphi}))^{-\frac{k}{2}} \sin^{k-2}(\varphi_1) \sin^{k-3}(\varphi_2) \cdots \sin(\varphi_{k-2}) \diff \vec{\varphi}
\]
with the notation $f(\vec{\varphi})$ from before, and then dominated convergence gives \eqref{eqn:weibull_ratio_k} with
\begin{align*}
	b &= \left( \frac{2}{(2\pi)^{k/2}} \int_{[0,\pi]^{k-1}} \left( \int_0^1 u^{k-1} \diff u \right) (-f(\vec{\varphi}))^{-\frac{k}{2}} \sin^{k-2}(\varphi_1) \sin^{k-3}(\varphi_2) \cdots \sin(\varphi_{k-2}) \diff \vec{\varphi} \right)^{-1} \\
	&= \left( \frac{2}{k(2\pi)^{k/2}} \int_{[0,\pi]^{k-1}} (-f(\vec{\varphi}))^{-\frac{k}{2}} \sin^{k-2}(\varphi_1) \sin^{k-3}(\varphi_2) \cdots \sin(\varphi_{k-2}) \diff \vec{\varphi} \right)^{-1} = \gamma_k^{-k/2} \sqrt{\prod_{j=1}^k \abs{a_j}},
\end{align*}
where the last equality is Lemma \ref{lem:calculusmiracle}. Then \eqref{eqn:weibullG_k} gives the result.
\end{proof}

\begin{lem}
\label{lem:calculusmiracle}
Fix $k \in \N$ and $a_1, \ldots, a_k > 0$. If we let
\begin{align*}
	f(\vec{\varphi}) =& \, a_1\cos^2(\varphi_1) + a_2\sin^2(\varphi_1)\cos^2(\varphi_2) + a_3\sin^2(\varphi_1)\sin^2(\varphi_2)\cos^2(\varphi_3) \\
	&+ \cdots + a_{k-1}\sin^2(\varphi_1)\cdots \sin^2(\varphi_{k-2})\cos^2(\varphi_{k-1}) + a_k\sin^2(\varphi_1) \cdots \sin^2(\varphi_{k-2})\sin^2(\varphi_{k-1}),
\end{align*}
then
\[
	 \int_{[0,\pi]^{k-1}} f(\vec{\varphi})^{-\frac{k}{2}} \sin^{k-2}(\varphi_1) \sin^{k-3}(\varphi_2) \cdots \sin(\varphi_{k-2}) \diff \vec{\varphi} = \frac{\pi^{k/2}}{\Gamma(k/2)} \frac{1}{\sqrt{\prod_{j=1}^k a_j}}.
\]
\end{lem}
\begin{proof}
Consider the axis-aligned ellipsoid $E = \{(z_1, \ldots, z_k) \in \R^k : a_1z_1^2 + \cdots + a_kz_k^2 \leq 1\}$. On the one hand, we know that its volume is $\frac{1}{\sqrt{\prod_{j=1}^k a_j}}$ times the volume of the $n$-dimensional unit sphere, i.e., 
\[
	\Vol(E) = \frac{\pi^{k/2}}{\Gamma(k/2+1)} \frac{1}{\sqrt{\prod_{j=1}^k a_j}} = \frac{2}{k} \frac{\pi^{k/2}}{\Gamma(k/2)} \frac{1}{\sqrt{\prod_{j=1}^k a_j}}.
\]
On the other hand, we can find the volume of $E$ in hyperspherical coordinates $(r,\varphi_1, \ldots, \varphi_{n-1})$; it is described by $\{(r,\vec{\varphi}) : 0 \leq r \leq r_{\textup{max}}(\vec{\varphi})\}$ for some function $r_{\textup{max}}(\vec{\varphi})$ which we now compute. Since one changes variables back as $z_1 = r\cos(\varphi_1), z_2 = r\sin(\varphi_1)\cos(\varphi_2), \ldots, z_k = r\sin(\varphi_1) \cdots \sin(\varphi_k)$, we have
\[
	1 = a_1r_{\textup{max}}(\vec{\varphi})^2\cos^2(\varphi_1) + a_2r_{\textup{max}}(\vec{\varphi})^2 \sin^2(\varphi_1)\cos^2(\varphi_2) + \cdots + a_kr_{\textup{max}}(\vec{\varphi})^2 \sin^2(\varphi_1) \cdots \sin^2(\varphi_k) = r_{\textup{max}}(\vec{\varphi})^2 f(\vec{\varphi}),
\]
so that 
\begin{align*}
	\Vol(E) &= \int_{S_\varphi} \int_0^{r_{\textup{max}}(\vec{\varphi})} r^{k-1} \sin^{k-2}(\varphi_1) \sin^{k-3}(\varphi_2) \cdots \sin(\varphi_{k-2}) \diff r \diff \vec{\varphi} \\
	&= \frac{2}{k} \int_{[0,\pi]^{k-1}} f(\vec{\varphi})^{-\frac{k}{2}} \sin^{k-2}(\varphi_1) \sin^{k-3}(\varphi_2) \cdots \sin(\varphi_{k-2}) \diff \vec{\varphi}
\end{align*}
and now we match terms.
\end{proof}

\begin{proof}[Proof of Theorem \ref{thm:main_fixed}]
This follows immediately from Theorem \ref{thm:main} for the Gumbel limits and Lemma \ref{lem:weibull} for the Weibull limits, along with corresponding Gaussian computations, namely Lemmas \ref{lem:rankk_allpos} and \ref{lem:rankk_somepossomeneg} (which give \eqref{eqn:main_apos_noabs} and \eqref{eqn:main_apos_abs}), and Lemma \ref{lem:rankk_allneg} (which gives \eqref{eqn:main_aneg}).
\end{proof}


\section{Gaussian Computations: Diverging Rank}
\label{sec:gaussian_diverging}

The goal of this section is to make Gaussian computations in a representative case of diverging rank.

\begin{lem}
\label{lem:gauss_diverging}
Fix $0 < \alpha < 1$, and let $(k_N)_{N=1}^\infty$ be any integer sequence such that, for some $\epsilon > 0$, 
\begin{equation}
\label{eqn:close_to_nalpha}
	\abs{k_N - N^\alpha} \leq N^{\frac{\alpha}{2}-\epsilon}.
\end{equation}
If we define
\[
	A_N = \diag(1,\ldots,1,0,\ldots,0), \quad \rank(A_N) = k_N,
\]
then (recalling that $(\mathbf{y}_i)_{i=1}^N$ are i.i.d. standard Gaussian vectors)
\begin{equation}
\label{eqn:gauss_diverging}
	\left(\frac{\sqrt{\log N}}{N^{\alpha/2}}\right) \max_{i=1}^N \ip{\mathbf{y}_i, A_N \mathbf{y}_i} - N^{\alpha/2}\sqrt{\log N} - 2\log N + \frac{\log \log N}{2} + \frac{\log (4\pi)}{2} \overset{N \to \infty}{\to} \Lambda \quad \text{in distribution.}
\end{equation}
\end{lem}
\begin{proof}
We will need the lower and upper incomplete gamma functions, respectively $\gamma(\cdot,\cdot)$ and $\Gamma(\cdot,\cdot)$. They can be defined for complex arguments, but for us it suffices to consider nonnegative real arguments and the definitions
\[
	\gamma(s,x) = \int_0^x t^{s-1}e^{-t}\diff t, \qquad \Gamma(s,x) = \int_x^\infty t^{s-1}e^{-t}\diff t.
\]
With the positive sequences $(p_N)_{N=1}^\infty, (q_N)_{N=1}^\infty$ defined by
\begin{align}
	p_N &= p_{N,\alpha} \defeq \frac{N^{\alpha/2}}{\sqrt{\log N}}, \label{eqn:cn} \\
	q_N &= q_{N,\alpha} \defeq N^\alpha + 2N^{\alpha/2}\left(\sqrt{\log N} - \frac{\log(4\pi\log N)}{4\sqrt{\log N}} \right), \label{eqn:dn}
\end{align}
notice that \eqref{eqn:gauss_diverging} is equivalent to
\[
	\frac{\max_{i=1}^N \ip{\mathbf{y}_i,A_N\mathbf{y}_i} - q_N}{p_N} \overset{N \to \infty}{\to} \Lambda \quad \text{in distribution},
\]
i.e., to showing that for every real $x$ we have
\begin{align*}
	\exp(-e^{-x}) &= \lim_{N \to \infty} \P\left(\max_{i=1}^N \ip{\mathbf{y}_i,A_N\mathbf{y}_i} \leq p_Nx + q_N\right) = \lim_{N \to \infty} \P(\ip{\mathbf{y}_1,A_N\mathbf{y}_1} \leq p_Nx+q_N)^N \\
	&= \lim_{N \to \infty} \left( 1 - \frac{\Gamma\left(\frac{k_N}{2}, \frac{p_Nx+q_N}{2}\right)}{\Gamma\left(\frac{k_N}{2}\right)} \right)^N,
\end{align*}
where the last equality is an elementary calculation. Writing
\begin{align*}
	n = n_N &\defeq \frac{k_N}{2}, \\
	a_N(x) = a_{N,\alpha}(x) &\defeq \frac{\Gamma\left(n, \frac{p_Nx+q_N}{2}\right)}{\Gamma\left(n\right)}
\end{align*}
and Taylor-expanding $\log(1-x) \approx -x$, we find that it suffices to show
\begin{equation}
\label{eqn:an_limit}
	\lim_{N \to \infty} Na_N(x) = e^{-x}
\end{equation}
for arbitrary $x$, an analysis exercise that takes up the remainder of the proof.

We will use the following estimate, due to Wimp, published in \cite[Proposition 3]{BoyGoh2007}: Uniformly for $t \geq 1$, with 
\begin{align*}
	\mf{e}_n(x) &\defeq \sum_{m=0}^n \frac{x^n}{n!}, \\
	\mu(t) &\defeq \sqrt{t-\log(t)-1},
\end{align*}
we have
\begin{equation}
\label{eqn:wimp}
	\frac{\mf{e}_n(nt)}{e^{nt}} = \frac{1}{\sqrt{2}} \frac{\mu(t)t}{t-1} \erfc(\sqrt{n}\mu(t))\left(1+\OO_{n \to +\infty}\left(\frac{1}{\sqrt{n}}\right)\right).
\end{equation}
(Notice that $\lim_{t \downarrow 1} \frac{\mu(t)}{t-1} = \frac{1}{\sqrt{2}}$, so there is no singularity at $t = 1$.) 

On the other hand, one can show that
\[
	\frac{t\mf{e}_{n-1}(nt)}{\mf{e}_n(nt)} = 1+\OO_{n \to +\infty} \left(\frac{1}{\sqrt{n}}\right), \quad \text{uniformly in $t \geq 1$}.
\]
Indeed, the left-hand side is, for each $n$, a non-decreasing function of $t \geq 1$, bounded above by one, whose value at $t = 1$ is $1+\OO(1/\sqrt{n})$. Since $\frac{\Gamma(n,nt)}{\Gamma(n)} = e^{-nt}\mf{e}_{n-1}(nt)$, these two estimates give
\begin{equation}
\label{eqn:wimp_for_us}
	\frac{\Gamma(n,nt)}{\Gamma(n)} = \frac{1}{\sqrt{2}} \frac{\mu(t)}{t-1} \erfc(\sqrt{n}\mu(t))\left(1+\OO_{n \to +\infty}\left(\frac{1}{\sqrt{n}}\right)\right), \quad \text{uniformly in $t \geq 1$}.
\end{equation}
We use this with
\begin{align*}
	t &= t_n(x) = t_N(x) \defeq \frac{p_Nx+q_N}{k_N} = \left(\frac{N^\alpha}{k_N}\right) \left(1+2N^{-\alpha/2}\left(\sqrt{\log N} - \frac{\log(4\pi e^{-2x} \log N)}{4\sqrt{\log N}}\right)\right).
\end{align*}
As above, we find $\lim_{n \to \infty} \frac{\mu(t_n(x))}{t_n(x)-1} = \frac{1}{\sqrt{2}}$. 
Thus to show \eqref{eqn:an_limit} it suffices to show
\begin{align}
	\lim_{N \to \infty} \frac{N\erfc(\sqrt{n}\left(\frac{t_n(x)-1}{\sqrt{2}}\right))}{2} &= e^{-x}, \label{eqn:an_limit_main} \\
	\lim_{N \to \infty} \frac{\erfc(\sqrt{n}\mu(t_n(x)))}{\erfc(\sqrt{n}\left(\frac{t_n(x)-1}{\sqrt{2}}\right))} &= 1. \label{eqn:an_limit_taylor}
\end{align}
To show these, we first observe the following elementary-calculus consequence of \eqref{eqn:close_to_nalpha}. For some $c, C > 0$, and for $N$ large enough depending on $x$, we have
\begin{equation}
\label{eqn:tn-1}
	cN^{-\alpha/2}\sqrt{\log N} \leq t_n(x)-1 \leq CN^{-\alpha/2}\sqrt{\log N}.
\end{equation}
This implies both $\lim_{N \to \infty} \sqrt{n}(t_n(x)-1) = +\infty$, thus by Taylor expansion $\lim_{N \to \infty} \sqrt{n}\mu(t_n(x))=+\infty$, as well as $\lim_{N \to \infty} n(t_n(x)-1)^3 = 0$. Applying these with the asymptotics $\lim_{x \to +\infty} \erfc(x) \exp(x^2)x\sqrt{\pi} = 1$ and the Taylor expansion 
\[
	\lim_{s \downarrow 1} \frac{\frac{(s-1)^2}{2}-\mu(s)^2}{\frac{(s-1)^3}{3}} = 1,
\]
we can check \eqref{eqn:an_limit_taylor} as
\[
	\lim_{N \to \infty} \frac{\erfc(\sqrt{n}\mu(t_n(x)))}{\erfc(\sqrt{n}\left(\frac{t_n(x)-1}{\sqrt{2}}\right))} = \lim_{N \to \infty} \exp\left[ n\left(\frac{(t_n(x)-1)^2}{2}-\mu(t_n(x))^2\right) \right] \frac{\frac{t_n(x)-1}{\sqrt{2}}}{\mu(t_n(x))} = 1.
\]
To check \eqref{eqn:an_limit_main}, we will need the quantities (recall that $k_N = 2n$)
\begin{align*}
	r_n(x) &\defeq \sqrt{\frac{\frac{N^\alpha}{2}}{2}} \left(2N^{-\alpha/2}\left(\sqrt{\log N}-\frac{\log(4\pi e^{-2x}\log N)}{4\sqrt{\log N}}\right) \right) = \sqrt{\log N} - \frac{\log(4\pi e^{-2x} \log N)}{4\sqrt{\log N}}, \\
	s_n(x) &\defeq \sqrt{\frac{n}{2}} (t_n(x)-1) = \sqrt{\frac{n}{2}} \left(\frac{N^\alpha}{k_N}(1+2N^{-\alpha/2}r_n(x)) - 1\right) = r_n(x) + \OO(N^{-\epsilon}),
\end{align*}
where the last estimate follows from \eqref{eqn:close_to_nalpha}. These therefore satisfy
\[
	\lim_{N \to \infty} (r_n(x)^2-s_n(x)^2) = 0,
\]
so we can verify \eqref{eqn:an_limit_main} as
\[
	\lim_{N \to \infty} \frac{N}{2} \erfc\left[ \sqrt{n}\left(\frac{t_n(x)-1}{\sqrt{2}}\right)\right] = \lim_{N \to \infty} \frac{N}{2\sqrt{\pi}} \frac{\exp(- s_n(x)^2)}{s_n(x)} = \lim_{N \to \infty} \frac{N}{2\sqrt{\pi}} \frac{\exp(-r_n(x)^2)}{r_n(x)} = e^{-x},
\]
where the last computation is elementary. This finishes the proof of \eqref{eqn:an_limit} and thus Lemma \ref{lem:gauss_diverging}.
\end{proof}

\begin{proof}[Proof of Theorem \ref{thm:main_diverging}]
This follows immediately from Theorem \ref{thm:main} and Lemma \ref{lem:gauss_diverging}, which give \eqref{eqn:main_diverging}.
\end{proof}

\addcontentsline{toc}{section}{References}
\bibliographystyle{alpha-abbrvsort}
\bibliography{gumbelbib}

\end{document}